\documentclass[a4paper]{amsart}
\pdfoutput=1
\usepackage{amsmath,amssymb}

\usepackage{subfiles}

\DeclareSymbolFont{bbold}{U}{bbold}{m}{n}
\DeclareSymbolFontAlphabet{\mathbbold}{bbold}

\DeclareSymbolFontAlphabet{\amsbb}{AMSb}
\usepackage{mbboard}
\renewcommand{\mathbb}[1]{\amsbb{#1}}

\usepackage[shortlabels]{enumitem}
\setitemize[1]{leftmargin=2em}
\setenumerate[1]{leftmargin=2em}
\usepackage{tikz-cd}
\usetikzlibrary{decorations.pathmorphing}
\usepackage[unicode,colorlinks=true,linktocpage=true,citecolor=blue]{hyperref}
\usepackage[capitalise]{cleveref}
\crefformat{equation}{(#2#1#3)}

\usepackage{graphicx,color}

\newtheorem{theorem}{Theorem}
\numberwithin{theorem}{section}
\newtheorem{thm}[theorem]{Theorem}
\newtheorem{proposition}[theorem]{Proposition}
\newtheorem{propn}[theorem]{Proposition}

\newtheorem{corollary}[theorem]{Corollary}
\newtheorem{cor}[theorem]{Corollary}
\newtheorem{lemma}[theorem]{Lemma}
\theoremstyle{definition}
\newtheorem{definition}[theorem]{Definition}
\newtheorem{defn}[theorem]{Definition}
\newtheorem{example}[theorem]{Example}

\newtheorem{notation}[theorem]{Notation}
\newtheorem{remark}[theorem]{Remark}

\providecommand{\op}{\mathrm{op}}
\providecommand{\xel}{\mathrm{el}}

\providecommand{\xint}{\mathrm{int}}

\newcommand{\act}{\name{act}}
\newcommand{\Act}{\name{Act}}
\newcommand{\el}{\name{el}}

\DeclareMathOperator{\colimP}{colim}
\newcommand{\colim}{\mathop{\colimP}}

\newcommand{\PSh}{\operatorname{P}}
\newcommand{\LPSh}{\widehat{\PSh}}

\newcommand{\Map}{\operatorname{Map}}

\newcommand{\xAlg}{\operatorname{Alg}}

\newcommand{\xMon}{\operatorname{Mon}}

\newcommand{\sSet}{\Set_{\Delta}}

\newcommand{\xxO}{\mathcal{O}}

\newcommand{\xcc}{\mathcal{C}}

\newcommand{\xS}{\mathcal{S}}

\newcommand{\xV}{\mathcal{V}}

\newcommand{\id}{\operatorname{id}}

\newcommand{\cocart}{\name{cocart}}

\newcommand{\Ass}{\name{Ass}}
\newcommand{\Bimod}{\name{Bimod}}
\newcommand{\CMod}{\name{CMod}}

\newcommand{\xF}{\mathbb{F}}

\newcommand{\icat}{$\infty$-category}

\newcommand{\igpd}{$\infty$-groupoid}
\newcommand{\igpds}{$\infty$-groupoids}
\newcommand{\icats}{$\infty$-categories}
\newcommand{\iopd}{$\infty$-operad}
\newcommand{\iopds}{$\infty$-operads}
\newcommand{\isoto}{\xrightarrow{\sim}}

\newcommand{\xto}[1]{\xrightarrow{#1}}

\makeatletter
\def\@tocline#1#2#3#4#5#6#7{\relax
  \ifnum #1>\c@tocdepth 
  \else
  \par \addpenalty\@secpenalty\addvspace{#2}%
  \begingroup \hyphenpenalty\@M
  \@ifempty{#4}{%
    \@tempdima\csname r@tocindent\number#1\endcsname\relax
  }{%
  \@tempdima#4\relax
}%
\parindent\z@ \leftskip#3\relax \advance\leftskip\@tempdima\relax
\rightskip\@pnumwidth plus4em \parfillskip-\@pnumwidth
#5\leavevmode\hskip-\@tempdima
\ifcase #1
\or \hskip -1em \or \hskip 1em \or \hskip 3em \else \hskip 5em \fi%
#6\nobreak\relax
\hfill\hbox to\@pnumwidth{\@tocpagenum{#7}}
\par
\nobreak
\endgroup
\fi}
\makeatother

\newcommand{\name}[1]{\ensuremath{\text{\textup{#1}}}}

\newcommand{\simp}{\bbDelta}
\newcommand{\Dop}{\simp^{\op}}

\newcommand{\bbO}{\bbOmega}

\newcommand{\Set}{\name{Set}}

\newcommand{\Mon}{\name{Mon}}

\newcommand{\Fun}{\name{Fun}}

\newcommand{\blank}{\text{\textendash}}

\newcommand{\Cat}{\name{Cat}}
\newcommand{\CatI}{\Cat_{\infty}}
\newcommand{\LCatI}{\widehat{\Cat}_{\infty}}
\newcommand{\IFF}{if and only if}

\newcommand{\Alg}{\name{Alg}}

\newcommand{\eg}{e.g.\@}
\newcommand{\ie}{i.e.\@}

\newcommand{\angled}[1]{\langle #1 \rangle}
\usepackage{eucal}

\usepackage[utf8]{inputenc}

\newcommand{\RFib}{\name{RFib}}
\newcommand{\LRFib}{\widehat{\RFib}}

\newcommand{\RSl}{\name{RSl}}
\usepackage{extarrows}
\newcommand{\ev}{\name{ev}}
\usepackage{amssymb}
\newcommand{\actto}{\rightsquigarrow}
\newcommand{\intto}{\rightarrowtail}

\tikzcdset{
   active/.style={rightsquigarrow},
   inert/.style={tail}
 }

\usepackage{mathtools}

\makeatletter
\MHInternalSyntaxOn
\def\MT_leftarrow_fill:{%
  \arrowfill@\leftarrow\relbar\relbar}
\def\MT_rightarrow_fill:{%
  \arrowfill@\relbar\relbar\rightarrow}
\newcommand{\xrightleftarrows}[2][]{\mathrel{%
  \raise.55ex\hbox{%
    $\ext@arrow 0359\MT_rightarrow_fill:{\phantom{#1}}{#2}$}%
  \setbox0=\hbox{%
    $\ext@arrow 3095\MT_leftarrow_fill:{#1}{\phantom{#2}}$}%
  \kern-\wd0 \lower.55ex\box0}}
\MHInternalSyntaxOff
\makeatother

\title{Free Algebras Through Day Convolution}
\author{Hongyi Chu}
\address{Max Planck Institute for Mathematics, Bonn, Germany}

\author{Rune Haugseng}
\address{Norwegian University of Science and Technology (NTNU), Trondheim, Norway}
\date{\today}
\begin{document}
\begin{abstract}
  Building on the foundations in our previous paper, we study Segal
  conditions that are given by finite products, determined by
  structures we call cartesian patterns. We set up Day convolution on
  presheaves in this setting and use it to give conditions under which
  there is a colimit formula for free algebras and other left
  adjoints. This specializes to give a simple proof of Lurie's results
  on operadic left Kan extensions and free algebras for symmetric
  \iopds{}.
\end{abstract}

\maketitle
\tableofcontents

\section{Introduction}
A key feature of symmetric \iopds{}, as defined in Lurie's book
\cite{ha}, is that there is an explicit formula for their free
algebras. More generally, for any morphism
$f \colon \mathcal{O} \to \mathcal{P}$ of \iopds{} there is a formula
for the corresponding left operadic Kan extension, \ie{} the left
adjoint to the functor
$f^{*} \colon \Alg_{\mathcal{P}}(\mathcal{C})\to
\Alg_{\mathcal{O}}(\mathcal{C})$ between \icats{} of algebras given by
composition with $f$. However, the construction of these left adjoints
in \cite{ha} is by a cumbersome simplex-by-simplex induction using a
delicate analysis of the inert--active factorization system on finite
pointed sets. Part of our goal in this paper is to give a new, simpler
construction of these left adjoints in the following three steps:
\begin{enumerate}[(1)]
\item We first consider algebras in the \icat{} of spaces (with the cartesian
  monoidal structure). Here it is easy to see that the left adjoint is
  just given by an ordinary left
  Kan extension.
\item Next we consider algebras in presheaves on small symmetric
  monoidal \icats{}. Here the universal property of the Day
  convolution structure allows us to reduce to the previous case.
\item Finally, we use that any presentably symmetric monoidal \icat{} is a
  symmetric monoidal localization of a presheaf \icat{} with Day
  convolution. Since the localization functor is symmetric monoidal
  and colimit-preserving, we can use it to transport the colimit
  formula for the left adjoint from the previous step.
\end{enumerate}
Symmetric \iopds{} in Lurie's sense are certain \icats{} over the
category $\xF_{*}$ of pointed finite sets; these are, in a sense, the
universal objects that have algebras in symmetric monoidal
\icats{}. In practice, however, it can be useful to describe algebraic
structures by more general \icats{} over $\xF_{*}$.  For example, we
can sometimes find a combinatorially simpler description by using an
\icat{} that is not an \iopd{}; as a somewhat trivial example,
associative algebras can be described in terms of the simplex category
$\Dop$, which is less complicated than the symmetric associative
operad. In other cases, even though it is formally known that a
certain structure is described by a symmetric \iopd{}, this object may
be difficult to describe explicitly; for instance, the structure of $n$
compatible associative algebra structures can trivially be described
in terms of the product $\simp^{n,\op}$, while describing the
associated \iopd{} $\mathbb{E}_{n}$ amounts to proving the Dunn--Lurie
additivity theorem.

It is therefore desirable to understand when free algebras and other
left adjoints can be described by an explicit formula without passing
to the associated symmetric \iopds{}. Our main goal in this paper is
to obtain a simple criterion for this by following the same three
steps we outlined above.

We start in \S\ref{sec:cartpatt} by making precise the class of
\icats{} over $\xF_{*}$ we are interested in, which we call
\emph{cartesian patterns}. This builds on the foundations in our
previous paper \cite{patterns}, where we studied Segal
conditions in general; here we are interested in those Segal-type
limit conditions that are given by finite products. We then give some
examples of cartesian patterns in \S\ref{sec:pattex} and introduce
monoidal \icats{} over a cartesian pattern and algebras therein in
\S\ref{sec:Omonoidal}. 

The bulk of the paper is then taken up by extending to the setting of
general cartesian patterns the results we need to carry out our proof
strategy:
\begin{itemize}
\item In \S\ref{sec:algmonoid} we show that if $\mathcal{O}$ is a
  cartesian pattern and $\mathcal{C}$ is an \icat{} with finite
  products, then $\mathcal{O}$-monoids in $\mathcal{C}$ are equivalent
  to $\mathcal{O}$-algebras in an $\mathcal{O}$-monoidal structure on
  $\mathcal{C}$ given by cartesian products.
\item In \S\ref{sec:dayconv} we introduce Day convolution on
  presheaves for monoidal \icats{} over a cartesian pattern.
\item In \S\ref{sec:monloc} we study monoidal localizations over a
  cartesian pattern and prove that any presentably monoidal \icat{} is
  a monoidal localization of a Day convolution structure.
\end{itemize}
These foundations allow us to prove our main result in
\S\ref{sec:freealg}:
\begin{thm}\ 
  \begin{enumerate}[(i)]
  \item Suppose $f \colon \mathcal{O} \to \mathcal{P}$ is a morphism
    of cartesian patterns that is \emph{extendable} in the sense of
    \cref{def ext} and $\mathcal{V}$ is a presentably
    $\mathcal{P}$-monoidal \icat{}. Then the functor
    \[ f^{*} \colon \Alg_{\mathcal{P}}(\mathcal{V}) \to
      \Alg_{\mathcal{O}/\mathcal{P}}(\mathcal{V})\]
    between \icats{} of algebras given by restriction along $f$
    has a left adjoint $f_{!}$, which for $P \in \mathcal{P}^{\el}$ satisfies
    \[ (f_{!}A)(P) \simeq \colim_{(O,\,\phi \colon \! f(O) \actto P) \in
        \mathcal{O}^{\act}_{/P}} \phi_{!}A(O).\]
  \item Suppose $\mathcal{O}$ is a cartesian pattern that is
    \emph{extendable} in the sense of \cref{def ext patt} and
    $\mathcal{V}$ is a presentably $\mathcal{O}$-monoidal
    \icat{}. Then the functor
    \[ U_{\mathcal{O}} \colon \Alg_{\mathcal{O}}(\mathcal{V}) \to
      \Fun_{/\mathcal{O}^{\el}}(\mathcal{O}^{\el}, \mathcal{V}) \]
    given by restriction to the subcategory $\mathcal{O}^{\el}$ of
    elementary objects has a left adjoint $F_{\mathcal{O}}$, which for
    $E \in \mathcal{O}^{\el}$ satisfies
    \[ U_{\mathcal{O}}F_{\mathcal{O}}\Phi(E) \simeq \colim_{\phi
        \colon \! O \actto E\, \in\, \Act_{\mathcal{O}}(E)}
      \phi_{!}(\Phi(O_{1}),\ldots,\Phi(O_{n})). \] Moreover, the
    adjunction $F_{\mathcal{O}} \dashv U_{\mathcal{O}}$ is monadic.
  \end{enumerate}
\end{thm}
This theorem is a combination of Corollaries~\ref{cor:f!alg} and
\ref{cor:FOmonadic}; we refer the reader to \S\S\ref{sec:cartpatt} and
\ref{sec:Omonoidal} for the notation used. We discuss some examples
of extendable cartesian patterns and morphisms to which the theorem
applies in \S\ref{sec:extex}. The explicit formula for free algebras
leads to a simple criterion for a morphism of cartesian patterns to
induce equivalences on \icats{} of algebras, which we consider in
\S\ref{sec:Morita} together with some easy applications.

\subsection*{Acknowledgments}
The first author thanks the Labex CEMPI (ANR-11-LABX-0007-01) and Max
Planck Institute for Mathematics for their hospitality and financial
support during the process of writing this article.

Some work on this paper was carried out while the second author was
employed by the IBS Center for Geometry and Physics in Pohang, in a
position funded by grant IBS-R003-D1 of the Institute for Basic
Science of the Republic of Korea; it was completed while he
was in residence at the Matematical Sciences Research Institute in
Berkeley, California, during the Spring 2020 semester, and is thereby
based upon work supported by the National Science Foundation under
grant DMS-1440140.

We also thank Hadrian Heine for sharing his work on Day convolution
\cite[\S 6.1]{HeineThesis} with us.

\section{Cartesian Patterns and Monoids}\label{sec:cartpatt}
In this section we review some basic definitions from \cite{patterns}
and introduce the algebraic structures we will study in
this paper, namely cartesian patterns and algebras and monoids over them.

\begin{defn}
  An \emph{algebraic pattern} consists of an \icat{} $\mathcal{O}$
  equipped with a factorization system, whereby each map factors essentially uniquely as an
  \emph{inert} map followed by an \emph{active} map, together with a
  collection of \emph{elementary} objects. We write
  $\mathcal{O}^{\act}$ and $\mathcal{O}^{\xint}$ for the subcategories
  of $\mathcal{O}$ containing only the active and inert morphisms,
  respectively, and $\mathcal{O}^{\el} \subseteq \mathcal{O}^{\xint}$
  for the full subcategory of elementary objects and inert morphisms
  among them. A morphism of algebraic patterns from $\mathcal{O}$ to
  $\mathcal{P}$ is a functor $f \colon \mathcal{O} \to \mathcal{P}$
  that preserves inert and active morphisms and elementary objects.
\end{defn}

\begin{notation}
  If $\mathcal{O}$ is an algebraic pattern, we will indicate an inert
  map between objects $O,O'$ of $\mathcal{O}$ as $O \intto O'$ and an
  active map as $O \actto O'$. These symbols are not meant to suggest
  any intuition about the nature of inert and active maps.
\end{notation}

\begin{example}
  Let $\xF_{*}$ denote a skeleton of the category of finite pointed
  sets with objects $\angled{n} := (\{0,1,\ldots,n\},0)$. This has an inert--active
  factorization system where a morphism $\phi \colon \angled{n} \to \angled{m}$ is 
  \begin{itemize}
  \item \emph{inert} if it is an isomorphism away from the base point,
    \ie{} $|\phi^{-1}(i)| = 1$ if $i \neq 0$,
  \item \emph{active} if it doesn't send anything except the base
    point to the base point, \ie{} $\phi^{-1}(0) = \{0\}$.
  \end{itemize}
  We write $\xF_{*}^{\flat}$ for the algebraic pattern given by this
  factorization system with $\angled{1}$ as the only elementary
  object.
\end{example}

\begin{notation}
  If $\mathcal{O}$ is an algebraic pattern and $O$ is an object of
  $\mathcal{O}$, we write
  \[\mathcal{O}^{\el}_{O/} :=
    \mathcal{O}^{\el}\times_{\mathcal{O}^{\xint}}
    \mathcal{O}^{\xint}_{O/}\]
  for the \icat{} of inert maps from $O$ to elementary objects, and
  inert maps between them.
\end{notation}

\begin{notation}
  We write $\rho_{i} \colon \angled{n} \to \angled{1}$, $i = 1,\ldots,n$, for the inert
  map given by
  \[ \rho_{i}(j) =
    \begin{cases}
      0, & j \neq i, \\
      1, & j = i.
    \end{cases}
  \]
  Then $(\xF_{*}^{\flat})^{\el}_{\angled{n}/}$ is equivalent to the
  discrete set $\{\rho_{1},\ldots,\rho_{n}\}$.
\end{notation}

\begin{defn}
  A \emph{cartesian pattern} is an algebraic pattern $\mathcal{O}$
  equipped with a morphism of algebraic patterns
  $|\blank| \colon \mathcal{O} \to \xF_{*}^{\flat}$ such that
  for every object $O \in \mathcal{O}$ the induced map
  \[ \mathcal{O}^{\el}_{O/} \to \xF_{*,|O|/}^{\el} \] is an
  equivalence. A \emph{morphism of cartesian patterns} is a
  morphism of algebraic patterns over $\xF_{*}^{\flat}$.
\end{defn}

\begin{notation}
  If $\mathcal{O}$ is a cartesian pattern and $O$ is an object of
  $\mathcal{O}$ such that $|O| \cong \angled{n}$, then the \icat{}
  $\mathcal{O}^{\el}_{O/}$ is equivalent to a discrete set consisting
  of $n$ inert morphisms with source $O$, with an essentially unique
  such morphism lying over
  $\rho_{i} \colon \angled{n} \to \angled{1}$ for $i = 1,\ldots,n$.
  We denote this inert morphism by $\rho_{i}^{O}\colon O \to O_{i}$.
\end{notation}

\begin{lemma}\label{rem:Oel}
  If $\mathcal{O}$ is a cartesian pattern, then the \icat{}
  $\mathcal{O}^{\el}$ is an \igpd{}.
\end{lemma}
\begin{proof}
  If $O \in \mathcal{O}$ lies over $\angled{1}$ in $\xF_{*}$, then it
  follows from the definition that the \icat{}
  $\mathcal{O}^{\el}_{O/}$ is a contractible \igpd{}. This holds in
  particular for $E \in \mathcal{O}^{\el}$ (since $E$ must map to the
  unique elementary object $\angled{1}$ in $\xF_{*}^{\flat}$), so that the
  fibres of the cartesian fibration
  \[ \ev_{0} \colon \Fun(\Delta^{1}, \mathcal{O}^{\el}) \to \mathcal{O}^{\el} \]
  are contractible \igpds{}. This functor is therefore an equivalence,
  which means in particular that $\mathcal{O}^{\el}$ is local with respect to
  $\Delta^{0} \to \Delta^{1}$; it is thus an \igpd{}.
\end{proof}

\begin{defn}
  If $\mathcal{O}$ is a cartesian pattern and $\mathcal{C}$ is an
  \icat{} with finite products, then a functor $F \colon
  \mathcal{O} \to \mathcal{C}$  is an \emph{$\mathcal{O}$-monoid} if
  for $O \in \mathcal{O}$ lying over $\angled{n}$ the natural map
  \[ F(O) \to \prod_{i = 1}^{n} F(O_{i}),\]
  induced by the maps $\rho_{i}^{O} \colon O \to O_{i}$,
  is an equivalence. We write
  $\Mon_{\mathcal{O}}(\mathcal{C})$ for the full subcategory of
  $\Fun(\mathcal{O}, \mathcal{C})$ consisting of $\mathcal{O}$-monoids.
\end{defn}

\begin{remark}
  If $\mathcal{O}$ is a general algebraic pattern, then in
  \cite{patterns} we defined a \emph{Segal
    $\mathcal{O}$-object} to be a functor $F \colon \mathcal{O} \to
  \mathcal{C}$ such that $F|_{\mathcal{O}^{\xint}}$ is a right
  Kan extension of $F|_{\mathcal{O}^{\el}}$, or equivalently if for
  every object $O \in \mathcal{O}$ the canonical map
  \[ F(O) \to \lim_{E \in \mathcal{O}^{\el}_{O/}} F(E) \] is an
  equivalence. Thus an $\mathcal{O}$-monoid is just a special case of
  a Segal $\mathcal{O}$-object. We choose to use different terminology
  for this and a few other concepts to emphasize the special features
  of cartesian patterns and the parallels between our definitions and
  the special cases that are studied in \cite{ha,bar}.
\end{remark}

\begin{lemma}\label{lem:monoidpb}
  Any morphism of cartesian patterns $f \colon \mathcal{O} \to
  \mathcal{P}$ gives by composition a functor
  \[ f^{*} \colon \Mon_{\mathcal{P}}(\mathcal{C}) \to \Mon_{\mathcal{O}}(\mathcal{C}).\]
\end{lemma}
\begin{proof}
  Immediate, since for $O \in \mathcal{O}$ we have $|O| \cong |f(O)|$
  and $f(\rho_{i}^{O}) \simeq \rho_{i}^{f(O)}$.
\end{proof}

\begin{remark}\label{OintmonRKE}
  Let $\mathcal{O}$ be a cartesian pattern and $\mathcal{C}$ an
  \icat{} with finite products. Then so is
  $\mathcal{O}^{\xint}$, and a functor $F \colon \mathcal{O} \to
  \mathcal{C}$ is an $\mathcal{O}$-monoid \IFF{} its restriction
  $F|_{\mathcal{O}^{\xint}}$ is an
  $\mathcal{O}^{\xint}$-monoid. Moreover, the \icat{}
  $\Mon_{\mathcal{O}^{\xint}}(\mathcal{C})$ is precisely the full
  subcategory of $\Fun(\mathcal{O}^{\xint}, \mathcal{C})$ consisting
  of functors that are right Kan-extended from $\mathcal{O}^{\el}$, so
  that right Kan extension gives an equivalence
  \[ \Mon_{\mathcal{O}^{\xint}}(\mathcal{C}) \simeq
    \Fun(\mathcal{O}^{\el}, \mathcal{C}).\]
\end{remark}

\begin{remark}\label{rmk:Monsifted}
  Let $\mathcal{C}$ be an \icat{} with sifted colimits and finite
  products where the cartesian product preserves sifted colimits in
  each variable. If $\mathcal{O}$ is a cartesian pattern,
  then the full subcategory
  $\Mon_{\mathcal{O}}(\mathcal{C})$ is closed under sifted colimits
  in
  $\Fun(\mathcal{O}, \mathcal{C})$: given a sifted diagram $\phi
  \colon \mathcal{I} \to \Mon_{\mathcal{O}}(\mathcal{C})$ its colimit
  in $\Fun(\mathcal{O}, \mathcal{C})$, which is computed pointwise,
  satisfies
  \[ (\colim_{\mathcal{I}}\phi)(O) \isoto \colim_{\mathcal{I}}
    \left(\phi(O_{1}) \times \cdots \times \phi(O_{n}) \right) \simeq
    \left(\colim_{\mathcal{I}} \phi(O_{1})\right) \times \cdots \times
    \left(\colim_{\mathcal{I}} \phi(O_{n})\right).\]
  It follows that for any morphism $f \colon \mathcal{O} \to
  \mathcal{P}$ of cartesian patterns, the functor $f^{*} \colon
  \Mon_{\mathcal{P}}(\mathcal{C}) \to \Mon_{\mathcal{O}}(\mathcal{C})$
  preserves sifted colimits.
\end{remark}

\section{Examples of Cartesian Patterns}\label{sec:pattex}
In this section we mention some examples of cartesian patterns, and
indicate where our definitions specialize to
more familiar notions:

\begin{example}
  For the base pattern $\xF_{*}^{\flat}$, an $\xF_{*}^{\flat}$-monoid
  is precisely a commutative monoid (in
  the sense considered in \cite{ha}, but going back to Segal's work
  \cite{SegalCatCohlgy} on special $\Gamma$-spaces). 
\end{example}

\begin{example}
  Let $\simp^{\op, \flat}$ denote the simplex category with the usual
  inert--active factorization system (where the inert maps in $\simp$
  are the subinterval inclusions and the active maps are those that
  preserve the end points) and $[1]$ as the unique elementary
  object. This is a cartesian pattern using the map to $\xF_{*}$ given
  by $|[n]| = \angled{n}$, and with the map $|\phi|$ for $\phi \colon
  [n] \to [m]$ in $\simp$ given by
  \[ |\phi|(i) =
    \begin{cases}
      j, & \phi(j-1) < i \leq \phi(j),\\
      0, & \text{otherwise}.
    \end{cases}
  \]
  Then a $\simp^{\op,\flat}$-monoid is an associative monoid.
\end{example}

\begin{example}
  The product $\simp^{n,\op}$ has a factorization system with inert
  and active maps defined componentwise. We get a cartesian pattern
  $\simp^{n,\op,\flat}$ by taking $([1],\ldots,[1])$ to be the unique
  elementary object and the map to $\xF_{*}$ to be given by the composite
  \[ \simp^{n,\op} \to \xF_{*}^{n} \to \xF_{*} \] where the second map
  is the ``smash product'' of pointed finite sets. (This takes
  $(\langle m_1\rangle, \ldots, \langle m_n \rangle)$ to
  $\langle \prod_{i=1}^n m_i\rangle$; see \cite[Notation 2.2.5.1]{ha}
  for a precise description.) Then a $\simp^{n,\op,\flat}$-monoid can
  be described as an $n$-fold iterated associative monoid, or
  equivalently an $\mathbb{E}_{n}$-monoid by the Dunn--Lurie Additivity
  Theorem.
\end{example}

\begin{example}
  If $\Phi$ is a perfect operator category in the sense of \cite{bar}
  and $\Lambda(\Phi)$ is its Leinster category, then there is a
  cartesian pattern $\Lambda(\Phi)^{\flat}$ given by the inert--active
  factorization system of \cite{bar}, with the terminal object of
  $\Phi$ as the unique elementary object and the map to $\xF_{*}$ that
  induced by the unique operator morphism to $\xF$.  
  By choosing $\Phi$ to be the
  operator categories $\xF$ of finite sets, $\mathbb{O}$ of finite
  ordered sets, and the cartesian product $\mathbb{O}^{n}$, this
  example specializes to the previous ones. Another example is the
  iterated \emph{wreath product} $\mathbb{O}^{\wr n}$, whose Leinster
  category is Joyal's category $\bbTheta_{n}^{\op}$ of $n$-dimensional
  pasting diagrams. It is proved in
  \cite{bar} that $\bbTheta_{n}^{\op,\flat}$-monoids are equivalent to
  $\simp^{n,\op,\flat}$-monoids, and so are also equivalent to
  $\mathbb{E}_{n}$-monoids by the additivity theorem.
\end{example}

\begin{example}
  Let $\bbO$ be the \emph{dendroidal category} of
  \cite{MoerdijkWeiss}, with the active--inert factorization system
  described in \cite{Kock,ChuHaugsengHeuts}. Then $\bbO^{\op,\flat}$
  denotes the cartesian pattern given by this factorization system,
  with the corollas as the elementary objects and the functor to
  $\xF_{*}$ as defined in \cite{ChuHaugseng}, given by counting the
  number of corollas in a tree.  An $\bbO^{\op,\flat}$-monoid in
  $\mathcal{S}$ then describes a (pointed) one-object \iopd{}.
\end{example}

\begin{example}
  The category $\bbGamma$ of acyclic connected finite graphs defined
  by Hackney, Robertson, and Yau in \cite{HRYProperad} has an
  active--inert factorization system by \cite[2.4.14]{Kock_Properads}
  (where the active maps are called ``refinements'' and the inert maps
  are called ``convex open inclusions''). We then write
  $\bbGamma^{\op, \flat}$ for the algebraic pattern given by this
  factorization system, with the graphs with exactly one vertex as
  the elementary objects. The functor
  $\bbGamma^{\op, \flat}\to \mathbb{F}_*^\flat$ given by counting the
  number of vertices in a graph exhibits $\bbGamma^{\op, \flat}$ as a
  cartesian pattern. According to \cite{PropLect},
  $\bbGamma^{\op, \flat}$-monoids in $\xS$ model (pointed) one-object
  $\infty$-properads.
\end{example}

\begin{example}\label{ex:opdcart}
  Any (symmetric) \iopd{} $\mathcal{O} \to \xF_{*}$ as in \cite{ha}
  has an inert--active factorization system where the inert morphisms
  are the cocartesian morphisms lying over inert morphisms in
  $\xF_{*}$ and the active ones are those that lie over active
  morphisms in $\xF_{*}$. If we take $\mathcal{O}^{\flat}$ to be the
  algebraic pattern defined by this factorization system, with the elementary objects
  those that map to $\angled{1}$ (so $\mathcal{O}^{\el}$ is the
  underlying \igpd{} $\mathcal{O}_{\angled{1}}^{\simeq}$ of the fibre
  at
  $\angled{1}$), then the given map to $\xF_{*}$ exhibits
  $\mathcal{O}^{\flat}$ as a cartesian pattern. (We will discuss a
  generalization of this class of cartesian patterns in
  \cref{rmk:Oopd}.)
\end{example}

\begin{example}
  A \emph{generalized \iopd{}} $\mathcal{E} \to \xF_{*}$ as defined in
  \cite[\S 2.3.2]{ha} also has an inert--active factorization system,
  defined in the same way as for \iopds{}. If we again choose the
  elementary objects to be those that lie over $\angled{1}$, we get a
  cartesian pattern $\mathcal{E}^{\flat}$ using the given functor to
  $\xF_{*}$. 
\end{example}

\section{$\mathcal{O}$-Monoidal $\infty$-Categories and Algebras}\label{sec:Omonoidal}
In this section we define $\mathcal{O}$-monoidal \icats{} and
$\mathcal{O}$-algebras in them.

\begin{defn}
  Let $\mathcal{O}$ be a cartesian pattern. An
  \emph{$\mathcal{O}$-monoidal \icat{}} is a cocartesian fibration
  $\mathcal{V}^{\otimes} \to \mathcal{O}$ whose associated functor
  $\mathcal{O} \to \CatI$ is an $\mathcal{O}$-monoid. (We will often
  not mention the fibration explicitly and simply say that
  $\mathcal{V}^{\otimes}$ is an $\mathcal{O}$-monoidal \icat{}.)
\end{defn}

\begin{notation}
  Let $\mathcal{V}^{\otimes} \to \mathcal{O}$ be an
  $\mathcal{O}$-monoidal \icat{}. If $O$ is an object of $\mathcal{O}$
  lying over $\angled{n}$ in $\xF_{*}$, we will often denote by
  $O(V_{1},\ldots,V_{n})$ or just $O(V_{i})$ the
  object of $\mathcal{V}^{\otimes}_{O}$ that corresponds to
  $(V_{1},\ldots,V_{n})$ under the equivalence
  \[ \mathcal{V}^{\otimes}_{O} \simeq \prod_{i = 1}^{n}
    \mathcal{V}_{O_{i}} \]
  given by the cocartesian pushforwards over $\rho_{i}^{O}$.
\end{notation}

\begin{remark}
  If $f \colon \mathcal{O} \to \mathcal{P}$ is a morphism of cartesian
  patterns, then it follows from \cref{lem:monoidpb} (reinterpreted
  through the straightening equivalence for functors to $\CatI$) that
  base change along $f$ takes a $\mathcal{P}$-monoidal \icat{}
  $\mathcal{V}^{\otimes}$ to an $\mathcal{O}$-monoidal \icat{} $f^{*}\mathcal{V}^{\otimes}$.
\end{remark}

\begin{remark}
  $\mathcal{O}$-monoidal \icats{} are a special case of Segal
  $\mathcal{O}$-fibrations in the terminology of \cite{patterns}.
\end{remark}

\begin{defn}\label{def Oalg}
  Let $\mathcal{V}^{\otimes}$ be an $\mathcal{O}$-monoidal \icat{}. An
  \emph{$\mathcal{O}$-algebra} in $\mathcal{V}^{\otimes}$ is a section
  \[ 
    \begin{tikzcd}
      \mathcal{V}^{\otimes}\arrow{d} \\
      \mathcal{O} \arrow[bend left=35]{u}{A}
    \end{tikzcd}
  \]
  such that $A$ takes inert morphisms in $\mathcal{O}$ to cocartesian
  morphisms in $\mathcal{V}^{\otimes}$. We write
  $\Alg_{\mathcal{O}}(\mathcal{V})$ for the full subcategory of
  $\Fun_{/\mathcal{O}}(\mathcal{O}, \mathcal{V}^{\otimes})$ spanned by
  the $\mathcal{O}$-algebras. 
\end{defn}

\begin{example}
  Taking $\mathcal{O}$ to be $\xF_{*}^{\flat}$ our definitions
  specialize to symmetric monoidal \icats{} and commutative algebras
  as defined in \cite{ha}, while if we take $\mathcal{O}$ to be the
  cartesian pattern associated to an \iopd{}, we get the notions of
  $\mathcal{O}$-monoidal \icats{} and $\mathcal{O}$-algebras of
  \cite{ha}. Similarly, from $\simp^{\op,\flat}$ we get monoidal
  \icats{} and associative algebras.
\end{example}

\begin{defn}
  More generally, if $f \colon \mathcal{O} \to \mathcal{P}$ is a
  morphism of cartesian patterns and $\mathcal{V}^{\otimes}$ is a
  $\mathcal{P}$-monoidal \icat{}, then an \emph{$\mathcal{O}$-algebra}
  in $\mathcal{V}^{\otimes}$ is a commutative triangle
  \[
    \begin{tikzcd}
      \mathcal{O} \arrow{rr}{A} \arrow{dr}[swap]{f} & &
      \mathcal{V}^{\otimes} \arrow{dl} \\
      & \mathcal{P}
    \end{tikzcd}
  \]
  such that $A$ takes inert morphisms in $\mathcal{O}$ to cocartesian
  morphisms in $\mathcal{V}^{\otimes}$. We write
  $\Alg_{\mathcal{O}/\mathcal{P}}(\mathcal{V})$ for the full
  subcategory of
  $\Fun_{/\mathcal{P}}(\mathcal{O}, \mathcal{V}^{\otimes})$ spanned by
  the $\mathcal{O}$-algebras; if $\mathcal{P}$ is clear from the
  context, we will sometimes just write
  $\Alg_{\mathcal{O}}(\mathcal{V})$.  Base change along $f$ induces a
  natural equivalence
  \[\Alg_{\mathcal{O}/\mathcal{P}}(\mathcal{V}) \simeq
  \Alg_{\mathcal{O}}(f^{*}\mathcal{V}).\]
\end{defn}

We can view $\mathcal{O}$-algebras as a special case of morphisms of
cartesian patterns, using the following canonical pattern structure on
an $\mathcal{O}$-monoidal \icat{}:
\begin{defn}
  Let $\pi \colon \mathcal{V}^{\otimes} \to \mathcal{O}$ be an
  $\mathcal{O}$-monoidal \icat{}. We say a morphism in
  $\mathcal{V}^{\otimes}$ is \emph{active} if it lies over an active
  morphism in $\mathcal{O}$, and \emph{inert} if it is cocartesian and
  lies over an inert morphism in $\mathcal{O}$. The inert and active
  morphisms then form a factorization system on
  $\mathcal{V}^{\otimes}$ by \cite[Proposition 2.1.2.5]{ha}; we make
  $\mathcal{V}^{\otimes}$ an algebraic pattern using this
  factorization system and all the objects that lie over elementary
  objects in $\mathcal{O}$ as elementary objects. The composite map
  $\mathcal{V}^{\otimes} \to \mathcal{O} \to \xF_{*}$ then exhibits
  $\mathcal{V}^{\otimes}$ as a cartesian pattern.
\end{defn}

\begin{remark}
  Let $\mathcal{V}^{\otimes}$ be an $\mathcal{O}$-monoidal \icat{} and
  $f \colon \mathcal{P} \to \mathcal{O}$ a morphism of cartesian
  patterns. Given a commutative triangle
  \[
    \begin{tikzcd}
      \mathcal{P} \arrow[rr, "F"]\arrow[rd, "f"{swap}] & &
      \mathcal{V}^{\otimes} \arrow[ld]\\
      & \xxO,
    \end{tikzcd}
  \]
  the functor $F$ is a morphism of cartesian patterns \IFF{} it
  preserves inert morphisms, \ie{} \IFF{} it is a
  $\mathcal{P}$-algebra, since the commutativity of the diagram
  automatically implies that $F$ preserves active morphisms and
  elementary objects.
\end{remark}

\begin{defn}
  If $\mathcal{V}^{\otimes}$ and $\mathcal{W}^{\otimes}$ are
  $\mathcal{O}$-monoidal \icats{}, then a \emph{lax
    $\mathcal{O}$-monoidal functor} between them is a commutative
  triangle
  \[
    \begin{tikzcd}
      \mathcal{V}^{\otimes} \arrow{rr}{F} \arrow{dr}& &
      \mathcal{W}^{\otimes} \arrow{dl} \\
       & \mathcal{O}
    \end{tikzcd}
  \]
  such that $F$ preserves inert morphisms. Equivalently, a lax
  $\mathcal{O}$-monoidal functor is a morphism of
  algebraic patterns over $\mathcal{O}$ or a
  $\mathcal{V}^{\otimes}$-algebra in $\mathcal{W}^{\otimes}$ over
  $\mathcal{O}$. If the
  functor from $\mathcal{V}^{\otimes} \to \mathcal{W}^{\otimes}$
  preserves \emph{all} cocartesian morphisms over $\mathcal{O}$, we
  call it an \emph{$\mathcal{O}$-monoidal functor}.
\end{defn}

\begin{remark}\label{Alg2fun}
  Since $\Alg_{\mathcal{O}}(\mathcal{V})$ is a full subcategory of
  $\Fun_{/\mathcal{O}}(\mathcal{O}, \mathcal{V}^{\otimes})$, it  
  is not just functorial
  in the $\mathcal{O}$-monoidal \icat{} $\mathcal{V}^{\otimes}$, but
  even 2-functorial: a lax $\mathcal{O}$-monoidal functor
  $F \colon \mathcal{V}^{\otimes} \to \mathcal{W}^{\otimes}$ gives a
  functor
  \[ F_{*}\colon \Alg_{\mathcal{O}}(\mathcal{V}) \to
    \Alg_{\mathcal{O}}(\mathcal{W}) \] given by composition with $F$,
  and a natural transformation $\eta \colon F \to F'$ over
  $\mathcal{O}$ gives (again by composition) a natural transformation
  $\eta_{*} \colon F_{*} \to F'_{*}$. This means, for example, that
  any adjunction between $\mathcal{O}$-monoidal \icats{} induces by
  composition an adjunction on \icats{} of $\mathcal{O}$-algebras.
\end{remark}

\begin{notation}\label{not:V}
  If $\mathcal{V}^{\otimes} \to \mathcal{O}$ is an
  $\mathcal{O}$-monoidal \icat{}, we'll write $\mathcal{V} :=
  \mathcal{V}^{\otimes} \times_{\mathcal{O}}
  \mathcal{O}^{\el}$. In particular, we
  have $\xV^\otimes_E\simeq \xV_E$ for every $E
  \in \mathcal{O}^{\el}$. As this notation is sometimes ambiguous,
  we'll also occasionally use $\mathcal{V}^{\otimes}_{/\el}$ instead of $\mathcal{V}$.
  (Note that $\mathcal{V} \simeq \mathcal{V}^{\otimes}_{/\el}$ must be
  distinguished from the \igpd{} $\mathcal{V}^{\el}$ of elementary
  objects, which is the underlying \igpd{} $\mathcal{V}^{\simeq}$ of
  the \icat{} $\mathcal{V}$.)
\end{notation}

\begin{notation}
  If $\mathcal{V}^{\otimes}$ is an $\mathcal{O}$-monoidal \icat{},
  then we write 
 \[\mathcal{V}^{\otimes}_{/\xint} := \mathcal{O}^{\xint}
   \times_{\mathcal{O}} \mathcal{V}^{\otimes}.\]
 (The \icat{} $\mathcal{V}^{\otimes}_{/\xint}$ must be distinguished
 from the subcategory $(\mathcal{V}^{\otimes})^{\xint}$ of inert
 morphisms in $\mathcal{V}^{\otimes}$:
 $\mathcal{V}^{\otimes}_{/\xint}$ contains \emph{all} morphisms that
 lie over inert morphisms in $\mathcal{O}$, while
 $(\mathcal{V}^{\otimes})^{\xint}$ contains only the
 \emph{cocartesian} ones.)
\end{notation}

\begin{lemma}\label{lem:AlgOint}
  If $\mathcal{V}^{\otimes}$ is an $\mathcal{O}$-monoidal \icat{},
  then there is a natural equivalence
  \[ \Alg_{\mathcal{O}^{\xint}/\mathcal{O}}(\mathcal{V}^{\otimes})
    \simeq \Fun_{/\mathcal{O}^{\el}}(\mathcal{O}^{\el},
    \mathcal{V}) \]
  between $\mathcal{O}^{\xint}$-algebras and sections of $\mathcal{V}
  \to \mathcal{O}^{\el}$.
\end{lemma}
\begin{proof}
  Pulling back $\mathcal{V}^{\otimes}$ to $\mathcal{O}^{\xint}$ we get
  (since all morphisms in $\mathcal{O}^{\xint}$ are inert)
  an equivalence
  \[ \Alg_{\mathcal{O}^{\xint}/\mathcal{O}}(\mathcal{V}^{\otimes})
    \simeq
    \Fun^{\cocart}_{/\mathcal{O}^{\xint}}(\mathcal{O}^{\xint},
    \mathcal{V}^{\otimes}_{/\xint}).\]
  By \cref{OintmonRKE} the cocartesian fibration
  $\mathcal{V}^{\otimes}_{\xint} \to \mathcal{O}^{\xint}$ corresponds
  to a functor $\mathcal{O}^{\xint} \to \CatI$ that
  is right Kan extended from $\mathcal{O}^{\el}$. Translating the
  universal property of right Kan extension along the straightening
  equivalence, we get for every cocartesian fibration $\mathcal{E} \to
  \mathcal{O}^{\xint}$ a natural equivalence
  \[ \Map^{\cocart}_{/\mathcal{O}^{\xint}}(\mathcal{E},
    \mathcal{V}^{\otimes}_{/\xint}) \simeq
    \Map^{\cocart}_{/\mathcal{O}^{\el}}(\mathcal{E}|_{\mathcal{O}^{\el}},
    \mathcal{V}). \]
  This upgrades to a natural equivalence of \icats{}
  \[ \Fun^{\cocart}_{/\mathcal{O}^{\xint}}(\mathcal{E},
    \mathcal{V}^{\otimes}_{/\xint}) \simeq
    \Fun^{\cocart}_{/\mathcal{O}^{\el}}(\mathcal{E}|_{\mathcal{O}^{\el}},
    \mathcal{V}), \]
  since for $\mathcal{C} \in \CatI$ there is
  a natural equivalence \[\Map_{\CatI}(\mathcal{C},
    \Fun^{\cocart}_{/\mathcal{O}^{\xint}}(\mathcal{E},
    \mathcal{V}^{\otimes})) \simeq
    \Map^{\cocart}_{/\mathcal{O}^{\xint}}(\mathcal{C} \times
    \mathcal{E}, \mathcal{V}^{\otimes}).\]
  In our case this gives
\[ \Fun^{\cocart}_{/\mathcal{O}^{\xint}}(\mathcal{O}^{\xint},
    \mathcal{V}^{\otimes}_{/\xint}) \simeq     \Fun^{\cocart}_{/\mathcal{O}^{\el}}(\mathcal{O}^{\el},
    \mathcal{V}) \simeq \Fun_{/\mathcal{O}^{\el}}(\mathcal{O}^{\el},
    \mathcal{V}), \]
  where the last equivalence holds since $\mathcal{O}^{\el}$ is an \igpd{}.
\end{proof}

\begin{remark}
  Similarly, if $f \colon \mathcal{O} \to \mathcal{P}$ is a morphism
  of cartesian patterns and $\mathcal{V}^{\otimes}$ is a
  $\mathcal{P}$-monoidal \icat{}, we have a natural equivalence
  \[ \Alg_{\mathcal{P}^{\xint}/\mathcal{O}}(\mathcal{V}^{\otimes})
    \simeq \Fun_{/\mathcal{O}^{\el}}(\mathcal{P}^{\el},
    \mathcal{V}).\]
\end{remark}

\begin{remark}\label{rmk:Oopd}
  The notion of $\mathcal{O}$-monoidal \icat{} can be weakened to that
  of an \emph{$\mathcal{O}$-\iopd{}}, which is a functor
  $p \colon \mathcal{E} \to \mathcal{O}$ such that:
  \begin{enumerate}[(i)]
  \item For every object $\overline{O}$ in $\mathcal{E}$ lying over
    $O \in \mathcal{O}$ and every inert morphism $\phi \colon O \intto O'$ in
    $\mathcal{O}$, there exists a $p$-cocartesian morphism
    $\overline{\phi} \colon \overline{O} \to \overline{O}'$ lying over $\phi$.
  \item For every object $O \in \mathcal{O}$, the functor
    \[ \mathcal{E}_{O} \to \prod_{i}\mathcal{E}_{O_{i}}, \]
    induced by the cocartesian morphisms over the inert maps $\rho_{i}^{O}$, is an
    equivalence.
  \item Given $\overline{O}$ in $\mathcal{E}_{O}$, choose cocartesian
    morphisms $\rho^{\overline{O}}_{i} \colon \overline{O} \to
    \overline{O}_{i}$ over $\rho^{O}_{i}$. Then for any $O' \in
    \mathcal{O}$ and $\overline{O}' \in \mathcal{E}_{O'}$, the
    commutative square
    \[
      \begin{tikzcd}
        \Map_{\mathcal{E}}(\overline{O}', \overline{O}) \arrow{r} \arrow{d} & \prod_{i} \Map_{\mathcal{E}}(\overline{O}',
        \overline{O}_{i}) \arrow{d} \\
        \Map_{\mathcal{O}}(O', O) \arrow{r} & \prod_{i} \Map_{\mathcal{O}}(O', O_{i})
      \end{tikzcd}
    \]
    is cartesian.
  \end{enumerate}
  This is the same as a \emph{weak Segal fibration} over $\mathcal{O}$
  as defined in \cite[\S 9]{patterns}; we use the term
  $\mathcal{O}$-\iopds{} to emphasize that this notion
  specializes to the (symmetric) \iopds{} of
  \cite{ha} over the pattern $\xF_{*}^{\flat}$, to non-symmetric (or
  planar) \iopds{} (as in \cite{enriched} and \cite[\S 4.1.3]{ha})
  over the pattern $\simp^{\op,\flat}$, and more generally for an
  operator category $\Phi$ to $\Phi$-\iopds{} in the sense of
  \cite{bar} for the pattern $\Lambda(\Phi)^{\flat}$. If $\mathcal{E}$ is an
  $\mathcal{O}$-\iopd{}, then it has a canonical pattern structure by
  \cite[Lemma 9.4]{patterns}, where the inert morphisms are those that
  are cocartesian and lie over inert morphisms in $\mathcal{O}$, and
  the active morphisms are those that lie over active morphisms in
  $\mathcal{O}$; this is a cartesian
  pattern via the composite $\mathcal{E} \to \mathcal{O}
  \xto{|\blank|} \xF_{*}$.
\end{remark}

\section{Monoids as Algebras}\label{sec:algmonoid}
Let $\mathcal{O}$ be a cartesian pattern. Our goal in this section is
to prove that $\xxO$-monoids in an $\infty$-category $\mathcal{C}$
with finite products are equivalent to $\mathcal{O}$-algebras in an
$\mathcal{O}$-monoidal \icat{} determined by the cartesian product in
$\mathcal{C}$. More precisely, we will prove the following
generalization of \cite[Proposition 2.4.1.7]{ha}:
\begin{propn}\label{propn:moniscartalg}
  Suppose $\mathcal{C}$ is an \icat{} with finite products. Let
  $\mathcal{C}^{\times} \to \xF_{*}$ be the corresponding cartesian
  symmetric monoidal \icat{}. If $\mathcal{O}$ is a cartesian pattern,
  then there is a natural equivalence
  \[ \Mon_{\mathcal{O}}(\mathcal{C}) \simeq
    \Alg_{\mathcal{O}/\xF_{*}}(\mathcal{C}^{\times}).\] 
\end{propn}
Here the cartesian symmetric monoidal \icat{} $\mathcal{C}^{\times}$
is defined as in \cite{ha}, and before we prove the proposition we
need to recall this definition, which we phrase using the terminology
of algebraic patterns:
\begin{definition}\label{def pattern structure}
  Let $\Gamma^\times$ (cf.~\cite[Notation 2.4.1.2]{ha}) denote the full
  subcategory of $\mathbb F_*^{\Delta^1}$ spanned by the inert morphisms in $\mathbb F_*$.   Let $\ev_{0},\ev_{1} \colon \Gamma^{\times} \to \xF_{*}$ denote the
  functors given by evaluation at $0$ and $1$, respectively.
\end{definition}

\begin{lemma}
  $\ev_{0} \colon \Gamma^{\times} \to \xF_{*}$ is a cartesian fibration.
\end{lemma}
\begin{proof}
  This follows from the uniqueness of inert--active factorizations, by
  the dual of the proof of \cite[Proposition 7.2]{patterns}.
\end{proof}

\begin{defn}
  We can apply the construction of \cite[Corollary 3.2.2.12]{ht} to
  the cartesian fibration $\ev_{0}$ and the cocartesian fibration
  $\mathcal{C} \times \xF_{*} \to \xF_{*}$ to obtain a cocartesian
  fibration $\overline{\mathcal{C}}^{\times} \to \xF_{*}$
  with the universal property that there is a natural equivalence
  \[ \Map_{/\xF_{*}}(K, \overline{\mathcal{C}}^{\times}) \simeq \Map(K
    \times_{\xF_{*}} \Gamma^{\times}, \mathcal{C}) \]
  for any \icat{} $K$ over $\xF_{*}$. In particular the \icat{} 
  $\overline{\mathcal{C}}^{\times}$ is
  given
  fibrewise over $\angled{n}$ by
  $\Fun(\Gamma^{\times}_{\angled{n}}, \mathcal{C})$. (And by the dual of
  \cite[Proposition 7.3]{freepres} this is indeed the corresponding
  functor.) If
  $\mathcal{C}$ is an \icat{} with finite products, we define
  $\mathcal{C}^{\times}$ to be the full subcategory of
  $\overline{\mathcal{C}}^{\times}$ whose objects over
  $\angled{n}$ are
  the functors $F \colon \Gamma^{\times}_{\angled{n}} \to
  \mathcal{C}$ such that for every object $\phi \colon \angled{n} \to \angled{m}$,
  the map $F(\phi) \to \prod_{i=1}^{m} F(\rho_{i}\phi)$ is an equivalence.
\end{defn}

\begin{propn}[Lurie, {\cite[Proposition 2.4.1.5]{ha}}]
  If $\mathcal{C}$ is an \icat{} with products, then the
  restricted functor $\mathcal{C}^{\times} \to \xF_{*}$ is a symmetric monoidal \icat{}. \qed
\end{propn}

\begin{remark}\label{rmk:mortimes}
  From the definition it follows that any functor $g \colon
  \mathcal{C} \to \mathcal{D}$ induces a natural morphism of
  cocartesian fibrations $\overline{g}^{\times} \colon
  \overline{\mathcal{C}}^{\times} \to
  \overline{\mathcal{D}}^{\times}$; if $g$ preserves products, then
  this restricts to a natural symmetric monoidal functor $g^{\times}
  \colon \mathcal{C}^{\times} \to \mathcal{D}^{\times}$.
\end{remark}

\begin{defn}
  We say a morphism
  \[
    \begin{tikzcd}
      \langle m\rangle \arrow[r] \arrow[d, inert] & \langle m'\rangle \arrow[d, inert]\\
      \langle n\rangle\arrow[r] & \langle n'\rangle
    \end{tikzcd}
  \]
  in $\Gamma^{\times}$ is inert or active if the horizontal maps in the square
  are inert or active, respectively.
\end{defn}

\begin{lemma}
  The inert and active morphisms determine a factorization system on $\Gamma^{\times}$.
\end{lemma}
\begin{proof}
  Given a morphism as above, then by the factorization system on
  $\mathbb F_*^\flat$ we get horizontal inert-active factorizations
    \[
    \begin{tikzcd}
    \langle m\rangle \arrow[r, inert] \arrow[d, inert] &\langle m''\rangle \arrow[r, active]\arrow[d, dotted, inert]& \langle m'\rangle\arrow[d, inert]\\
    \langle n \rangle \arrow[r, inert] &\langle n'' \rangle \arrow[r, active]&\langle n' \rangle.
    \end{tikzcd}
    \]
    The existence of a factorization system on $\Gamma^\times$ is
    equivalent to the existence of an inert morphism indicated by the
    dotted arrow.  The map
    $\langle m''\rangle\actto \langle m' \rangle\intto\langle n'
    \rangle$ factors into
    $\langle m''\rangle \intto\langle q\rangle\actto\langle n'
    \rangle$. The essential uniqueness of factorizations implies that
    $\langle m\rangle \intto \langle n \rangle\intto\langle n''
    \rangle \actto\langle n' \rangle$ coincides with
    $\langle m\rangle \intto \langle m''\rangle \intto \langle
    q\rangle \actto\langle n' \rangle$. Hence, there is an
    equivalence $\langle n'' \rangle \simeq \langle q\rangle$ which
    proves the existence of the dotted morphism in the diagram above.
\end{proof}

\begin{defn}
  We give $\Gamma^{\times}$ an algebraic pattern structure using this
  factorization system, with $\id_{\langle 1\rangle}$ as the only
  elementary object.
\end{defn}
  
\begin{remark}\label{rem morphism}
  It is clear from the definition of $\Gamma^\times$ that the evaluations
  at $0$ and $1$ give morphisms of algebraic patterns
  $\ev_{0},\ev_{1}\colon \Gamma^\times\to \mathbb F_*^\flat$. Moreover, the evaluation at
  $1$ exhibits $\Gamma^{\times}$ as a cartesian pattern.
\end{remark}

\begin{remark}
  If $\mathcal{O}$ is a
  cartesian pattern, then we can equip the pullback
  $\xxO\times_{\mathbb F_*} \Gamma^\times$ over evaluation at $\{0\}$
  with a canonical pattern structure (which according to
  \cite[Corollary 5.5]{patterns} gives the fibre product in the
  \icat{} of algebraic patterns). Here a morphism in
  $\xxO\times_{\mathbb F_*}\Gamma^\times$ that is given by $O\to O'$
  and a commutative square
  \[
    \begin{tikzcd}
      {|O|} \arrow[r] \arrow[d, inert] & {|O'|} \arrow[d, inert]\\
      \langle n\rangle\arrow[r] & \langle n'\rangle
    \end{tikzcd}
  \]
  is inert or active if $O\to O'$ and the horizontal maps in the
  square are inert or active, respectively, and the elementary objects
  are the pairs $(E, \id_{\angled{1}})$ with $E \in
  \mathcal{O}^{\el}$. The composite
  \[ \mathcal{O} \times_{\xF_{*}} \Gamma^{\times} \to \Gamma^{\times}
    \xto{\ev_{1}} \xF_{*} \]
  exhibits $\mathcal{O} \times_{\xF_{*}} \Gamma^{\times}$ as a
  cartesian pattern.
\end{remark}

\begin{definition}
  Let $i \colon \xF_{*} \to \Gamma^{\times}$ be the functor that takes
  $\angled{n}$ to $\id_{\angled{n}}$ (given by composition with
  $\Delta^{1} \to \Delta^{0}$); this is  fully faithful and a
  morphism of cartesian patterns. If $\mathcal{O}$ is a cartesian
  pattern, we define
  $i_{\mathcal{O}} \colon \mathcal{O} \to
  \mathcal{O}\times_{\xF^{*}}\Gamma^{\times}$ by pulling back $i$, so
  that $i_{\mathcal{O}}$ takes $O\in \xxO$ to $(O, \id_{|O|})$. This
  is again fully faithful, and since the target is a fibre product of
  patterns it is also a morphism of cartesian patterns.
\end{definition}  
  
\begin{remark}\label{rem active lifting}
  An active morphism
  $(O, |O|\intto \langle n\rangle)\actto i_{\mathcal{O}}(O')\simeq (O',
  |O'| \xlongequal{\phantom{i}} |O'|)$
  is given by an active map $O \actto O'$ together with a commutative
  square
  \[
    \begin{tikzcd}
      {|O|} \arrow[r, active] \arrow[d, inert] & {|O'|}\hphantom{.} \arrow[d, equal]\\
      \langle n\rangle\arrow[r, active] & {|O'|}.
    \end{tikzcd}
  \]
  The uniqueness of the inert--active factorization then implies that
  $|O| \cong \langle n\rangle$, and hence $i_{\mathcal{O}}$ has unique
  lifting of active morphisms in the sense of \cite[Definition
  6.1]{patterns}.
\end{remark}

\begin{lemma}\label{lem:Moneq}
  Let $\mathcal{O}$ be a cartesian pattern and $\mathcal{C}$ an
  \icat{} with finite products. Composition with $i_{\mathcal{O}}$ and right Kan
  extension along it restrict to an adjoint equivalence
  \[ i_{\mathcal{O}}^{*} : \Mon_{\mathcal{O} \times_{\xF_{*}}
      \Gamma^{\times}}(\mathcal{C})  \xrightleftarrows[\sim]{\sim}   
    \Mon_{\mathcal{O}}(\mathcal{C}) : i_{\mathcal{O},*}. \]
\end{lemma}

\begin{proof}
  Since $i_{\mathcal{O}}$ is a morphism of cartesian patterns, the
  functor $i_{\mathcal{O}}^{*}$ restricts to the full subcategories of
  monoids. Moreover, since $i_{\mathcal{O}}$ has unique lifting of
  active morphisms, its right adjoint $i_{\mathcal{O},*}$ also
  restricts to monoids by \cite[Proposition 6.3]{patterns}, and it is
  fully faithful since $i_{\mathcal{O}}$ is fully faithful. It remains
  to show that $i_{\mathcal{O},*}$ is essentially surjective on
  monoids. Let
  $M \colon \mathcal{O} \times_{\xF_{*}} \Gamma^{\times} \to
  \mathcal{C}$ be a monoid, and let
  $(O, j\colon |O| \intto \angled{n})$ be an object of
  $\mathcal{O} \times_{\xF_{*}} \Gamma^{\times}$. We must show that
  the canonical map $M(O,j) \to
  (i_{\mathcal{O},*}i_{\mathcal{O}}^{*}M)(O,j)$ is an equivalence. But
  we have a commutative square
  \[
    \begin{tikzcd}
      M(O,j) \arrow{r} \arrow{d} &
      (i_{\mathcal{O},*}i_{\mathcal{O}}^{*}M)(O,j) \arrow{d} \\
      \prod_{i=1}^{n} M(O_{i},\id_{\angled{1}}) \arrow{r} &
      \prod_{i =1}^{n} (i_{\mathcal{O},*}i_{\mathcal{O}}^{*}M)(O_{i},\id_{\angled{1}}),
    \end{tikzcd}
  \]
  where the vertical maps are equivalences since $M$ and
  $i_{\mathcal{O},*}i_{\mathcal{O}}^{*}M$ are monoids. Moreover, the
  bottom horizontal map is an equivalence since $i_{\mathcal{O},*}$ is
  fully faithful and the objects $(O_{i},\id_{\angled{1}}) \simeq
  i_{\mathcal{O}}(O_{i})$ are in the image of $i_{\mathcal{O}}$. The
  top horizontal map is therefore also an equivalence.
\end{proof}

\begin{lemma}\label{lem:Algeq}
  Under the natural equivalence
  \[\Fun_{/\xF_{*}}(\mathcal{O}, \overline{\mathcal{C}}^{\times}) \simeq
    \Fun(\mathcal{O} \times_{\xF_{*}} \Gamma^{\times}, \mathcal{C}),\]
  the full subcategory $\Mon_{\mathcal{O} \times_{\xF_{*}}
    \Gamma^{\times}}(\mathcal{C})$ is identified with
  $\Alg_{\mathcal{O}/\xF_{*}}(\mathcal{C}^{\times})$.
\end{lemma}
\begin{proof}
  By definition, a functor
  $F \colon \mathcal{O} \to \overline{\mathcal{C}}^{\times}$ over
  $\xF_{*}$ lies in $\Alg_{\mathcal{O}/\xF_{*}}(\mathcal{C}^{\times})$
  \IFF{} $F$ factors through the full subcategory
  $\mathcal{C}^{\times}$, and $F$ takes inert morphisms to cocartesian
  morphisms. We can reformulate these requirements in terms of the
  corresponding functor
  $F' \colon \mathcal{O} \times_{\xF_{*}} \Gamma^{\times} \to
  \mathcal{C}$ as the following pair of conditions:
  \begin{enumerate}[(1)]
  \item The map
    \[ F'(O,\phi) \to \prod_{i = 1}^{n} F'(O,\rho_{i}\phi) \] is an
    equivalence for every object
    $(O, \phi \colon |O| \intto \angled{n})$.
  \item For every inert map $\psi \colon O' \intto O$
    in
    $\mathcal{O}$ the morphism
    \[ F'(O', \phi|\psi|) \to F'(O,\phi)\]
    is an equivalence.
  \end{enumerate}
  By the definition of $\mathcal{C}^{\times}$, condition (1) holds
  \IFF{} $F$ factors through $\mathcal{C}^{\times}$, while the
  description of cocartesian morphisms in $\xcc^\times$ in
  \cite[Proposition 2.4.1.5.(2)]{ha} shows that $F$ takes inert
  morphisms to cocartesian morphisms if and only if condition (2) holds.

  On the other hand, $F'$ is a monoid if for every object $(O,\phi)$
  the map
  \[ F'(O,\phi) \to \prod_{i=1}^{n} F'(O_{\phi^{-1}i},\id_{\angled{1}}) \]
  is an equivalence. To see that this is equivalent to the first
  pair of conditions, observe that we have commutative triangles
  \[
    \begin{tikzcd}
      F'(O,\phi) \arrow{rr} \arrow{dr} &  & \prod_{i = 1}^{n}
      F'(O,\rho_{i}\phi) \arrow{dl} \\
      & \prod_{i=1}^{n} F'(O_{\phi^{-1}i}, \id_{\angled{1}}),
    \end{tikzcd}
  \]
  \[
    \begin{tikzcd}
      F'(O',\phi|\psi|) \arrow{rr} \arrow{dr} &  & F'(O,\phi) \arrow{dl} \\
      & \prod_{i=1}^{n} F'(O_{\phi^{-1}i}, \id_{\angled{1}}).
    \end{tikzcd}
  \]
  First suppose conditions (1) and (2) hold. In the first triangle, the
  horizontal map is an equivalence by (1), while the right diagonal
  map is an equivalence by (2) using the identity
  $\rho_i\phi=\id_{\angled{1}}|\rho_i^\xxO|$. Hence the left diagonal
  map is also an equivalence, which shows that $F'$ is a monoid.

  Conversely, if $F'$ is a monoid, then both diagonal maps in the
  second triangle are equivalences (using the equivalences
  $O'_{(\phi|\psi|)^{-1}i}\simeq O_{\phi^{-1}i}$). Thus the horizontal
  map is also an equivalence, which gives condition (2). In the first
  triangle we then have that the left diagonal map is an equivalence
  since $F'$ is a monoid, and the right diagonal map is an
  equivalence by condition (2) applied to the inert maps
  $\rho_{i}^{O}$. Thus the horizontal map in the first triangle is an
  equivalence, which gives condition (1).
\end{proof}

\begin{proof}[Proof of Proposition~\ref{propn:moniscartalg}]
	Combine Lemmas~\ref{lem:Moneq} and \ref{lem:Algeq}. 
\end{proof}

\begin{remark}\label{rmk:monalgnat}
  If $f \colon \mathcal{O} \to \mathcal{P}$ is a morphism of cartesian
  patterns, then it is clear from the proof that the equivalence of
  \cref{propn:moniscartalg} is natural. Thus if $\mathcal{C}$ is an
  \icat{} with finite products, composition with $f$ gives a
  commutative square
  \[
    \begin{tikzcd}
      \Mon_{\mathcal{P}}(\mathcal{C}) \arrow{r}{\sim} \arrow{d}{f^{*}}
      & \Alg_{\mathcal{P}/\xF_{*}}(\mathcal{C}^{\times})
      \arrow{d}{f^{*}} \\
      \Mon_{\mathcal{O}}(\mathcal{C}) \arrow{r}{\sim} & \Alg_{\mathcal{O}/\xF_{*}}(\mathcal{C}^{\times})
    \end{tikzcd}
  \]
  Moreover, if $\mathcal{D}$ is another \icat{} with finite products
  and $g \colon \mathcal{C} \to \mathcal{D}$ is a product-preserving
  functor, then $g$ and the symmetric monoidal functor $g^{\times}
  \colon \mathcal{C}^{\times} \to \mathcal{D}^{\times}$ from
  \cref{rmk:mortimes} fit in a commutative square
  \[
    \begin{tikzcd}
      \Mon_{\mathcal{O}}(\mathcal{C}) \arrow{r}{\sim} \arrow{d}{g_{*}}
      & \Alg_{\mathcal{O}/\xF_{*}}(\mathcal{C}^{\times})
      \arrow{d}{g^{\times}_{*}} \\
      \Mon_{\mathcal{O}}(\mathcal{D}) \arrow{r}{\sim} & \Alg_{\mathcal{O}/\xF_{*}}(\mathcal{D}^{\times}),
    \end{tikzcd}
  \]
  which combines with composition with $f$ to give a natural
  commutative cube.
\end{remark}

\section{Day Convolution over Cartesian Patterns}\label{sec:dayconv}

In this section we introduce Day convolution for presheaves of spaces
on $\mathcal{O}$-monoidal \icats{}, where $\mathcal{O}$ is any
cartesian pattern. We generalize a simple and elegant approach to Day
convolution for presheaves due to Heine~\cite[\S 6.1]{HeineThesis}. Other constructions of Day
convolution (which work for functors to more general targets than just
spaces) are due to Glasman \cite{GlasmanDay} and Lurie~\cite{ha}. We
begin by describing a useful model of the universal cocartesian
fibration, using right fibrations:
\begin{notation}
  Let $\RFib$ denote the full subcategory of
  $\CatI^{\Delta^{1}}$ spanned by the right fibrations.
\end{notation}

\begin{propn}\label{propn:RFibfib}
  The functor $\ev_{1} \colon \RFib \to \CatI$, given by evaluation at
  $1 \in \Delta^{1}$, is a cartesian and cocartesian
  fibration, and the corresponding contravariant functor is equivalent to
  \[\PSh(\blank) := \Fun((\blank)^{\op},\mathcal{S}) \colon \CatI^{\op} \to \LCatI.\]
\end{propn}
\begin{proof}
  The \icat{} $\CatI$ has pullbacks, so the functor
  $\ev_{1} \colon \Fun(\Delta^{1}, \CatI) \to \CatI$ is a cartesian
  fibration, with cartesian morphisms given by pullback squares. Since
  the pullback of a right fibration is again a right fibration, 
  the full subcategory $\RFib$ inherits cartesian morphisms from
  $\Fun(\Delta^{1}, \CatI)$, which implies that $\ev_{1}\colon \RFib
  \to \CatI$ is a cartesian fibration. To show that it is also a
  cocartesian fibration it then suffices by \cite[Corollary
  5.2.2.5]{ht} to check that for any morphism
  $f \colon \mathcal{C}' \to \mathcal{C}$ in $\CatI$, the cartesian
  pullback functor
  \[ f^{*} \colon \RFib_{\mathcal{C}} \to \RFib_{\mathcal{C}'} \] has
  a left adjoint. Under the straightening equivalence, this functor
  corresponds to the functor
  $f^{*} \colon \PSh(\mathcal{C}) \to \PSh(\mathcal{C}')$ given by
  composition with $f^{\op}$, which indeed has a left adjoint given by
  left Kan extension along $f^{\op}$.

  To identify the corresponding functor we use the naturality of the
  straightening equivalence, as discussed in \cite[Appendix
  A]{freepres}. By a variant of \cite[Corollary A.32]{freepres}, the
  pseudo-naturality of straightening on the model category level
  induces a natural equivalence
  \[\RFib_{(\blank)} \isoto \Fun((\blank)^{\op}, \mathcal{S}).\]
  Here the functoriality in the domain is induced by
  strict pullbacks on the model
  category level using the covariant model structures on slices of
  simplicial sets \cite[\S 2.1]{ht}. 

  It thus suffices to identify this functor with that corresponding to
  our cartesian fibration. We can model $\RFib$ as a relative category
  by the full subcategory $\Fun(\Delta^{1}, \sSet)_{\RFib}$ of
  $\Fun(\Delta^{1}, \sSet)$ spanned by the right fibrations between
  quasicategories in the sense of \cite{ht}, with the weak equivalences inherited from
  the Joyal model structure on $\sSet$. Evaluation at $1$ gives a
  cartesian fibration of relative categories
  $\Fun(\Delta^{1}, \sSet)_{\RFib} \to \sSet$, whose associated
  functor from $\sSet$ to relative categories induces on the
  $\infty$-level the functor $\RFib_{(\blank)}$ we considered
  above. In this situation Hinich~\cite{hinloc} has shown that the
  functor associated to the induced cartesian fibration of \icats{} is
  that obtained from the associated functor to relative categories by
  inverting weak equivalences.
\end{proof}

\begin{defn}
  Let $\RSl$ denote
  the full subcategory of $\RFib$ spanned by the right fibrations given by slice categories, \ie{} right fibrations of
  the form $\mathcal{C}_{/c} \to \mathcal{C}$.
\end{defn}

\begin{cor}\label{cor:RSlYoneda}
  $\ev_{1} \colon \RSl \to \CatI$ is a cocartesian fibration, and the
  inclusion $\RSl \to \RFib$ preserves cocartesian morphisms. The
  corresponding functor $\CatI \to \CatI$ is the identity, and the
  inclusion corresponds to the Yoneda embeddings $\mathcal{C}
  \hookrightarrow \PSh(\mathcal{C})$.
\end{cor}
\begin{proof}
  Under the straightening equivalence $\RFib_{\mathcal{C}} \simeq
  \PSh(\mathcal{C})$, the slice fibration
  $\mathcal{C}_{/c} \to \mathcal{C}$ corresponds to the representable
  presheaf $\Map_{\mathcal{C}}(\blank, c)$. Moreover, for a functor
  $f \colon \mathcal{C} \to \mathcal{C}'$ the cocartesian pushforward
  $f_{!} \colon \RFib_{\mathcal{C}} \to \RFib_{\mathcal{C}'}$
  corresponds to the functor $f_{!} \colon \PSh(\mathcal{C}) \to
  \PSh(\mathcal{C}')$ given by left Kan extension along
  $f^{\op}$. This lives in a commutative square
  \[
    \begin{tikzcd}
      \mathcal{C} \arrow[hookrightarrow]{r}  \arrow{d}[swap]{f} &
      \PSh(\mathcal{C}) \arrow{d}{f_{!}} \\
      \mathcal{C}' \arrow[hookrightarrow]{r} & 
      \PSh(\mathcal{C}')
    \end{tikzcd}
  \]
  where the horizontal maps are the Yoneda embeddings. In particular,
  $f_{!}$ takes the representable presheaf $\Map_{\mathcal{C}}(\blank,
  c)$ to $\Map_{\mathcal{C}'}(\blank, f(c))$. In terms of right
  fibrations, this means
  \[ f_{!}(\mathcal{C}_{/c}\to \mathcal{C}) \simeq
    (\mathcal{C}'_{/f(c)} \to \mathcal{C}');\] thus $\RSl$ inherits
  cocartesian morphisms from $\RFib$, and this makes
  $\ev_{1} \colon \RSl \to \CatI$ a cocartesian fibration.  The
  corresponding functor $\CatI \to \LCatI$ is equivalent to that
  taking $\mathcal{C}$ to the full subcategory of $\PSh(\mathcal{C})$
  spanned by the representable presheaves; the naturality of the
  Yoneda embedding shows that this is equivalent to the identity of $\CatI$, as
  required, and that the inclusion $\RSl \hookrightarrow \RFib$ corresponds
  to the Yoneda embedding.
\end{proof}

\begin{cor}\label{cor:RSluniv}
  Suppose $\pi \colon \mathcal{F} \to \mathcal{B}$ is the cocartesian fibration
  corresponding to a functor $F \colon \mathcal{B} \to \CatI$. Then
  there is a pullback square of \icats{}
  \[
    \begin{tikzcd}
      \mathcal{F} \arrow{r} \arrow{d}[swap]{\pi} & \RSl \arrow{d}{\ev_{1}} \\
      \mathcal{B} \arrow{r}{F} & \CatI.
    \end{tikzcd}
  \]
  In other words, $\ev_{1} \colon \RSl \to \CatI$ is the universal
  cocartesian fibration.
\end{cor}
\begin{proof}
  Pullback of cocartesian fibrations corresponds to composition of
  functors, so $F^{*}\RSl \to \mathcal{B}$ is the cocartesian
  fibration for $\id \circ F \simeq F$. 
\end{proof}

Our next goal is to prove an  $\mathcal{O}$-monoidal version of
\cref{cor:RSluniv}, which needs some preliminaries.

\begin{remark}
  The \icats{} $\RFib$ and $\RSl$ are both closed under
  cartesian products as full subcategories of $\Fun(\Delta^{1}, \CatI)$. These \icats{}
  therefore have cartesian products, and the functor to $\CatI$ given
  by evaluation at $1$ in $\Delta^{1}$ preserves products. We
  hence have induced symmetric monoidal functors
  \[ \ev_{1}^{\times} \colon \RFib^{\times}, \RSl^{\times} \to
    \CatI^{\times}.\]
\end{remark}
Our next goal is to show that these functors are both cocartesian
fibrations, and indeed exhibit $\RFib^{\times}$ and $\RSl^{\times}$ as
$\CatI^{\times}$-monoidal \icats{}. To see this we apply the following
general criterion:
\begin{propn}\label{propn:moncocart}
  Let $\mathcal{O}$ be a cartesian pattern. Suppose
  $F \colon \mathcal{C}^{\otimes} \to \mathcal{D}^{\otimes}$ is an
  $\mathcal{O}$-monoidal functor between $\mathcal{O}$-monoidal
  \icats{} such that
  \begin{enumerate}[(1)]
  \item for every $E \in \mathcal{O}^{\el}$, the functor on fibres
    $F_{E} \colon \mathcal{C}_{E} \to \mathcal{D}_{E}$ is a
    cocartesian fibration,
  \item for every active morphism $\phi \colon O \actto E$ in
    $\mathcal{O}$ with $E  \in \mathcal{O}^{\el}$, in the
    commutative square
    \[
      \begin{tikzcd}
        \prod_{i} \mathcal{C}_{O_{i}}
        \arrow{r}{\phi_{!}^{\mathcal{C}}} \arrow{d}[swap]{\prod_{i} F_{O_{i}}} &
        \mathcal{C}_{E} \arrow{d}{F_{E}} \\
        \prod_{i} \mathcal{D}_{O_{i}}  \arrow{r}{\phi_{!}^{\mathcal{D}}} & \mathcal{D}_{E},
      \end{tikzcd}
      \]
    the functor $\phi_{!}^{\mathcal{C}}$ takes $\prod_{i}F_{O_{i}}$-cocartesian
    morphisms (\ie{} tuples of $F_{O_{i}}$-cocartesian morphisms) to $F_{E}$-cocartesian morphisms.
  \end{enumerate}
  Then:
  \begin{enumerate}[(i)]
  \item $F$ is a
    cocartesian fibration that exhibits $\mathcal{C}^{\otimes}$ as a
    $\mathcal{D}^{\otimes}$-monoidal \icat{}.
  \item Given a morphism $\psi \colon O(D_{i}) \to O'(D'_{j})$ in
    $\mathcal{D}^{\otimes}$ lying over $\phi \colon O \to O'$ in
    $\mathcal{O}$ and an object $O(C_{i})$ in $\mathcal{C}^{\otimes}$
    over $O(D_{i})$, the cocartesian morphism over $\psi$ with this
    domain is the composite
    \[ O(C_{i}) \to \phi_{!}O(C_{i}) \to \psi'_{!}\phi_{!}O(C_{i}) \]
    where $O(C_{i}) \to \phi_{!}O(C_{i})$ is cocartesian over $\phi$,
    the map $\psi$ factors uniquely (since $F$ preserves cocartesian
    morphisms) as
    \[ O(D_{i}) \to F(\phi_{!}O(C_{i})) \xto{\psi'} O'(D'_{j}), \]
    where $\psi'$ lies in
    the fibre $\mathcal{D}^{\otimes}_{O'}$ and $\phi_{!}O(C_{i})
    \to \psi'_{!}\phi_{!}O(C_{i})$ is $F_{O'}$-cocartesian over $\psi'$.
  \item If $f \colon \mathcal{P} \to \mathcal{D}^{\otimes}$ is a
  morphism of cartesian patterns, we have a natural cartesian square
  \[
    \begin{tikzcd}
      \Alg_{\mathcal{P}/\mathcal{D}^{\otimes}}(\mathcal{C}^{\otimes})
      \arrow{r} \arrow{d} &
      \Alg_{\mathcal{P}/\mathcal{O}}(\mathcal{C}^{\otimes}) \arrow{d} \\
      \{f\} \arrow{r} & \Alg_{\mathcal{P}/\mathcal{O}}(\mathcal{D}^{\otimes})
    \end{tikzcd}
    \]
  \end{enumerate}
\end{propn}
\begin{proof}
  To prove that $F$ is a cocartesian fibration we use the criterion of (the dual of) \cite[Lemma A.1.8]{cois}, which requires us
  to check that
  \begin{enumerate}[(a)]
  \item for every $O \in \mathcal{O}$, the functor
    $F_{O} \colon \mathcal{C}^{\otimes}_{O} \to \mathcal{D}^{\otimes}_{O}$ is a
    cocartesian fibration,
  \item for every morphism $\phi \colon O \to O'$ in $\mathcal{O}$,
    in the commutative square
    \[
      \begin{tikzcd}
        \mathcal{C}^{\otimes}_{O} \arrow{r}{\phi_{!}^{\mathcal{C}}} \arrow{d}{F_{O}}&
        \mathcal{C}^{\otimes}_{O'} \arrow{d}{F_{O'}}\\
        \mathcal{D}^{\otimes}_{O} \arrow{r}{\phi_{!}^{\mathcal{D}}} & \mathcal{D}^{\otimes}_{O'},
      \end{tikzcd}
    \]
    the functor $\phi_{!}^{\mathcal{C}}$ takes $F_{O}$-cocartesian
    morphisms to $F_{O'}$-cocartesian morphisms.    
  \end{enumerate}
  Condition (a) is clear, since $F_{O}$ is equivalent to the product
  \[ \prod_{i} F_{O_{i}} \colon \prod_{i} \mathcal{C}_{O_{i}} \to
    \prod_{i} \mathcal{D}_{O_{i}},\] where each $F_{O_{i}}$ is by
  assumption a cocartesian fibration. Moreover, a morphism in
  $\mathcal{C}^{\otimes}_{O}$ is $F_{O}$-cocartesian \IFF{} under the
  equivalence with $\prod_{i} \mathcal{C}^{O_{i}}$ it corresponds to a
  tuple of $F_{O_{i}}$-cocartesian morphisms. If $\phi$ is an inert
  morphism in $\mathcal{O}$, then $\phi_{!}^{\mathcal{C}}$ corresponds
  to a projection to some factors in this product, hence condition (b)
  is immediate for inert morphisms. Using the factorization system we
  can then reduce condition (b) to the case of an active morphism
  $O \actto E$ with $E \in \mathcal{O}^{\el}$, where it holds by
  assumption.  The description of the cocartesian morphisms in the
  proof of \cite[Lemma A.1.8]{cois} also gives (ii).

  To see that $F$ exhibits $\mathcal{C}^{\otimes}$ as
  $\mathcal{D}^{\otimes}$-monoidal, observe that for $O \in
  \mathcal{O}$ the commutative square
  \[
    \begin{tikzcd}
      \mathcal{C}^{\otimes}_{O} \arrow{r}{\sim} \arrow{d}{F_{O}} & \prod_{i}
      \mathcal{C}_{O_{i}} \arrow{d}{F_{O_{i}}} \\
      \mathcal{D}^{\otimes}_{O} \arrow{r}{\sim} & \prod_{i} \mathcal{D}_{O_{i}}
    \end{tikzcd}
  \]
  is cartesian, since the horizontal maps are equivalences. For any
  $O(D_{i}) \in \mathcal{D}^{\otimes}_{O}$ we therefore have an
  equivalence on fibres
  \[ \mathcal{C}^{\otimes}_{O(D_{i})} \isoto \prod_{i}
    \mathcal{C}_{D_{i}},\]
  as required.

  To prove (iii), observe that we have a pullback square
  \[
    \begin{tikzcd}
      \Fun_{/\mathcal{D}^{\otimes}}(\mathcal{P},
      \mathcal{C}^{\otimes}) \arrow{d} \arrow{r} &
      \Fun_{/\mathcal{O}}(\mathcal{P}, \mathcal{C}^{\otimes})
      \arrow{d} \\
      \{f\} \arrow{r} & \Fun_{/\mathcal{O}}(\mathcal{P}, \mathcal{D}^{\otimes}).
    \end{tikzcd}
  \]
  This restricts to the full subcategories of algebras because (ii)
  implies that a functor from $\mathcal{P} \to \mathcal{C}^{\otimes}$
  over $\mathcal{D}^{\otimes}$ is an algebra \IFF{} the underlying
  functor over $\mathcal{O}$ is an algebra (since every inert morphism
  in $\mathcal{C}^{\otimes}$ is $F$-cocartesian over an inert
  morphisms in $\mathcal{D}^{\otimes}$).
\end{proof}

\begin{cor}\label{cor:cartcoc}
  Suppose $\pi \colon \mathcal{E} \to \mathcal{B}$ is a cocartesian
  fibration such that
  \begin{enumerate}[(1)]
  \item   $\mathcal{E}$ and $\mathcal{B}$ have finite
    products,
  \item $\pi$ preserves these,
  \item if $e_{i} \to e_{i}'$ ($i = 1,2$) are $\pi$-cocartesian morphisms in
    $\mathcal{E}$ then $e_{1} \times e_{2} \to e_{1}' \times e_{2}'$
    is again $\pi$-cocartesian.
  \end{enumerate}
  Then the induced symmetric monoidal functor
  $\pi^{\times} \colon \mathcal{E}^{\times} \to \mathcal{B}^{\times}$
  is a cocartesian fibration that exhibits $\mathcal{E}^{\times}$ as
  $\mathcal{B}^{\times}$-monoidal. \qed
\end{cor}

\begin{lemma}\label{lem:prodrfib}
  Suppose
  $\mathcal{E} \to \mathcal{A}$ and $\mathcal{F} \to \mathcal{B}$ are
  right fibrations corresponding to functors
  $E \colon \mathcal{A}^{\op} \to \mathcal{S}$,
  $F \colon \mathcal{B}^{\op} \to \mathcal{S}$. Then the right
  fibration
  $\mathcal{E} \times \mathcal{F} \to \mathcal{A} \times \mathcal{B}$
  corresponds to the functor
  $E \times F \colon \mathcal{A}^{\op} \times \mathcal{B}^{\op} \to
  \mathcal{S}$ (taking $(a,b)$ to $E(a) \times F(b)$).
\end{lemma}
\begin{proof}
  Consider the right fibrations
  \[\mathcal{A} \times \mathcal{E}, \mathcal{F} \times \mathcal{B} \to
    \mathcal{A} \times \mathcal{B}.\]
  These are the pullbacks of $\mathcal{E}$ and $\mathcal{F}$ along the
  projections from $\mathcal{A} \times \mathcal{B}$ to $\mathcal{A}$
  and $\mathcal{B}$, respectively. Since pullbacks of right fibrations
  correspond to compositions of functors, the associated functors are therefore
  \[ \mathcal{A}^{\op} \times \mathcal{B}^{\op} \to \mathcal{A}^{\op}
    \xto{E} \mathcal{S},\]
  \[ \mathcal{A}^{\op} \times \mathcal{B}^{\op} \to \mathcal{B}^{\op}
    \xto{F} \mathcal{S},\]
  respectively. Now we can identify $\mathcal{E} \times \mathcal{F}$
  with $(\mathcal{E}\times \mathcal{B}) \times_{\mathcal{A} \times
    \mathcal{B}} (\mathcal{A} \times \mathcal{F})$, which is a product
  of right fibrations over $\mathcal{A} \times \mathcal{B}$. Since
  straightening preserves limits, the corresponding functor is the
  product $E \times F$ as required.
\end{proof}

\begin{cor}\label{cor:RFibtimescoc}
  The functors
  \[ \ev^{\times}_{1} \colon \RSl^{\times}, \RFib^{\times} \to
    \CatI^{\times} \]
  are cocartesian fibrations that exhibit their domains as
  $\CatI^{\times}$-monoidal \icats{}. The inclusion $\RSl^{\times}
  \hookrightarrow \RFib^{\times}$ is $\CatI^{\times}$-monoidal (\ie{}
  preserves cocartesian morphisms over $\CatI^{\times}$).
\end{cor}
\begin{proof}
  The only non-obvious condition in \cref{cor:cartcoc} is that
  products of cocartesian morphisms in $\RFib$ are again cocartesian
  (since the cocartesian morphisms in $\RSl$ are inherited from
  $\RFib$, it is enough to consider the case of $\RFib$).

  Suppose 
  $\mathcal{E} \to \mathcal{A}$ and $\mathcal{F} \to \mathcal{B}$ are
  right fibrations corresponding to functors
  $E \colon \mathcal{A}^{\op} \to \mathcal{S}$,
  $F \colon \mathcal{B}^{\op} \to \mathcal{S}$. Then the product
  $\mathcal{E} \times \mathcal{F} \to \mathcal{A} \times \mathcal{B}$
  corresponds to $E \times F$ by \cref{lem:prodrfib}.
  Given $\alpha \colon \mathcal{A} \to \mathcal{A}'$ and
  $\beta \colon \mathcal{B} \to \mathcal{B}'$, the cocartesian
  pushforward
  $(\alpha \times \beta)_{!}(\mathcal{E} \times \mathcal{F})$
  corresponds to the left Kan extension of $E \times F$ along
  $(\alpha \times \beta)^{\op}$, given by
  \[ (x,y) \in (\mathcal{A}' \times \mathcal{B}')^{\op} \mapsto
    \colim_{(a,b) \in (\mathcal{A}_{x/} \times
      \mathcal{B}_{y}/)^{\op}} E(a) \times F(b).\]
  Since the product in $\mathcal{S}$ commutes with colimits in each
  variable, this is equivalent to
  \[ (x,y) \in (\mathcal{A}' \times \mathcal{B}')^{\op} \mapsto
    \left(\colim_{a \in (\mathcal{A}_{x/})^{\op}} E(a) \right) \times
    \left( \colim_{b \in (\mathcal{B}_{y/})^{\op}} F(b) \right)
  \]
  which by \cref{lem:prodrfib} is the functor corresponding to $\alpha_{!}\mathcal{E} \times
  \beta_{!}\mathcal{F}$, as required. 
  
  For the last claim, note that cocartesian
  morphisms in $\RSl$ are inherited from $\RFib$.
  The description of cocartesian morphisms in \cref{propn:moncocart}(ii)
  therefore shows that the 
  inclusion $\RSl^{\times} \hookrightarrow \RFib^{\times}$ preserves
  cocartesian morphisms over $\CatI^{\times}$.
\end{proof}

\begin{remark}\label{rmk:RFibxcoc}
  Using \cref{propn:moncocart}(ii) we can describe the cocartesian
  morphisms in $\RFib^{\times}$ as follows: Given a morphism
  $(\mathcal{C}_{1},\ldots,\mathcal{C}_{n}) \to \mathcal{D}$ in
  $\CatI^{\times}$, which corresponds to a functor
  $\Phi \colon \mathcal{C}_{1} \times \cdots \times \mathcal{C}_{n}
  \to \mathcal{D}$, and an object of $\RFib^{\times}$ over
  $(\mathcal{C}_{1},\ldots,\mathcal{C}_{n})$, which we can identify
  with a family of presheaves
  $F_{i} \colon \mathcal{C}_{i}^{\op} \to \mathcal{S}$,
  ($i = 1,\ldots,n$), the cocartesian morphism over $\Phi$ takes this to the left Kan extension
  $\Phi_{!}(\prod_{i}F_{i}) \colon \mathcal{D}^{\op} \to \mathcal{S}$
  along $\Phi^{\op}$ of the product
  \[ \prod_{i} F_{i} \colon \prod_{i} \mathcal{C}_{i}^{\op}
    \xto{(F_{i})} \mathcal{S}^{\times n} \xto{\times} \mathcal{S}.\]
\end{remark}

\begin{propn}\label{propn:RSlunivmon}
  Let $\mathcal{O}$ be a cartesian pattern, and suppose
  $\pi \colon \mathcal{C}^{\otimes} \to \mathcal{O}$ is an
  $\mathcal{O}$-monoidal \icat{}, corresponding to an
  $\mathcal{O}$-monoid $M$ in $\CatI$. Let $A \colon \mathcal{O} \to
  \CatI^{\times}$ be the corresponding $\mathcal{O}$-algebra under the
  equivalence of \cref{propn:moniscartalg}.
  Then there is a pullback square
  \[
    \begin{tikzcd}
      \mathcal{C}^{\otimes} \arrow{r} \arrow{d} & \RSl^{\times}
      \arrow{d}{\ev_{1}^{\times}} \\
      \mathcal{O} \arrow{r}{A} & \CatI^{\times}.
    \end{tikzcd}
  \]
\end{propn}

\begin{remark}
  We can interpret this as exhibiting the pullback $\mathcal{O} \times_{\xF_{*}}
  \RSl^{\times} \to \mathcal{O} \times_{\xF_{*}} \CatI^{\times}$ as
  the universal cocartesian fibration of $\mathcal{O}$-monoidal
  \icats{}.
\end{remark}

\begin{proof}[Proof of \cref{propn:RSlunivmon}]
  By \cref{cor:RSluniv} we have a pullback square
  \[
    \begin{tikzcd}
      \xcc^\otimes \arrow[r, "\overline{M}"] \arrow[d] & \RSl\arrow[d]\\
      \xxO \arrow[r, "M"] & \CatI,
    \end{tikzcd}\]
  where $M$ is by assumption an $\mathcal{O}$-monoid. The functor
  $\overline{M}$ takes an object $X \simeq O(X_{1},\ldots,X_{n}) \in
  \mathcal{C}^{\otimes}$ to the slice
  $(\mathcal{C}^{\otimes}_{O})_{/X}$. But the equivalence
  $\mathcal{C}^{\otimes}_{O} \simeq
  \prod_{i=1}^{n}\mathcal{C}_{O_{i}}$ induces an equivalence
  \[ (\mathcal{C}^{\otimes}_{O})_{/X} \simeq \prod_{i=1}^{n}
    \mathcal{C}_{O_{i}/X_{i}}. \]
  This shows that the functor $\overline{M}$ is a
  $\mathcal{C}^{\otimes}$-monoid. By the naturality of the equivalence
  between monoids and algebras, as discussed in
  \cref{rmk:monalgnat}, our pullback square therefore corresponds to a
  commutative square
\[
\begin{tikzcd}
\xcc^\otimes \arrow[r, "\overline{A}"] \arrow[d] & \RSl^\times\arrow[d]\\
\xxO \arrow[r, "A"] & \CatI^\times,
\end{tikzcd}\]
where $\overline{A}$ is the $\mathcal{C}^{\otimes}$-algebra
corresponding to $\overline{M}$. It then remains to verify that this
square is cartesian.

Here the vertical maps are cocartesian
fibrations, and we first observe that $\overline{A}$ preserves cocartesian
morphisms. Since $\overline{A}$ is an algebra it suffices to check
this in the
case of a cocartesian morphism $\overline{\phi} \colon O(C_{1}, \ldots, C_{n}) \to C'$ over an
active morphism $\phi \colon O \actto E$ with $E \in \mathcal{O}^{\el}$.
Here $C' \simeq M(\phi)(O(C_{1},\ldots,C_{n}))$, and
$\overline{A}(\overline{\phi})$ is by construction the morphism
\[ (\mathcal{C}_{O_{1}/C_{1}},\ldots,\mathcal{C}_{O_{n}/C_{n}}) \to
  \mathcal{C}_{E/C'} \]
corresponding to the functor
\[ \overline{M}(\overline{\phi}) \colon \prod_{i}
  \mathcal{C}_{O_{i}/C_{i}} \simeq
  \mathcal{C}^{\otimes}_{O/O(C_{1},\ldots,C_{n})} \to
  \mathcal{C}_{E/C'}.\]
By definition of $\overline{M}$ as a pullback it preserves cocartesian
morphisms, so this is a cocartesian
morphism in $\RSl$. The description of cocartesian morphisms in
$\RSl^{\times}$ in \cref{rmk:RFibxcoc} now implies that
$A(\overline{\phi})$ is therefore cocartesian in $\RSl^{\times}$, as required.

To prove that the square is cartesian it now suffices see that it
induces equivalences on fibres over all $O \in \mathcal{O}$. But since
$\overline{A}$ preserves cocartesian morphisms, we have for $O \in
\mathcal{O}$ a commutative square
\[
  \begin{tikzcd}
    \mathcal{C}^{\otimes}_{O} \arrow{r} \arrow{d}{\sim} &
    \RSl^{\times}_{A(O)} \arrow{d}{\sim} \\
    \prod_{i}\mathcal{C}_{O_{i}} \arrow{r} & \prod_{i} \RSl_{M(O_{i})}
  \end{tikzcd}
  \]
  where the vertical maps are equivalences since
  $\mathcal{C}^{\otimes}$ is $\mathcal{O}$-monoidal and
  $\RSl^{\times}$ is $\CatI^{\times}$-monoidal, while the bottom
  horizontal map is an equivalence since $\mathcal{C}^{\otimes}$ is
  pulled back from $\RSl$ along $M$.
\end{proof}

\begin{definition}\label{def Day}
  Suppose $\mathcal{C}^{\otimes}$ is an $\mathcal{O}$-monoidal
  \icat{}. By Proposition~\ref{propn:moniscartalg} this corresponds to
  an $\mathcal{O}$-algebra $C \colon \mathcal{O} \to
  \CatI^{\times}$. 
  The \emph{(contravariant) Day convolution} of
  $\mathcal{C}^{\otimes}$ is the
  $\mathcal{O}$-monoidal \icat{} given by the pullback
  \[
    \begin{tikzcd}
      \PSh_{\mathcal{O}}(\mathcal{C})^{\otimes} \arrow{r}
      \arrow{d} & \RFib^{\times} \arrow{d} \\
      \mathcal{O} \arrow{r}{C} & \CatI^{\times}.
    \end{tikzcd}
  \]
\end{definition}

\begin{remark}
  Using \cref{rmk:RFibxcoc} we can describe the cocartesian morphisms 
  in $\PSh_{\mathcal{O}}(\mathcal{C})^{\otimes}$: given $\phi \colon O
  \to E$ in $\mathcal{O}$ active with $E \in \mathcal{O}^{\el}$ and
  $F_{i} \in \PSh(\mathcal{C}_{O_{i}})$, we take the left Kan extension
  along
  \[\phi^{\op}_{!} \colon \prod \mathcal{C}_{O_{i}}^{\op} \simeq
    (\mathcal{C}^{\otimes}_{O})^{\op} \to \mathcal{C}_{E}^{\op} \]
  of the product
  \[\prod_{i}F_{i} \colon \prod_{i}\mathcal{C}_{O_{i}}^{\op} \xto{(F_{i})}
    \prod_{i}\mathcal{S} \xto{\times} \mathcal{S}.\]
\end{remark}

\begin{proposition}\label{prop:Day}
  The Yoneda embedding gives a natural $\mathcal{O}$-monoidal functor
  $\mathcal{C}^{\otimes} \to
  \PSh_{\mathcal{O}}(\mathcal{C})^{\otimes}$.
\end{proposition}
\begin{proof}
  By \cref{cor:RFibtimescoc} the inclusion $\RSl\hookrightarrow \RFib$
  induces a $\CatI^{\times}$-monoidal functor $\RSl^{\times}
  \hookrightarrow \RFib^{\times}$. 
  Pulling this back along the $\mathcal{O}$-algebra $C \colon
  \mathcal{O} \to \CatI^{\times}$ corresponding to
  $\mathcal{C}^{\otimes}$ we get using \cref{propn:RSlunivmon} an
  $\mathcal{O}$-monoidal functor $\mathcal{C}^{\otimes} \to 
  \PSh_{\mathcal{O}}(\mathcal{C})^{\otimes}$. Over $E \in
  \mathcal{O}^{\el}$ it follows from \cref{cor:RSlYoneda} that this
  functor is given by the Yoneda embedding $\mathcal{C}_{E}
  \hookrightarrow \PSh(\mathcal{C}_{E})$.
\end{proof}

\begin{notation}
  Suppose $\mathcal{C}^{\otimes} \to \mathcal{O}$ is an
  $\mathcal{O}$-monoidal \icat{}, corresponding to an
  $\mathcal{O}$-monoid $M \colon \mathcal{O} \to \CatI$. Since
  $(\blank)^{\op}$ is an automorphism of $\CatI$, the composite 
  \[ \mathcal{O} \xto{M} \CatI \xto{(\blank)^{\op}} \CatI \]
  is also an $\mathcal{O}$-monoid. We write $\mathcal{C}^{\op,\otimes}
  \to \mathcal{O}$ for the corresponding cocartesian fibration. (Note
  that if we write $\mathcal{C}_{\otimes} \to \mathcal{O}^{\op}$ for
  the \emph{cartesian} fibration for $M$, then
  $\mathcal{C}^{\op,\otimes} \simeq (\mathcal{C}_{\otimes})^{\op}$.)
  To avoid confusion we will avoid the notation $\mathcal{C}^{\op}$
  for the pullback of $\mathcal{C}^{\op,\otimes}$ to
  $\mathcal{O}^{\el}$ (since this is not the opposite \icat{} of
  $\mathcal{C}$, but the fibrewise opposite) and instead write
  $\mathcal{C}^{\op,\otimes}_{/\el}$.
\end{notation}

We can now prove the universal property for mapping into the Day convolution:
\begin{propn}\label{propn:AlgMonPsh}
  Suppose $\xcc^\otimes\to \xxO$ is an $\xxO$-monoidal $\infty$-category, then there is an equivalence
  \[\xAlg_{\xxO}(\name{P}_{\xxO}(\xcc)^\otimes)\simeq \xMon_{\xcc^{\op,\otimes}}(\xS).\]
\end{propn}
\begin{proof}
  Let $M\colon \mathcal{O} \to \CatI$ be the monoid corresponding to
  $\mathcal{C}^{\otimes}$, and let $A \colon \xxO\to \CatI^{\times}$
  be the corresponding algebra. Using the definition of
  $\PSh_{\mathcal{O}}(\mathcal{C})^{\otimes}$ as a pullback we then have natural equivalences
    \begin{align*}
      \xAlg_{\xxO}(\PSh_{\mathcal{O}}(\xcc)^\otimes)& \simeq \Alg_{\xxO/\CatI^{\times}}(\RFib^\times)
      && \text{(by pullback)} \\
                                                    & \simeq
                                                      \{A\}\times_{\Alg_{\xxO/\xF_{*}}(\CatI^{\times})}\Alg_{\xxO/\xF_{*}}(\RFib^\times)
      && \text{(by \cref{propn:moncocart}(iii))}     \\
                                                    & \simeq
                                                      \{M\}\times_{\Mon_{\xxO}(\CatI)}\Mon_{\xxO}(\RFib),
      && \text{(by \cref{propn:moniscartalg})}
    \end{align*}
  where the right-hand side is the full subcategory of
  $\Fun_{/\CatI}(\mathcal{O}, \RFib)$ spanned by the monoids. Since
  the cartesian fibration $\ev_{1} \colon \RFib \to \CatI$ corresponds
  to the functor $\Fun((\blank)^{\op}, \mathcal{S})$ by 
  \cite[Proposition 7.3]{freepres},
  we have an equivalence
  \[ \Fun_{/\CatI}(\xxO, \RFib) \simeq \Fun(\xxO\times_{\CatI} \mathcal{E},\xS),\] where 
  $\mathcal{E} \to \CatI$ is the cocartesian fibration
  corresponding to the functor $\op \colon \CatI \to \CatI$.
  By definition, the pullback $\xxO\times_{\CatI} \mathcal{E}$
  is precisely $\mathcal{C}^{\op,\otimes} \to \mathcal{O}$.
  We have therefore identified
  $\xAlg_{\xxO}(\PSh_{\mathcal{O}}(\mathcal{C})^\otimes)$ with a full subcategory
  of $\Fun(\mathcal{C}^{\op,\otimes}, \mathcal{S})$, and we need to
  check that this is precisely the full subcategory of
  $\mathcal{C}^{\op,\otimes}$-monoids.

  Under the equivalence of   \cite[Proposition 7.3]{freepres}, a
  functor $\phi \colon \mathcal{C}^{\op,\otimes} \to \mathcal{S}$ corresponds to
  the functor $\Phi \colon \mathcal{O} \to \RFib$ that takes $O \in \mathcal{O}$
  to the right fibration for the presheaf
  $\phi|_{\mathcal{C}^{\otimes}_{O}} \colon
  (\mathcal{C}^{\otimes}_{O})^{\op} \to \mathcal{S}$. We then observe
  that this gives an $\mathcal{O}$-monoid in $\RFib$ precisely when
  $\phi$ is a $\mathcal{C}^{\op,\otimes}$-monoid, using the
  commutative squares
  \[
    \begin{tikzcd}
      \Phi(O) \arrow{r} \arrow{d} & \prod_{i} \Phi(O_{i}) \arrow{d} \\
      \mathcal{C}^{\otimes}_{O} \arrow{r}{\sim} & \prod_{i}\mathcal{C}_{O_{i}}.
    \end{tikzcd}
  \]
  Here the vertical maps are right fibrations, so that the top
  horizontal map is an equivalence \IFF{} the square is cartesian,
  which is equivalent to the map on fibres being an equivalence for every
  $O(C_{i}) \in \mathcal{C}^{\otimes}_{O}$. The map on fibres we can
  identify with
  \[ \phi(O(C_{i})) \to \prod_{i} \phi(C_{i}),\]
  so all of these are equivalences precisely when $\phi$ is a $\mathcal{C}^{\op,\otimes}$-monoid.
\end{proof}

\begin{cor}
  Let $\mathcal{C}^{\otimes}$ be a small $\mathcal{O}$-monoidal
  \icat{}. We say a $\mathcal{C}^{\op,\otimes}$-monoid
  \[ M \colon \mathcal{C}^{\op,\otimes} \to \mathcal{S} \]
  is \emph{representable} if for every $E \in \mathcal{O}^{\el}$, the
  restriction
  \[ M|_{\mathcal{C}_{E}^{\op}} \colon \mathcal{C}_{E}^{\op} \simeq
    (\mathcal{C}^{\op,\otimes})_{E} \to \mathcal{S} \]
  is a representable presheaf. There is a natural equivalence
  \[ \Alg_{\mathcal{O}}(\mathcal{C}) \simeq \Mon^{\name{rep}}_{\mathcal{C}^{\op,\otimes}}(\mathcal{S}).\]
\end{cor}
\begin{proof}
  The $\mathcal{O}$-monoidal inclusion $\mathcal{C}^{\otimes}
  \hookrightarrow \PSh_{\mathcal{O}}(\mathcal{C})^{\otimes}$
  identifies $\Alg_{\mathcal{O}}(\mathcal{C})$ with the full
  subcategory of $\Alg_{\mathcal{O}}(\PSh_{\mathcal{O}}(\mathcal{C}))$
  spanned by the $\mathcal{O}$-algebras $A$ such that $A(E) \in
  \PSh(\mathcal{C}_{E})$ is representable for every $E \in
  \mathcal{O}^{\el}$. This full subcategory is identified with
  $\Mon^{\name{rep}}_{\mathcal{C}^{\op,\otimes}}(\mathcal{S})$ under
  the equivalence of \cref{propn:AlgMonPsh}.
\end{proof}

\begin{lemma}\label{lem:pbDC}
  Suppose $f \colon \mathcal{O} \to \mathcal{P}$ is a morphism of
  cartesian patterns and $\mathcal{C}^{\otimes} \to \mathcal{P}$ is a
  $\mathcal{P}$-monoidal \icat{}. Then there is a natural equivalence
  \[ \PSh_{\mathcal{O}}(f^{*}\mathcal{C})^{\otimes} \simeq
    f^{*}\PSh_{\mathcal{P}}(\mathcal{C})^{\otimes}\] of
  $\mathcal{O}$-monoidal \icats{}.
\end{lemma}
\begin{proof}
  By definition we have a commutative diagram
  \[
    \begin{tikzcd}
      f^{*}\PSh_{\mathcal{P}}(\mathcal{C})^{\otimes} \arrow{d}
      \arrow{r} & \PSh_{\mathcal{P}}(\mathcal{C})^{\otimes} \arrow{r} \arrow{d}
      & \RFib^{\times}  \arrow{d} \\
      \mathcal{O} \arrow{r}{f} & \mathcal{P} \arrow{r}{C} & \CatI^{\times},
    \end{tikzcd}
  \]
  where $C$ is the algebra corresponding to the $\mathcal{P}$-monoidal
  \icat{} $\mathcal{C}^{\otimes}$, and both squares are cartesian. 
  Then the composite square is also cartesian. On the other hand,
  by \cref{rmk:monalgnat} the composite $C \circ f$ is the
  $\mathcal{O}$-algebra corresponding to $f^{*}\mathcal{C}^{\otimes}$,
  and so the pullback of $\RFib^{\times}$ along $C \circ f$ is by
  definition $\PSh_{\mathcal{O}}(f^{*}\mathcal{C})^{\otimes}$.
\end{proof}

\begin{corollary}\label{cor:f*CopPSh}
  Let $f\colon \mathcal{O} \to \mathcal{P}$ be a morphism of cartesian
  pattern. We have natural equivalences of $\infty$-categories 
  \[
    \Alg_{\mathcal{O}/\mathcal{P}}(\PSh_{\mathcal{P}}(\mathcal{C})^{\otimes})
    \simeq \Mon_{f^{*}\mathcal{C}^{\op,\otimes}}(\xS). \]
\end{corollary}
\begin{proof}
  By \cref{lem:pbDC} pulling back along $f$ gives natural equivalences
  \[ \xAlg_{\mathcal{O}/\mathcal{P}}(\PSh_{\mathcal{P}}(\mathcal{C})^{\otimes})
    \simeq
    \xAlg_{\mathcal{O}}(f^{*}\PSh_{\mathcal{P}}(\mathcal{C})^{\otimes})
    \simeq
    \xAlg_{\mathcal{O}}(\PSh_{\mathcal{O}}(f^{*}\mathcal{C})^{\otimes}).\]
  Since we also have $(f^{*}\mathcal{C})^{\op,\otimes} \simeq
  f^{*}(\mathcal{C}^{\op,\otimes})$, the result now follows from
  \cref{propn:AlgMonPsh}.
\end{proof}

\begin{remark}
  As a special case, for the inclusion $\mathcal{O}^{\xint} \to
  \mathcal{O}$ we get a commutative square of equivalences
  \[
    \begin{tikzcd}
      \Alg_{\mathcal{O}^{\xint}/\mathcal{O}}(\PSh_{\mathcal{O}}(\mathcal{C})^{\otimes})
      \arrow{r}{\sim} \arrow{d}{\sim} &
      \Mon_{\mathcal{C}^{\op,\otimes}_{/\xint}}(\mathcal{S}) \arrow{d}{\sim}
      \\
      \Fun_{/\mathcal{O}^{\el}}(\mathcal{O}^{\el},
      \PSh_{\mathcal{O}}(\mathcal{C})) \arrow{r}{\sim} &
      \Fun(\mathcal{C}^{\op,\otimes}_{/\el}, \mathcal{S})
    \end{tikzcd}
  \]
  where the top horizontal map is an equivalence by
  \cref{cor:f*CopPSh}, the right vertical map by \cref{OintmonRKE}, the
  left vertical map by \cref{lem:AlgOint}, and the bottom horizontal
  map by a trivial version of \cref{cor:f*CopPSh} or by
  \cite[Proposition 7.3]{freepres}.
\end{remark}

We will next prove that every $\mathcal{O}$-monoidal functor from a
small $\mathcal{O}$-monoidal \icat{} extends to the Day convolution,
provided the target is compatible with small colimits in the following
sense:
\begin{defn}
  If $\mathcal{K}$ is some class of \icats{}, we say that an
  $\mathcal{O}$-monoidal \icat{} $\mathcal{C}^{\otimes} \to
  \mathcal{O}$ is \emph{compatible with $\mathcal{K}$-colimits} if the
  \icats{} $\mathcal{C}_{E}$ for $E \in \mathcal{O}^{\el}$ have
  $\mathcal{K}$-shaped colimits, and for every active map $\phi \colon
  O \to E$ in $\mathcal{O}$ with $E \in \mathcal{O}^{\el}$, the
  functor
  \[ \phi_{!} \colon \prod_{i} \mathcal{C}_{O_{i}} \simeq
    \mathcal{C}^{\otimes}_{O} \to \mathcal{C}_{E} \]
  preserves $\mathcal{K}$-shaped colimits in each variable. If
  $\mathcal{K}$ is the class of all small \icats{}, we say that
  $\mathcal{C}^{\otimes}$ is \emph{compatible with (small) colimits} or
  \emph{cocontinuously $\mathcal{O}$-monoidal}.
\end{defn}

\begin{remark}
  For any small $\mathcal{O}$-monoidal \icat{}
  $\mathcal{C}^{\otimes}$, the Day convolution
  $\PSh_{\mathcal{O}}(\mathcal{C})^{\otimes}$ is compatible with small
  colimits. This is easy to see using the description of cocartesian
  morphisms in terms of products and left Kan extensions, since
  products in $\mathcal{S}$ preserve colimits in each variable.
\end{remark}

\begin{defn}
  Suppose $\mathcal{C}^{\otimes}$ and $\mathcal{D}^{\otimes}$ are
  $\mathcal{O}$-monoidal \icats{} that are compatible with small
  colimits. We say that an $\mathcal{O}$-monoidal functor
  $F \colon \mathcal{C}^{\otimes} \to \mathcal{D}^{\otimes}$ is
  \emph{cocontinuous} if the underlying functors
  $F_{E} \colon \mathcal{C}_{E} \to \mathcal{D}_{E}$ preserve small
  colimits for $E \in \mathcal{O}^{\el}$.
\end{defn}

\begin{propn}\label{propn:extPC}
  Let $\mathcal{C}^{\otimes}$ be a small $\mathcal{O}$-monoidal
  \icat{} and $\mathcal{V}^{\otimes}$ be an $\mathcal{O}$-monoidal
  \icat{} compatible with small colimits. Then every
  $\mathcal{O}$-monoidal functor
  $F\colon \xcc^\otimes \to \mathcal{V}^\otimes$ induces a cocontinuous
  $\mathcal{O}$-monoidal functor
  \[ F_{!} \colon \PSh_{\mathcal{O}}(\mathcal{C})^{\otimes} \to
    \mathcal{V}^{\otimes} \]
  such that the composite
  \[ \mathcal{C}^{\otimes} \xto{y^{\otimes}}
    \PSh_{\mathcal{O}}(\mathcal{C})^{\otimes} \xto{F_{!}}
    \mathcal{V}^{\otimes} \]
  is equivalent to $F$, and $F_{!,E} \colon \PSh(\mathcal{C}_{E}) \to
  \mathcal{V}_{E}$ for $E \in \mathcal{O}^{\el}$ is the unique
  cocontinuous functor extending $F_{E}$ along the Yoneda embedding of
  $\mathcal{C}_{E}$. If in addition $\mathcal{V}^{\otimes}$ is locally
  small, then $F_{!}$ has a lax $\mathcal{O}$-monoidal right adjoint
  $F^{*}$, given over $E \in \mathcal{O}^{\el}$ by the restricted
  Yoneda embedding \[F_{E}^{*} \colon
  \mathcal{V}_{E} \to \PSh(\mathcal{V}_{E}) \to \PSh(\mathcal{C}_{E}).\]
\end{propn}
\begin{proof}
  The $\mathcal{O}$-monoidal functor
  $F \colon \mathcal{C}^\otimes \to \mathcal{V}^\otimes$ corresponds
  under the equivalence of \cref{propn:moniscartalg} to a morphism of
  $\mathcal{O}$-algebras in $\LCatI^{\times}$, \ie{} a natural
  transformation
  $\phi \colon \mathcal{O} \times \Delta^{1} \to \LCatI^{\times}$ over
  $\xF_{*}$. We can pull back the cocartesian fibration $\LRFib^{\times}\to \LCatI^{\times}$ along this to obtain a
  cocartesian fibration
  $\mathcal{E}_{+} \to \mathcal{O} \times \Delta^{1}$. Let
  $\mathcal{E}$ be the full subcategory of $\mathcal{E}_{+}$
  containing those objects over $0$ that lie in
  $\PSh_{\mathcal{O}}(\mathcal{C})^{\otimes}$ and those objects over
  $1$ that lie in $\mathcal{V}^{\otimes}$. We claim that the
  restricted projection
  $\mathcal{E} \to \mathcal{O} \times \Delta^{1}$ is again cocartesian
  (but the inclusion into $\mathcal{E}_{+}$ does not preserve all
  cocartesian morphisms).  Given an object
  $\Phi \in \PSh(\mathcal{C}_{E})$, which we can write as a
  small colimit $\colim_{x \in \mathcal{I}}y(\phi(x))$ of representable
  presheaves, the cocartesian morphism in $\mathcal{E}_{+}$ over
  $(\id_{E},0\to 1)$ takes $\Phi \in \PSh(\mathcal{C}_{E})$ to the colimit
  $\colim_{x \in \mathcal{I}} y(F\phi(x))$ computed in (large) presheaves
  on $\mathcal{V}_{E}$. Considering maps in $\LPSh(\mathcal{V}_{E})$
  from this to representable presheaves, we see there is an initial
  one, given by the same colimit computed in the \icat{}
  $\mathcal{V}_{E}$. This gives a cocartesian morphism in
  $\mathcal{E}$ over $(\id_{E}, 0 \to 1)$, and combining this
  observation with the compatibility of $\mathcal{V}^{\otimes}$ with
  small colimits we see easily that $\mathcal{E}$ is a cocartesian
  fibration. Unstraightening this over $\Delta^{1}$ we
  get a commutative diagram
  \[
    \begin{tikzcd}
      \mathcal{P}_{\mathcal{O}}(\mathcal{C})^{\otimes} \arrow{dr}
      \arrow{rr}{F_{!}} & &
      \mathcal{V}^{\otimes} \arrow{dl} \\
       & \mathcal{O}
    \end{tikzcd}
  \]
  where $F_{!}$ preserves cocartesian morphisms, as required. To get
  the right adjoint we apply the criterion of \cite[Lemma
  A.1.10]{cois} to see that the composite
  $\mathcal{E} \to \mathcal{O} \times \Delta^{1} \to \Delta^{1}$ is
  also a cartesian fibration, and the cartesian morphisms lie over
  equivalences on $\mathcal{O}$. Since we already know that the
  composition $\mathcal{E} \to \mathcal{O}$ is a cocartesian
  fibration, the only thing to check is that for every $O \in
  \mathcal{O}$ the cocartesian fibration $\mathcal{E}_{O} \to \Delta^{1}$ is a
  cartesian fibration. By identifying this cocartesian fibration with the functor
  $\PSh_{\mathcal{O}}(\mathcal{C})^{\otimes}_{O} \to
  \mathcal{V}^{\otimes}_{O}$, it suffices to show that it has a right adjoint. Since this functor is equivalent to $\prod_{i} \PSh(\mathcal{C}_{O_{i}})
  \to \prod_{i} \mathcal{V}_{O_{i}}$ and products of adjoints are adjoints, we only need to see that each component $F_{O_{i},!}
  \colon \PSh(\mathcal{C}_{O_{i}}) \to \mathcal{V}_{O_{i}}$ has a
  right adjoint. If $\mathcal{V}_{O_{i}}$ is locally
  small, then $F_{O_{i},!}$ has a right adjoint $F^*_{O_{i}}\colon \xV_{O_{i}}\to \PSh(\xV_{O_{i}})\to \PSh(\xcc_{O_{i}})$ given by the composite of the Yoneda embedding and the precomposition with $F_{O_{i}}$. This gives a commutative diagram
  \[
    \begin{tikzcd}
      \mathcal{V}^{\otimes} \arrow{rr}{F^{*}} \arrow{dr} & &
      \PSh_{\mathcal{O}}(\mathcal{C})^{\otimes} \arrow{dl} \\
       & \mathcal{O},
    \end{tikzcd}
  \]
  and it only remains to show that $F^{*}$ preserves inert
  morphisms. Given an inert morphism $\phi \colon O \intto O'$
  we want to show that the canonical natural transformation
  \[ F_{O'}^{*}\phi_{!}\to \phi_{!}F_{O}^{*}, \]
  which arises as the mate transformation of the square
  \[
    \begin{tikzcd}
      \PSh_{\mathcal{O}}(\mathcal{C})^{\otimes}_{O} \arrow{d}[swap]{\phi_{!}}
      \arrow{r}{F_{O,!}} & \mathcal{V}^{\otimes}_{O}
        \arrow{d}{\phi_{!}} \\
        \PSh_{\mathcal{O}}(\mathcal{C})^{\otimes}_{O'} \arrow{r}{F_{O',!}} & \mathcal{V}^{\otimes}_{O'},
    \end{tikzcd}
  \]
  is an equivalence.
  We can identify this with the square
  \[
    \begin{tikzcd}[column sep=large]
      \prod_{i} \PSh(\mathcal{C}_{O_{i}}) \arrow{d}
      \arrow{r}{\prod_{i}F_{O_{i},!}} & \prod_{i} \mathcal{V}_{O_{i}}
        \arrow{d} \\
        \prod_{j} \PSh(\mathcal{C}_{O'_{j}})
        \arrow{r}{\prod_{j}F_{O'_{j},!}}  & \prod_{j}\mathcal{V}_{O'_{j}},
    \end{tikzcd}
  \]
  where the vertical maps are given by projections to the same subset
  of factors in the product. It is then clear that the mate square
  also commutes, since the right adjoint of $\prod_{i} F_{!,O_{i}}$ is
  the product $\prod_{i} F^{*}_{O_{i}}$.
\end{proof}

\begin{remark}
  The extension of $F \colon \mathcal{C}^{\otimes} \to
  \mathcal{V}^{\otimes}$ to
  $\PSh_{\mathcal{O}}(\mathcal{C})^{\otimes}$ is in fact unique. Since
  we do not need this universal property, we will only give a sketch of
  the argument: Given a cocontinuous $\mathcal{O}$-monoidal functor
  $\Phi \colon \PSh_{\mathcal{O}}(\mathcal{C})^{\otimes} \to
  \mathcal{V}^{\otimes}$, we can compose this with the Yoneda embedding
  $\mathcal{C}^{\otimes} \to
  \PSh_{\mathcal{O}}(\mathcal{C})^{\otimes}$; this data we can
  interpret as a 2-simplex in the \icat{} of large $\mathcal{O}$-monoidal
  \icats{}, which corresponds to a 2-simplex of algebras
  \[ \mathcal{O} \times \Delta^{2} \to \LCatI^{\times}.\] We can pull
  back $\LRFib^{\times}$ along this to obtain a cocartesian fibration
  $\mathcal{E}_{+} \to \mathcal{O} \times \Delta^{2}$. Then we
  consider the full subcategory $\mathcal{E}$ whose objects over $0$
  are those in $\PSh_{\mathcal{O}}(\mathcal{C})^{\otimes}$, whose
  objects over $1$ are those in
  $\PSh_{\mathcal{O}}(\mathcal{C})^{\otimes}$ (but now viewed inside
  $\LPSh_{\mathcal{O}}(\PSh_{\mathcal{O}}(\mathcal{C}))^{\otimes}$),
  and whose objects over $2$ are those in $\mathcal{V}^{\otimes}$. As
  before, we can check that
  $\mathcal{E} \to \mathcal{O} \times \Delta^{2}$ is a cocartesian
  fibration; its fibre over $\Delta^{\{0,2\}}$ corresponds to the
  extension $(\Phi|_{\mathcal{C}^{\otimes}})_{!}$, its fibre over
  $\Delta^{\{1,2\}}$ to $\Phi$, and using the universal property for
  mapping out of presheaves its fibre over $\Delta^{\{0,1\}}$
  is an $\mathcal{O}$-monoidal equivalence between two versions of
  $\PSh_{\mathcal{O}}(\mathcal{C})^{\otimes}$.
\end{remark}

\section{$\mathcal{O}$-Monoidal Localizations and Presentability of
  Algebras}\label{sec:monloc}
In this section we first discuss $\mathcal{O}$-monoidal localizations
and then consider \emph{presentably} $\mathcal{O}$-monoidal \icats{},
in the following sense:
\begin{defn}
  We say an $\mathcal{O}$-monoidal \icat{} $\mathcal{V}^{\otimes}$ is
  \emph{presentably $\mathcal{O}$-monoidal} if it is compatible with
  small colimits and the \icats{} $\mathcal{V}_{E}$ for
  $E \in \mathcal{O}^{\el}$ are all presentable. 
\end{defn}
Our main result is that every presentably $\mathcal{O}$-monoidal \icat{}
is an $\mathcal{O}$-monoidal localization of a Day convolution
$\mathcal{O}$-monoidal structure on a small full subcategory.  We apply
this to show that if $\mathcal{V}$ is presentably
$\mathcal{O}$-monoidal, then the \icat{}
$\Alg_{\mathcal{O}}(\mathcal{V})$ is presentable.

We begin by proving a general existence result for $\mathcal{O}$-monoidal
adjunctions, in the following sense:
\begin{defn}
  Consider a commutative triangle
  \[
    \begin{tikzcd}
      \mathcal{C} \arrow{rr}{G} \arrow{dr}[swap]{p} & & \mathcal{D}
      \arrow{dl}{q} \\
       & \mathcal{B}.
    \end{tikzcd}
  \]
  We say that $G$ has a left adjoint \emph{relative to $\mathcal{B}$}
  if $G$ has a left adjoint $F \colon \mathcal{D} \to \mathcal{C}$ and
  the unit map $d \to GFd$ maps to an equivalence in $\mathcal{B}$ for
  all $d \in \mathcal{D}$.
\end{defn}
\begin{remark}
  If $F$ is a left adjoint relative to $\mathcal{B}$ as above, then it
  follows that $pF(d) \simeq qGF(d) \simeq q(d)$, so that $F$ is a
  functor over $\mathcal{B}$. Moreover, the counit map $FGc \to c$
  also lies over an equivalence in $\mathcal{B}$: applying $p$
  is the same as applying $qG$, and the map
  $qGFGc \to qGc$ is an equivalence since the composite
  \[ qGc \to qGFGc \to qGc \] is the identity by one of the adjunction
  equivalences, and the first unit map is an equivalence by
  assumption. Thus $G$ has a left adjoint relative to $\mathcal{B}$
  \IFF{} it has a left adjoint in the $(\infty,2)$-category of
  \icats{} over $\mathcal{B}$.
\end{remark}

\begin{remark}\label{rmk:reladjcocart}
  If $F$ is a left adjoint relative to $\mathcal{B}$ as above, then
  $F$ takes any $q$-cocartesian morphism in $\mathcal{D}$ to a
  $p$-cocartesian morphism in $\mathcal{C}$: If $\delta \colon d \to
  d'$ is a morphism in $\mathcal{D}$, then the relative adjunction data lets us identify
  the following pair of commutative squares:
  \[
    \begin{tikzcd}
    \Map_{\mathcal{C}}(Fd', c) \arrow{r}{(F\delta)^{*}} \arrow{d}{p} &
    \Map_{\mathcal{C}}(Fd, c) \arrow{d}{p} \\
    \Map_{\mathcal{B}}(pFd', pc) \arrow{r}{(pF\delta)^{*}} &
    \Map_{\mathcal{B}}(pFd, pc)      
    \end{tikzcd}
    \qquad
      \begin{tikzcd}
      \Map_{\mathcal{D}}(d', Gc) \arrow{r}{\delta^{*}} \arrow{d}{q} &
      \Map_{\mathcal{D}}(d, Gc) \arrow{d}{q} \\
      \Map_{\mathcal{B}}(qd', qGc) \arrow{r}{(q\delta)^{*}} &
      \Map_{\mathcal{B}}(qd, qGc).
    \end{tikzcd}
    \]
    If $\delta$ is $q$-cocartesian then the second square is cartesian
    for any $c \in \mathcal{C}$, hence so is the first square, which
    says precisely that $F\delta$ is $p$-cocartesian.
\end{remark}

\begin{defn}
  Suppose $G \colon \mathcal{C}^{\otimes} \to \mathcal{D}^{\otimes}$
  is a lax $\mathcal{O}$-monoidal functor. We say that $G$ has an
  $\mathcal{O}$-monoidal left adjoint if it has a left adjoint
  relative to $\mathcal{O}$; this is then automatically
  $\mathcal{O}$-monoidal by \cref{rmk:reladjcocart}.
\end{defn}

\begin{propn}\label{propn:Omonladj}
  Consider a commutative triangle
  \[
    \begin{tikzcd}
      \mathcal{C}^{\otimes} \arrow{dr}[swap]{p} \arrow{rr}{G} & &
      \mathcal{D}^{\otimes} \arrow{dl}{q} \\
      & \mathcal{O},
    \end{tikzcd}
  \]
  where $p$ and $q$ are $\mathcal{O}$-monoidal \icats{} and $G$ is lax
  $\mathcal{O}$-monoidal. Suppose
  \begin{enumerate}[(1)]
  \item   $G_{E} \colon \mathcal{C}_{E} \to
    \mathcal{D}_{E}$ admits a left adjoint $F_{E}$ for all $E \in
    \mathcal{O}^{\el}$,
  \item for every active map $\phi \colon O \to E$
    in $\mathcal{O}$ with $E \in \mathcal{O}^{\el}$, the natural transformation
    \[ F_{E} \circ \phi_{!}^{\mathcal{C}} \to \phi_{!}^{\mathcal{D}}
      \circ \prod_{i} F_{O_{i}} \]
    of functors $\prod_{i} \mathcal{C}_{O_{i}} \to \mathcal{D}_{E}$,
    is an equivalence.
  \end{enumerate}
  Then the functor $G$ admits an $\mathcal{O}$-monoidal left adjoint $F$.
\end{propn}
\begin{proof}
  This is the $\mathcal{O}$-monoidal analogue of
  \cite[Corollary 7.3.2.12]{ha} and follows from the criterion for
  existence of relative left adjoints in \cite[Proposition
  7.3.2.11]{ha}. To apply this we must show the following conditions
  hold:
  \begin{enumerate}[(a)]
  \item For every $O \in \mathcal{O}$, the functor $G_{O}\colon
    \mathcal{C}^{\otimes}_{O} \to \mathcal{D}^{\otimes}_{O}$ admits a
    left adjoint $F_{O}$.
  \item for every morphism $\phi \colon O \to O'$ in $\mathcal{O}$ 
    the natural transfromation
    \[ F_{O'}\circ \phi_{!}^{\mathcal{C}} \to \phi_{!}^{\mathcal{D}}
      \circ F_{O} \]
    is an equivalence.
  \end{enumerate}
  Condition (a) follows from (1) since the functor $G_{O}$ is
  equivalent to the product $\prod_{i} G_{O_{i}} \colon \prod_{i}
  \mathcal{C}_{O_{i}} \to \prod_{i} \mathcal{D}_{O_{i}}$, which has
  left adjoint $\prod_{i} F_{O_{i}}$. Condition (b) is obvious for
  inert maps $\phi$ (since then $\phi_{!}^{\mathcal{C}}$ and
  $\phi_{!}^{\mathcal{D}}$ are both projections unto the same
  collection of factors in a product), so using the factorization
  system it is enough to consider $\phi$ active. But then we can write
  $\phi_{!}^{\mathcal{C}}$ as the product $\prod_{i}
  \phi_{i,!}^{\mathcal{C}}$ where $\phi_{i} \colon O_{i} \actto O'_{i}$
  comes from the factorization
  \[
    \begin{tikzcd}
      O \arrow[active]{r}{\phi} \arrow[inert]{d} & O' \arrow[inert]{d}{\rho_{i}^{O'}} \\
      O_{i} \arrow[active]{r}{\phi_{i}} & O'_{i},
    \end{tikzcd}
  \]
  where we know the natural transformation for $\phi_{i}$ is an
  equivalence by assumption (2). The map in question is therefore a
  product of maps we know are equivalences. 
\end{proof}

\begin{remark}\label{Algadj}
  For any morphism of cartesian patterns $f \colon \mathcal{P} \to
  \mathcal{O}$, a relative adjunction between $\mathcal{O}$-monoidal
  \icats{} as in \cref{propn:Omonladj} induces via composition an adjunction on
  \icats{} of algebras,
  \[ F_{*} : \Alg_{\mathcal{P}/\mathcal{O}}(\mathcal{D})
    \rightleftarrows \Alg_{\mathcal{P}/\mathcal{O}}(\mathcal{C}) :
    G_{*}. \]
  This follows from the 2-functoriality of
  $\Alg_{\mathcal{P}/\mathcal{O}}(\blank)$ as discussed in \cref{Alg2fun}.
\end{remark}

\begin{defn}
  An \emph{$\mathcal{O}$-monoidal localization} is an
  $\mathcal{O}$-monoidal functor
  $\mathcal{V}^{\otimes} \to \mathcal{U}^{\otimes}$ that has a right
  adjoint relative to $\mathcal{O}$ which is fully faithful.
\end{defn}
\begin{remark}
  Suppose $L \colon \mathcal{V}^{\otimes} \to
  \mathcal{U}^{\otimes}$ is an $\mathcal{O}$-monoidal localization
  with right adjoint $i$ relative to $\mathcal{O}$. Then $i$ is
  automatically a lax $\mathcal{O}$-monoidal functor. Indeed, more
  generally if $F \colon \mathcal{V}^{\otimes} \to
  \mathcal{U}^{\otimes}$ is any $\mathcal{O}$-monoidal functor that
  has a right adjoint $G$ relative to $\mathcal{O}$, then $G$ is lax
  $\mathcal{O}$-monoidal. This is because when $\phi \colon O \intto O'$ is inert,
  the cocartesian pushforward functor $\phi_{!} \colon
  \mathcal{V}^{\otimes}_{O} \to \mathcal{V}^{\otimes}_{O'}$ is given
  by projecting to some subset of the factors in
  $\mathcal{V}^{\otimes}_{O} \simeq \prod_{i \in |O|}
  \mathcal{V}_{O_{i}}$. The right adjoint $G_{O}$ of $F_{O} \simeq
  \prod_{i}F_{O_{i}}$ is given by the product $\prod_{i} G_{O_{i}}$,
  and so the canonical map
  \[ \phi_{!} G_{O} \to G_{O'} \phi_{!} \]
  is an equivalence in this case, since both sides are identified with $\prod_{j
    \in |O'|} G_{O'_{j}}$.
\end{remark}

\begin{remark}\label{rmk:algloc}
  Let $L \colon \mathcal{V}^{\otimes} \to
  \mathcal{U}^{\otimes}$ be an $\mathcal{O}$-monoidal localization
  with right adjoint $i\colon \mathcal{U}^{\otimes} \hookrightarrow
  \mathcal{V}^{\otimes}$. For any morphism of cartesian patterns $f
  \colon \mathcal{P} \to
  \mathcal{O}$, we get an induced localization of \icats{} of
  algebras: as in \cref{Algadj} we have an induced adjunction
  \[ L_{*} : \Alg_{\mathcal{P}/\mathcal{O}}(\mathcal{V})
    \rightleftarrows \Alg_{\mathcal{P}/\mathcal{O}}(\mathcal{U}) : i_{*},
  \]
  given by composition with $L$ and $i$, and the equivalence
  $L_{*}i_{*} \simeq (Li)_{*} \simeq \id$ implies that $i_{*}$ is
  fully faithful. Since a $\mathcal{P}$-algebra $A \colon \mathcal{O}
  \to \mathcal{V}^{\otimes}$ factors through the full subcategory
  $\mathcal{U}^{\otimes}$ \IFF{} $A(E) \in \mathcal{U}_{E}$ for every
  $E \in \mathcal{O}^{\el}$, we see that
  $\Alg_{\mathcal{P}/\mathcal{O}}(\mathcal{U})$ is the full subcategory
  of $\Alg_{\mathcal{P}/\mathcal{O}}(\mathcal{V})$ spanned by the
  algebras with this property. We can interpret this as the
  commutative square
  \[
    \begin{tikzcd}
      \Alg_{\mathcal{P}/\mathcal{O}}(\mathcal{U})
      \arrow[hookrightarrow]{r} \arrow{d} &
      \Alg_{\mathcal{P}/\mathcal{O}}(\mathcal{V}) \arrow{d} \\
      \Fun_{\mathcal{P}^{\el}/\mathcal{O}^{\el}}(\mathcal{U})
      \arrow[hookrightarrow]{r} &
      \Fun_{\mathcal{P}^{\el}/\mathcal{O}^{\el}}(\mathcal{V})
    \end{tikzcd}
  \]
  being cartesian.
\end{remark}

\begin{notation}
  Suppose $\mathcal{V}^{\otimes}$ is an $\mathcal{O}$-monoidal
  \icat{}. Given a collection of full subcategories $\mathcal{U}_{E}
  \subseteq \mathcal{V}_{E}$ for $E \in \mathcal{O}^{\el}$, the full
  subcategory $\mathcal{U}^{\otimes} \subseteq \mathcal{V}^{\otimes}$
  \emph{generated} by $(\mathcal{U}_{E})_{E \in \mathcal{O}^{\el}}$ is
  that spanned by the objects over $O \in \mathcal{O}$ that lie in the
  full subcategory of $\mathcal{V}^{\otimes}_{O}$ that is identified
  with $\prod_{i} \mathcal{U}_{O_{i}}$ under the equivalence
  $\mathcal{V}^{\otimes}_{O} \simeq
  \prod_{i}\mathcal{V}_{O_{i}}$. (Note that in general
  $\mathcal{U}^{\otimes}$ is not an $\mathcal{O}$-monoidal \icat{},
  though it is an $\mathcal{O}$-\iopd{}.)
\end{notation}

\begin{cor}\label{cor:Omonloc}
  Let $\mathcal{V}^{\otimes}$ be an $\mathcal{O}$-monoidal \icat{},
  and suppose given full subcategories $\mathcal{U}_{E} \subseteq
  \mathcal{V}_{E}$ for all $E \in \mathcal{O}^{\el}$ such that
  \begin{enumerate}[(1)]
  \item each inclusion
    $i_{E} \colon \mathcal{U}_{E} \hookrightarrow \mathcal{V}_{E}$ has
    a left adjoint $L_{E}$,
  \item for every
  active morphism $\phi \colon O \actto E$ in $\mathcal{O}$ with $E \in
  \mathcal{O}^{\el}$, the natural map
  \[ \phi_{!}^\xV O(X_{1},\ldots,X_{O_{n}}) \to
    \phi_{!}^\xV O(L_{O_{1}}X_{1},\ldots,L_{O_{n}}X_{n}) \] is taken to an
  equivalence by $L_{E}$.
  \end{enumerate}
  If $\mathcal{U}^{\otimes}\subseteq
  \mathcal{V}^{\otimes}$ is the full subcategory generated by $(\mathcal{U}_{E})_{E \in
    \mathcal{O}^{\el}}$, then:
  \begin{enumerate}[(i)]
  \item The restricted functor $\mathcal{U}^{\otimes} \to \mathcal{O}$
    is an $\mathcal{O}$-monoidal \icat{}.
  \item The inclusion $i \colon \mathcal{U}^{\otimes} \hookrightarrow
    \mathcal{V}^{\otimes}$ is lax $\mathcal{O}$-monoidal.
  \item The lax monoidal functor $i$ exhibits $\mathcal{U}^{\otimes}$
    as an $\mathcal{O}$-monoidal localization. In other words, the functor $i$ has a left adjoint $L \colon
    \mathcal{V}^{\otimes} \to \mathcal{U}^{\otimes}$ relative to
    $\mathcal{O}$ (which is then automatically an $\mathcal{O}$-monoidal functor).
  \end{enumerate}
\end{cor}
\begin{proof}
 For $O \in \mathcal{O}$, let $L_{O} \colon \mathcal{V}^{\otimes}_{O}
 \to \mathcal{U}^{\otimes}_{O}$ denote the functor corresponding to
 the product $\prod_{i} L_{O_{i}}$, which is left adjoint to the
 inclusion $i_{O}$. We claim that for $\phi \colon O \to O'$ in
 $\mathcal{O}$ and $O(U_{i}) \in \mathcal{U}^{\otimes}_{O}$, the
 morphism $O(U_{i}) \to L_{O'}\phi_{!}^{\mathcal{V}}O(U_{i})$ is
 cocartesian. It is easy to see that it is locally cocartesian, and
 then condition (2) implies that these locally cocartesian morphisms
 compose. This proves (i). If $\phi$ is inert, then we do not have to
 apply $L_{O'}$ (since $\phi_{!}^{\mathcal{V}}$ just projects to some
 factors in a product), so the inclusion $i$ preserves inert
 morphisms, which gives (ii). This means all the conditions of
 \cref{propn:Omonladj} hold for $i$, which gives (iii).
\end{proof}

Our next goal is to show that every presentably $\mathcal{O}$-monoidal
\icat{} can be described as an $\mathcal{O}$-monoidal localization of
a Day convolution. We start by briefly discussing the more general
case of \emph{accessibly} $\mathcal{O}$-monoidal \icats{}:
\begin{defn}
  We say an $\mathcal{O}$-monoidal \icat{} $\mathcal{V}^{\otimes}$ is
  \emph{accesibly $\mathcal{O}$-monoidal} if the \icats{}
  $\mathcal{V}_{E}$ for $E \in \mathcal{O}^{\el}$ are all accessible,
  and for every active map $\phi \colon O \to E$ with
  $E \in \mathcal{O}^{\el}$, the functor
  $\phi_{!}^{\mathcal{V}} \colon \prod_{i} \mathcal{V}_{O_{i}} \to
  \mathcal{V}_{E}$ is accessible. We say that $\mathcal{O}$ is
  \emph{$\kappa$-accessibly $\mathcal{O}$-monoidal} for some regular
  cardinal $\kappa$ if the \icats{} $\mathcal{V}_{E}$ are all
  $\kappa$-accessible and for every active map $\phi \colon O \to E$
  with $E \in \mathcal{O}^{\el}$, the functor
  $\phi_{!}^{\mathcal{V}} \colon \prod_{i} \mathcal{V}_{O_{i}} \to
  \mathcal{V}_{E}$ preserves $\kappa$-filtered colimits. We say that
  $\mathcal{V}^{\otimes}$ is \emph{$\kappa$-presentably
    $\mathcal{O}$-monoidal} for some regular cardinal $\kappa$ if
  $\mathcal{V}^{\otimes}$ is both presentably and $\kappa$-accessibly
  $\mathcal{O}$-monoidal.
\end{defn}
\begin{remark}
  If $\mathcal{V}^{\otimes}$ is accessibly (presentably)
  $\mathcal{O}$-monoidal,
  then we can always choose a regular cardinal $\kappa$ such that
  $\mathcal{V}^{\otimes}$ is $\kappa$-accessibly
  ($\kappa$-presentably) $\mathcal{O}$-monoidal.
\end{remark}

\begin{propn}\label{accessiblyOmon}
  Suppose $\mathcal{V}^{\otimes}$ is an accessibly
  $\mathcal{O}$-monoidal \icat{}. Then there exists a regular cardinal
  $\kappa$ such that $\mathcal{V}^{\otimes}$ is $\kappa$-accessibly
  $\mathcal{O}$-monoidal, the full subcategory
  $\mathcal{V}^{\kappa,\otimes}$ generated by the collection
  $(\mathcal{V}^{\kappa}_{E})_{E \in \mathcal{O}^{\el}}$ of
  $\kappa$-compact objects is an $\mathcal{O}$-monoidal \icat{}, and
  the inclusion
  $\mathcal{V}^{\kappa,\otimes} \hookrightarrow \mathcal{V}^{\otimes}$
  is $\mathcal{O}$-monoidal.
\end{propn}
\begin{proof}
  We can choose a regular cardinal $\lambda$ such that
  $\mathcal{V}_{E}$ is $\lambda$-accessible for each $E \in
  \mathcal{O}^{\el}$ and the functor $\phi_{!}^{\mathcal{V}} \colon
  \prod_{i} \mathcal{V}_{O_{i}} \to \mathcal{V}_{E}$ preserves
  $\lambda$-filtered colimits in each variable. We can then choose a
  regular cardinal $\kappa \gg \lambda$ such that for every such
  active map $\phi$, we have
  \[ \phi_{!}^{\mathcal{V}}\left(\prod_{i}
      \mathcal{V}^{\lambda}_{O_{i}} \right) \subseteq
    \mathcal{V}_{E}^{\kappa}. \]
  By \cite[Lemma 2.6.11]{ChuHaugseng}, any object of
  $\mathcal{V}_{E}^{\kappa}$ is the colimit of a $\kappa$-small
  $\lambda$-filtered diagram in $\mathcal{V}_{E}^{\lambda}$. Since
  $\phi_{!}^{\mathcal{V}}$ preserves $\lambda$-filtered colimits in
  each variable, for $O(v_{i}) \in \mathcal{V}^{\kappa,\otimes}_{O}$
  we can write $\phi_{!}^{\mathcal{V}}(O(v_{i}))$ as a $\kappa$-small
  colimit of $\kappa$-compact objects, and hence this object is also
  $\kappa$-compact by \cite[Corollary 5.3.4.15]{ht}. The full
  subcategory $\mathcal{V}^{\kappa,\otimes}$ therefore inherits
  cocartesian morphisms from $\mathcal{V}^{\otimes}$, which means that
  it is an $\mathcal{O}$-monoidal \icat{} and the inclusion into
  $\mathcal{V}^{\otimes}$ is an $\mathcal{O}$-monoidal functor.
\end{proof}

\begin{cor}\label{cor:presOmonloc}
  Suppose $\mathcal{V}^{\otimes}$ is a presentably
  $\mathcal{O}$-monoidal \icat{}. Then there exists a regular cardinal
  $\kappa$ such that $\mathcal{V}^{\otimes}$ is an 
  $\mathcal{O}$-monoidal localization of
  $\PSh_{\mathcal{O}}(\mathcal{V}^{\kappa})^{\otimes}$.
\end{cor}
\begin{proof}
  By \cref{accessiblyOmon} we can choose a regular cardinal $\kappa$
  such that the full subcategory $\mathcal{V}^{\kappa,\otimes}$ is an
  $\mathcal{O}$-monoidal \icat{} and the inclusion $i \colon
  \mathcal{V}^{\kappa,\otimes} \hookrightarrow \mathcal{V}^{\otimes}$
  is $\mathcal{O}$-monoidal. Since $\mathcal{V}^{\otimes}$ is compatible
  with small colimits, by \cref{propn:extPC} there then exists a
  cocontinuous $\mathcal{O}$-monoidal functor
  \[ L\colon \PSh_{\mathcal{O}}(\mathcal{V}^{\kappa})^{\otimes} \to
    \mathcal{V}^{\otimes}\]
  extending $i$ along the Yoneda embedding
  $\mathcal{V}^{\kappa,\otimes} \hookrightarrow
  \PSh_{\mathcal{O}}(\mathcal{V}^{\kappa})^{\otimes}$, and this has a
  lax $\mathcal{O}$-monoidal right adjoint $R \colon \mathcal{V}^{\otimes} \to
  \PSh_{\mathcal{O}}(\mathcal{V}^{\kappa})^{\otimes}$. It remains to
  show that $R$ is fully faithful. This amounts to showing that for
  $\phi \colon O \to E$ in $\mathcal{O}$ active with $E \in
  \mathcal{O}^{\el}$ and $O(V_{i}) \in \mathcal{V}^{\otimes}_{O}$, the
  natural map
  \[
    L_{E}\phi_{!}^{\PSh_{\mathcal{O}}(\mathcal{V}^{\kappa})}R_{O}O(V_{i})
    \to \phi_{!}^{\mathcal{V}}O(V_{i}) \] is an equivalence. This
  follows since it is true when $O(V_{i})$ lies in
  $\mathcal{V}^{\kappa,\otimes}_{O}$ and all the functors preserve
  $\kappa$-filtered colimits in each variable. (For $R_{O}$ this is
  true since it is equivalent to the product of the restricted Yoneda
  embeddings
  $\mathcal{V}_{O_{i}} \hookrightarrow
  \PSh(\mathcal{V}_{O_{i}}^{\kappa})$, which tautologically preserve
  $\kappa$-filtered colimits.)
\end{proof}

\begin{remark}
  \cref{cor:presOmonloc} says in particular that for any presentably
  $\mathcal{O}$-monoidal \icat{} $\mathcal{V}^{\otimes}$ there exists
  a small $\mathcal{O}$-monoidal \icat{} $\mathcal{C}^{\otimes}$
  and an $\mathcal{O}$-monoidal localization
  \[ \PSh_{\mathcal{O}}(\mathcal{C})^{\otimes} \xto{L}
    \mathcal{V}^{\otimes}.\]
\end{remark}

Our next goal is to obtain another description of localizations of Day
convolutions, in terms of localizing at classes of maps compatible
with the $\mathcal{O}$-monoidal structure:
\begin{notation}
  Let $\mathcal{C}$ be a small \icat{} and $\mathbb{S}$ a set of
  maps in $\PSh(\mathcal{C})$. We write
  $\PSh_{\mathbb{S}}(\mathcal{C})$ for the full subcategory of
  $\PSh(\mathcal{C})$ spanned by the \emph{$\mathbb{S}$-local}
  objects, \ie{} the objects $\Phi$ such that
  $\Map_{\PSh(\mathcal{C})}(\blank, \Phi)$ takes the morphisms in
  $\mathbb{S}$ to equivalences.
\end{notation}

\begin{definition}
  Let $\mathcal{C}^{\otimes}$ be a small $\mathcal{O}$-monoidal
  \icat{}. A collection $\mathbb{S} = (\mathbb{S}_{E})_{E \in \mathcal{O}^{\el}}$
  of sets of morphisms $\mathbb{S}_{E}$ in $\mathcal{C}_{E}$ is
  \emph{compatible with the $\mathcal{O}$-monoidal structure} if for
  every active morphism $\phi \colon O
  \actto E$ in $\mathcal{O}$ with $E \in \mathcal{O}^{\el}$, the functor
  \[\phi_{!}^{\PSh_\xxO(\xcc)} \colon \prod_{i} \PSh(\mathcal{C}_{O_{i}}) \simeq
    \PSh_{\mathcal{O}}(\mathcal{C})^{\otimes}_{O} \to
    \PSh(\mathcal{C}_{E}) \]
  takes a morphism $(\id,\ldots,\id,s,\id,\ldots,\id)$ with $s
  \in \mathbb{S}_{O_{i}}$ into the strongly saturated
  class $\overline{\mathbb{S}}_{E}$ generated by $\mathbb{S}_{E}$. We
  then write $\PSh_{\mathcal{O},\mathbb{S}}(\mathcal{C})^{\otimes}$
  for the full subcategory generated by the collection of full subcategories
  $\PSh_{\mathbb{S}_{E}}(\mathcal{C}_{E})$ of $\mathbb{S}_{E}$-local
  objects.
\end{definition}

By \cref{cor:Omonloc} we get an $\mathcal{O}$-monoidal localization of
$\PSh_{\mathcal{O}}(\mathcal{C})^{\otimes}$:
\begin{corollary}\label{cor Sloc}
  Let $\mathcal{C}^{\otimes}$ be a small $\mathcal{O}$-monoidal \icat{}
  and $\mathbb{S}$ a collection of sets of morphisms compatible with the $\mathcal{O}$-monoidal
  structure. Then there is an $\xxO$-monoidal localization
  \[\PSh_{\mathcal{O}}(\xcc)^\otimes\xto{L_{\mathbb{S}}}
    \PSh_{\mathcal{O},\mathbb{S}}(\xcc)^\otimes\]
  left adjoint to the inclusion
  $\PSh_{\mathcal{O},\mathbb{S}}(\mathcal{C})^{\otimes}
  \hookrightarrow \PSh_{\mathcal{O}}(\mathcal{C})^{\otimes}$.
\end{corollary}
\begin{proof}
  It suffices to verify the two conditions in \cref{cor:Omonloc}. Since $\mathbb{S}_{E}$ is a set for every $E\in \xxO^\xel$, the
  inclusion $\PSh_{\mathbb{S}_{E}}(\mathcal{C}_{E}) \hookrightarrow
  \PSh(\mathcal{C}_{E})$ has a left adjoint $L_{\mathbb S, E}$, which exhibits
  $\PSh_{\mathbb{S}_{E}}(\mathcal{C}_{E})$ as the localization at the
  strongly saturated class of maps generated by $\mathbb{S}_{E}$. Hence, the first condition is satisfied, and for the second condition we observe that the compatibility of
  the maps in $\mathbb{S}$ with the $\mathcal{O}$-monoidal structure implies that for every active map $\phi\colon O\actto E$, the map
  $\phi_{!}^{\PSh_\xxO(\xcc)}\xxO(X_1, \ldots, X_i, \ldots, X_n)\to \phi_{!}^{\PSh_\xxO(\xcc)}\xxO(X_1, \ldots, L_{\mathbb S, O_i}X_i, \ldots, X_n)$ lies in $\overline{\mathbb S}_E$. In particular, the map $\phi_{!}^{\PSh_\xxO(\xcc)}\xxO(X_1,\ldots, X_n)\to\phi_{!}^{\PSh_\xxO(\xcc)}\xxO(L_{\mathbb S, O_1}X_1, \ldots, L_{\mathbb S, O_n}X_n)$ is taken to an equivalence by $L_{\mathbb S, E}$, which is the second condition of
  \cref{cor:Omonloc}.
\end{proof}

\begin{remark}
  Let us say an $\mathcal{O}$-monoidal localization is
  \emph{accessible} if the component of the fully faithful right
  adjoint at each object of $\mathcal{O}^{\el}$ is an accessible
  functor. Then it follows from the classification of accessible
  localizations of presheaf \icats{} in \cite[\S 5.5.4]{ht} that every
  accessible $\mathcal{O}$-monoidal localization of
  $\PSh_{\mathcal{O}}(\mathcal{C})^{\otimes}$ is of the form
  $\PSh_{\mathcal{O},\mathbb{S}}(\mathcal{C})^{\otimes}$ for some
  collection of sets of morphisms $\mathbb{S}$ compatible with the
  $\mathcal{O}$-monoidal structure.
\end{remark}

\begin{remark}
  Suppose $\mathcal{V}^{\otimes}$ is a presentably
  $\mathcal{O}$-monoidal \icat{} and $\kappa$ is a regular cardinal
  such that $\mathcal{V}^{\otimes}$ is $\kappa$-presentably
  $\mathcal{O}$-monoidal and the full subcategory
  $\mathcal{V}^{\kappa,\otimes}$ is $\mathcal{O}$-monoidal. Then
  $\mathcal{V}^{\otimes}$ is equivalent to
  $\PSh_{\mathcal{O},\mathbb{S}^{\kappa}}(\mathcal{V}^{\kappa})^{\otimes}$
  where $\mathbb{S}^{\kappa}_{E}$ consists of the maps
  \[ \colim_{\mathcal{I}} y(\phi) \to y(\colim_{\mathcal{I}} \phi) \]
  in $\PSh(\mathcal{V}^{\kappa}_{E})$ where $\phi \colon \mathcal{I}
  \to \mathcal{V}^{\kappa}_{E}$ ranges over a set of representatives
  of $\kappa$-small colimit diagrams.
\end{remark}

\begin{definition}\label{def S Segal presheaves}
  Let $\mathcal{C}^{\otimes}$ be a small $\mathcal{O}$-monoidal
  \icat{} and $\mathbb{S}$ a collection of sets of morphisms
  compatible with the $\mathcal{O}$-monoidal structure on
  $\PSh_{\mathcal{O}}(\mathcal{C})^{\otimes}$. We say that a $\mathcal{C}^{\op,\otimes}$-monoid
  $M \colon \mathcal{C}^{\op,\otimes} \to
  \mathcal{S}$ is \emph{$\mathbb{S}$-local}
  if for every $E \in \mathcal{O}^{\el}$ the
  restriction
  \[ M_{E} \colon \mathcal{C}_{E}^{\op} \simeq
    (\mathcal{C}^{\op,\otimes})_{E} \to \mathcal{S}\]
  is  $\mathbb{S}_{E}$-local. We write
  $\Mon_{\mathcal{C}^{\op,\otimes}}^{\mathbb{S}}(\mathcal{S})$ for the
  full subcategory of $\Mon_{\mathcal{C}^{\op,\otimes}}(\mathcal{S})$
  spanned by the $\mathbb{S}$-local monoids. 
\end{definition}

\begin{propn}\label{propn MonS}
  Let $\mathcal{C}^{\otimes}$ and $\mathbb{S}$ be as in \cref{def S
    Segal presheaves}. Then there is a natural equivalence 
  \[\Alg_\xxO(\PSh_{\mathcal{O},\mathbb{S}}(\mathcal{C})) \simeq
    \Mon_{\mathcal{C}^{\op,\otimes}}^{\mathbb S}(\mathcal{S}). \]
\end{propn}
\begin{proof}
  By \cref{rmk:algloc} the inclusion
  $\PSh_{\mathcal{O},\mathbb{S}}(\mathcal{C})^{\otimes}
  \hookrightarrow \PSh_{\mathcal{O}}(\mathcal{C})^{\otimes}$
  identifies $\Alg_\xxO(\PSh_{\mathcal{O},\mathbb{S}}(\mathcal{C}))$
  with the full subcategory of
  $\Alg_{\mathcal{O}}(\PSh_{\mathcal{O}}(\mathcal{C}))$ spanned by the
  algebras $A$ such that $A(E)$ lies in
  $\PSh_{\mathbb{S}_{E}}(\mathcal{C}_{E})$ for every $E \in \mathcal{O}^{\el}$.
  Under the equivalence
  \[\Alg_{\mathcal{O}}(\PSh_{\mathcal{O}}(\mathcal{C})) \simeq
    \Mon_{\mathcal{C}^{\op,\otimes}}(\mathcal{S})\]
  of \cref{propn:AlgMonPsh} this full subcategory is identified with
  $\Mon_{\mathcal{C}^{\op,\otimes}}^{\mathbb S}(\mathcal{S})$.
\end{proof}

\begin{remark}
  \cref{propn MonS} implies that the \icat{}
  $\Alg_\xxO(\PSh_{\mathcal{O},\mathbb{S}}(\mathcal{C}))$ is
  equivalent to the full subcategory of
  $\Fun(\mathcal{C}^{\op,\otimes}, \mathcal{S})$ spanned by objects
  that are local with respect to a set of maps. Thus
  $\Alg_\xxO(\PSh_{\mathcal{O},\mathbb{S}}(\mathcal{C}))$ is an
  accessible localization of a presheaf \icat{} and so is in
  particular a presentable \icat{}. Since every presentably
  $\mathcal{O}$-monoidal \icat{} is equivalent to one of this form by
  \cref{cor:presOmonloc}, we have proved the following:
\end{remark}
\begin{cor}
  Suppose $\mathcal{V}^{\otimes}$ is a presentably
  $\mathcal{O}$-monoidal \icat{}. Then the \icat{}
  $\Alg_{\mathcal{O}}(\mathcal{V})$ is presentable. \qed
\end{cor}

\section{Extendability and Free Algebras}\label{sec:freealg}
In this section we recall the notion of \emph{extendability} for a
morphism $f \colon \mathcal{O} \to \mathcal{P}$ of cartesian patterns,
and show that if $\mathcal{V}$ is a presentably
$\mathcal{P}$-monoidal \icat{}, then the left adjoint
\[ f_{!} \colon \Alg_{\mathcal{O}}(\mathcal{V}) \to
  \Alg_{\mathcal{P}}(\mathcal{V}) \] to the functor given by
composition with $f$ can be described by an explicit colimit formula. In
particular, if $\mathcal{O}^{\xint} \to \mathcal{O}$ is extendable (in
which case we just say that $\mathcal{O}$ is extendable), then we get
an explicit formula for free $\mathcal{O}$-algebras.

\begin{defn}
  A morphism $f \colon \mathcal{O} \to \mathcal{P}$ of cartesian
  patterns has \emph{unique lifting of inert morphisms} if for every
  inert morphism $\phi \colon f(O) \intto P$ in $\mathcal{P}$ there is
  a unique lift to an inert morphism $\psi \colon O \intto O'$ in
  $\mathcal{O}$ such that $f(\psi) \simeq \phi$. In other words, the
  induced map of \igpds{}
  \[ (\mathcal{O}^{\xint}_{O/})^{\simeq} \to
    (\mathcal{P}^{\xint}_{f(O)/})^{\simeq} \]
  is an equivalence.
\end{defn}

\begin{remark}
  If $f \colon \mathcal{O} \to \mathcal{P}$ has unique lifting of
  inert morphisms, then by \cite[Corollary 7.4]{patterns} we can
  define a functor $\mathcal{P} \to \CatI$ that takes $P \in
  \mathcal{P}$ to $\mathcal{O}^{\act}_{/P}$. A morphism $\alpha
  \colon P \to P'$ in $\mathcal{P}$ is sent to the functor $\alpha_{!} \colon \mathcal{O}^{\act}_{/P}
  \to \mathcal{O}^{\act}_{/P'}$ that takes a pair $(O, f(O) \actto P)$ to
  the pair
  $(O', f(O') \actto P')$ given by first forming the commutative
  square
  \[
    \begin{tikzcd}
      f(O) \arrow[inert]{r} \arrow[active]{d} & Q
      \arrow[active]{d} \\
      P \arrow{r}{\alpha} & P'
    \end{tikzcd}
  \]
  by taking the inert--active factorization of $f(O) \actto P
  \xto{\alpha} P'$ and then lifting the inert map $f(O) \intto Q$ to a
  unique inert map $O \intto O'$ in $\mathcal{O}$.
\end{remark}

\begin{defn}\label{def ext}
  A morphism $f \colon \mathcal{O} \to \mathcal{P}$ of cartesian
  patterns is \emph{extendable} if
  \begin{enumerate}[(1)]
  \item $f$ has unique lifting of inert morphisms,
  \item for $P \in \mathcal{P}$ over $\angled{n}$, the functor
    \[ \mathcal{O}^{\act}_{/P} \to \prod_{i=1}^{n}
      \mathcal{O}^{\act}_{/P_{i}} \] taking
    $(O, \phi \colon f(O) \to P)$ to $(\rho_{i,!}^{P}(O,\phi))_{i}$ is
    cofinal.
  \end{enumerate}
\end{defn}
\begin{remark}
  The general definition of an extendable morphism in \cite[Definition
  7.7]{patterns} has a third condition, but this is automatic in the case
  of cartesian patterns: Namely, given an active morphism $\phi
  \colon f(O) \actto P$, we can use the unique lifting of inert morphisms
  to define a functor
  $\mathcal{P}^{\el,\op}_{P/} \to \CatI$ taking $\alpha \colon P \to E$ to
  $\mathcal{O}^{\el}_{\alpha_{!}O/}$. If $\mathcal{O}^{\el}(\phi) \to
  \mathcal{P}^{\el}_{P/}$ denotes the corresponding cartesian
  fibration, then there is a functor $\mathcal{O}^{\el}(\phi) \to
  \mathcal{O}^{\el}_{O/}$ that takes $(\alpha, \alpha_{!}O \to E')$ to
  $O \to \alpha_{!}O \to E'$. The condition is that this functor
  should induce an equivalence
  \[ \lim_{\mathcal{O}^{\el}_{O/}} F \to
    \lim_{\mathcal{O}^{\el}(\phi)} F \]
  for every functor $F \colon \mathcal{O}^{\el} \to
  \mathcal{S}$. However, if $f$ is a morphism of cartesian patterns,
  then this functor is necessarily an equivalence: if $\angled{n} = |O|$ and
  $\angled{m} = |P|$, then $\mathcal{O}^{\el}_{O/}$ and
  $\mathcal{P}^{\el}_{P/}$ are isomorphic to the discrete sets
 $\{\rho_{i}^{O} \colon
  i = 1,\ldots,n\}$ and $\{\rho_{i}^{P} \colon
  i = 1,\ldots,m\}$, while 
  $\mathcal{O}^{\el}_{\rho^{P}_{i,!}O/}$ is isomorphic to the set
  $|\phi|^{-1}(i)$. Thus $\mathcal{O}^{\el}(\phi)$ is the set of pairs
  $\{(i,j) : 1 \leq i \leq m, j \in |\phi|^{-1}(i) \}$ and the map to
  $\mathcal{O}^{\el}_{O/}$ is the obvious isomorphism of this with
  $\{1,\ldots,n\}$ (implied by $\phi$ being active).
\end{remark}

\begin{example}\label{monext}
  If $f \colon \mathcal{O} \to \mathcal{P}$ is an extendable morphism
  of cartesian patterns, then for any commutative square
  \[
    \begin{tikzcd}
      \mathcal{C}^{\otimes} \arrow{r}{F} \arrow{d}&
      \mathcal{D}^{\otimes} \arrow{d} \\
      \mathcal{O} \arrow{r}{f} &  \mathcal{P},
    \end{tikzcd}
  \]
  where $\mathcal{C}^{\otimes}$ is an $\mathcal{O}$-monoidal
  \icat{}, $\mathcal{D}^{\otimes}$ is a $\mathcal{P}$-monoidal
  \icat{}, and $f$ preserves cocartesian morphisms, then the morphism $F$ is
  extendable. This is a special case of  \cite[Proposition
  9.5]{patterns}. In particular, for any $\mathcal{P}$-monoidal
  \icat{} $\mathcal{D}^{\otimes}$, in the pullback square
  \[
    \begin{tikzcd}
      f^{*}\mathcal{D}^{\otimes} \arrow{r}{\bar{f}} \arrow{d}&
      \mathcal{D}^{\otimes} \arrow{d} \\
      \mathcal{O} \arrow{r}{f} &  \mathcal{P},
    \end{tikzcd}
  \]
  the morphism $\bar{f} \colon f^{*}\mathcal{D}^{\otimes} \to
  \mathcal{D}^{\otimes}$ is extendable.
\end{example}

\begin{propn}\label{propn:monext}
  Suppose $f \colon \mathcal{O} \to \mathcal{P}$ is an extendable
  morphism of cartesian patterns. Then the functor
  $f^{*} \colon \Mon_{\mathcal{P}}(\mathcal{S}) \to
  \Mon_{\mathcal{O}}(\mathcal{S})$ has a left adjoint $f_{!}$, given
  by left Kan extension along $f$, which satisfies
  \[ f_{!}M(P) \simeq \colim_{O \in \mathcal{O}^{\act}_{/P}} M(O).\]
\end{propn}
\begin{proof}
  This is a special case of \cite[Proposition
  7.13]{patterns}. We give a brief sketch of the proof, as it is
  particularly simple in the case of cartesian patterns. Since
  $f^{*}\colon \Fun(\mathcal{P}, \mathcal{S}) \to \Fun(\mathcal{O},
  \mathcal{S})$ has a left adjoint $f_{!}$ given by left Kan
  extension, it suffices to show that if $M$ is an
  $\mathcal{O}$-monoid, then the left Kan extension $f_{!}M$ is an
  $\mathcal{O}$-monoid. We have natural equivalences
  \begin{align*}
    f_{!}M(P) & \simeq \colim_{O \in \mathcal{O}_{/P}} M(O) && \\
    & \simeq \colim_{O \in \mathcal{O}^{\act}_{/P}} M(O) & &
                                                             \text{(\cref{def
                                                             ext}(1)
                                                             and
                                                             \cite[7.2]{patterns})}
    \\
    & \simeq \colim_{(O_{i}) \in \prod_{i}
      \mathcal{O}^{\act}_{/P_{i}}}  \prod_{i} M(O_{i})  && \text{(\cref{def
                                                             ext}(2))}\\
    & \simeq \prod_{i} \colim_{O_{i} \in
      \mathcal{O}^{\act}_{/P_{i}}}M(O_{i}) && \text{($\mathcal{S}$
                                              cartesian closed)} \\
    & \simeq \prod_{i} f_{!}M(P_{i}),
  \end{align*}
  as required.
\end{proof}

\begin{remark}
  The same result is true more generally for monoids in any \icat{}
  where the cartesian product commutes with colimits indexed by the
  \icats{} $\mathcal{O}^{\act}_{/P}$.
\end{remark}

\begin{defn}\label{def ext patt}
  Let $\mathcal{O}$ be a cartesian pattern. We write
  $\Act_{\mathcal{O}}(O)$ for the \igpd{} of active morphisms to $O$
  in $\mathcal{O}$. We say $\mathcal{O}$ is \emph{extendable} if the functor
  \[ \Act_{\mathcal{O}}(O) \to \prod_{i} \Act_{\mathcal{O}}(O_{i}), \]
  taking $\phi \colon O' \actto O$ to the morphism $\rho_{i,!}^{O}\phi$
  given by the inert-active factorization
  \[
    \begin{tikzcd}
      O' \arrow[active]{r}{\phi} \arrow[inert]{d} & O \arrow[inert]{d}{\rho^{O}_{i}} \\
      \rho_{i,!}^{O}O' \arrow[active]{r}{\rho_{i,!}^{O}\phi} & O_{i},
    \end{tikzcd}
    \]
    is an equivalence.
\end{defn}
\begin{remark}
  Since the equivalences are precisely the morphisms that are both
  active and inert, we can identify $\Act_{\mathcal{O}}(O)$ with
  $(\mathcal{O}^{\xint})^{\act}_{/O}$. Thus $\mathcal{O}$ is extendable
  \IFF{} the inclusion $\mathcal{O}^{\xint} \to \mathcal{O}$ is an
  extendable morphism of cartesian patterns (since unique lifting of
  inert morphisms is tautological in this case).
\end{remark}

\begin{cor}
  Suppose $\mathcal{O}$ is an extendable cartesian pattern. Then the
  functor
  \[U_{\mathcal{O}} \colon \Mon_{\mathcal{O}}(\mathcal{S}) \to
    \Fun(\mathcal{O}^{\el}, \mathcal{S})\] given by restriction to
  $\mathcal{O}^{\el}$ has a left adjoint $F_{\mathcal{O}}$ given by
  right Kan extension along
  $\mathcal{O}^{\el} \hookrightarrow \mathcal{O}^{\xint}$ followed by
  left Kan extension along $\mathcal{O}^{\xint} \to \mathcal{O}$. This
  satisfies
  \[ F_{\mathcal{O}}(\Phi)(O) \simeq \colim_{O' \actto O \in
      \Act_{\mathcal{O}}(O)} \prod_{i} \Phi(O'_{i}). \]
\end{cor}
\begin{proof}
  Combine \cref{propn:monext} with \cref{OintmonRKE}.
\end{proof}

\begin{example}\label{ex:opdext}
  If $\mathcal{O}$ is an extendable cartesian pattern, then for any
  morphism of $\mathcal{O}$-\iopds{} (\ie{} weak Segal
  $\mathcal{O}$-fibrations in the terminology of \cite{patterns})
  \[
    \begin{tikzcd}
      \mathcal{E}  \arrow{rr}{f} \arrow{dr} && \mathcal{F}
      \arrow{dl} \\
      & \mathcal{O},
    \end{tikzcd}
  \]
  the morphism $f$ is extendable by \cite[Corollary 9.16]{patterns}. In particular, any
  $\mathcal{O}$-\iopd{} is an extendable cartesian pattern.
\end{example}

We now want to extend these descriptions of left adjoints from monoids
to algebras, starting with the special case of Day convolution:
\begin{propn}
  Suppose $\mathcal{C}$ is a small $\mathcal{P}$-monoidal
  \icat{}, and $f \colon \mathcal{O} \to \mathcal{P}$ is an extendable
  morphism of cartesian patterns. Then the functor
  \[ f^{*} \colon \Alg_{\mathcal{P}}(\PSh_{\mathcal{P}}(\mathcal{C})^{\otimes}) \to
    \Alg_{\mathcal{O}/\mathcal{P}}(\PSh_{\mathcal{P}}(\mathcal{C})^{\otimes})\]
  has a left adjoint $f_{!}$, which for $P \in \mathcal{P}^{\el}$ satisfies
  \[ (f_{!}A)(P) \simeq \colim_{(O,\phi \colon f(O) \actto P) \in
      \mathcal{O}^{\act}_{/P}} \phi_{!}A(O).\]
\end{propn}
\begin{proof}
  Consider the pullback square
  \[
    \begin{tikzcd}
    f^{*}\mathcal{C}^{\op,\otimes} \arrow{r}{\bar{f}} \arrow{d} &
    \mathcal{C}^{\op,\otimes} \arrow{d} \\
    \mathcal{O} \arrow{r}{f} & \mathcal{P},
    \end{tikzcd}
    \]
    where the morphism
    $\bar{f} \colon f^{*}\mathcal{C}^{\op,\otimes} \to
    \mathcal{C}^{\op,\otimes}$ is an extendable morphism of cartesian
    patterns by \cref{monext}. From \cref{propn:monext} we
    therefore get a left adjoint
    \[ \bar{f}_{!} \colon \Mon_{f^{*}\mathcal{C}^{\op,\otimes}}(\mathcal{S})
    \to \Mon_{\mathcal{C}^{\op,\otimes}}(\mathcal{S}),\]
    given by left Kan extension; for $X \in \mathcal{C}^{\op,\otimes}$ and
    $M \in \Mon_{f^{*}\mathcal{C}^{\op,\otimes}}(\mathcal{S})$ this satisfies
    \[ \bar{f}_{!}M(X) \simeq
    \colim_{(Y, \bar{f}(Y) \to X) \in
      (f^{*}\mathcal{C}^{\op,\otimes})^{\act}_{/X}} M(Y).\]
    If $X$ lies over $P \in \mathcal{P}$, we see from the proof of
    \cite[Proposition 9.5]{patterns} that the canonical projection 
    $(f^{*}\mathcal{C}^{\op,\otimes})^{\act}_{/X} \to
    \mathcal{O}^{\act}_{/P}$ is a cocartesian fibration, whose fibre at
    $(O, f(O) \actto P)$ is
    \[(f^{*}\mathcal{C}^{\op,\otimes})_{O}
    \times_{\mathcal{C}^{\op,\otimes}_{P}} \mathcal{C}^{\op,\otimes}_{P/X}
    \simeq \mathcal{C}^{\op,\otimes}_{f(O)}
    \times_{\mathcal{C}^{\op,\otimes}_{P}}
    \mathcal{C}^{\op,\otimes}_{P/X} \simeq \mathcal{C}^{\op,\otimes}_{f(O)/X}.\]
    Thus we can rewrite the formula for $\bar{f}_{!}M(X)$ as
    \[ \bar{f}_{!}M(X) \simeq \colim_{(O, f(O) \overset{\phi}{\actto} P) \in
      \mathcal{O}^{\act}_{/P}} \colim_{(Y, \phi_{!}Y \to X) \in
      \mathcal{C}^{\op,\otimes}_{f(O)/X}} M(Y),\]
    where we are omitting notation for the equivalence
    $(f^{*}\mathcal{C}^{\otimes})_{O} \simeq
    \mathcal{C}^{\otimes}_{f(O)}$.
    Here we can identify $\colim_{(Y, \phi_{!}Y \to X) \in
      \mathcal{C}^{\op,\otimes}_{f(O)/X}} M(Y)$ with the value at $X$ of
    the left Kan extension of $M|_{\mathcal{C}^{\op,\otimes}_{f(O)}}$
    along $\phi_{!} \colon \mathcal{C}^{\op,\otimes}_{f(O)} \to
    \mathcal{C}^{\op,\otimes}_{P}$, which is the cocartesian
    pushforward in $\PSh_{\mathcal{P}}(\mathcal{C}^{\otimes})$. 
    Under the natural equivalences
    \[\Mon_{\mathcal{C}^{\op,\otimes}}(\mathcal{S}) \simeq
    \Alg_{\mathcal{P}}(\PSh_{\mathcal{P}}(\mathcal{C})^{\otimes}),
    \quad \Mon_{f^{*}\mathcal{C}^{\op,\otimes}}(\mathcal{S}) \simeq
    \Alg_{\mathcal{O}}(\PSh_{\mathcal{O}}(f^{*}\mathcal{C})^{\otimes})
    \simeq
    \Alg_{\mathcal{O}/\mathcal{P}}(\PSh_{\mathcal{P}}(\mathcal{C})^{\otimes})\]
    this formula therefore corresponds to the one above.  
  \end{proof}
  
  \begin{cor}\label{cor:f!alg}
    Suppose $\mathcal{V}$ is a presentably $\mathcal{P}$-monoidal
    \icat{}, and $f \colon \mathcal{O} \to \mathcal{P}$ is an extendable
    morphism of cartesian patterns. Then the functor
    \[ f^{*} \colon \Alg_{\mathcal{P}}(\mathcal{V}) \to
    \Alg_{\mathcal{O}/\mathcal{P}}(\mathcal{V})\]
    has a left adjoint $f_{!}$, which for $P \in \mathcal{P}^{\el}$ satisfies
    \[ (f_{!}A)(P) \simeq \colim_{(O,\phi \colon f(O) \actto P) \in
      \mathcal{O}^{\act}_{/P}} \phi_{!}A(O).\]
  \end{cor}
  \begin{proof}
    Since $\mathcal{V}$ is presentably $\mathcal{P}$-monoidal, by
    \cref{cor:presOmonloc} there
    exists a small $\mathcal{P}$-monoidal \icat{} $\mathcal{C}$ and a
    $\mathcal{P}$-monoidal localization
    \[ L \colon \PSh_{\mathcal{P}}(\mathcal{C})^{\otimes} \to
    \mathcal{V}^{\otimes},\]
    with a fully faithful lax $\mathcal{P}$-monoidal right adjoint $i \colon
    \mathcal{V}^{\otimes} \to
    \PSh_{\mathcal{P}}(\mathcal{C})^{\otimes}$. From
    \cref{rmk:algloc} we then have a
    commutative square
    \[
    \begin{tikzcd}
    \Alg_{\mathcal{P}}(\mathcal{V}) \arrow{r}{f^{*}_{\mathcal{V}}} \arrow{d}{i_{*}}
    & \Alg_{\mathcal{O}/\mathcal{P}}(\mathcal{V}) \arrow{d}{i_{*}} \\
    \Alg_{\mathcal{P}}(\PSh_{\mathcal{P}}(\mathcal{C})^{\otimes}) \arrow{r}{f^{*}} 
    & \Alg_{\mathcal{O}/\mathcal{P}}(\PSh_{\mathcal{P}}(\mathcal{C})^{\otimes}), 
    \end{tikzcd}
    \]
    where the vertical functors are both fully faithful, with left
    adjoints given by $L_{*}$. It follows that we have a commutative
    square of left adjoints
    \[
    \begin{tikzcd}
    \Alg_{\mathcal{O}/\mathcal{P}}(\PSh_{\mathcal{P}}(\mathcal{C})^{\otimes})
    \arrow{r}{f_{!}} \arrow{d}{L_{*}} & 
    \Alg_{\mathcal{P}}(\PSh_{\mathcal{P}}(\mathcal{C})^{\otimes})
    \arrow{d}{L_{*}} \\
    \Alg_{\mathcal{O}/\mathcal{P}}(\mathcal{V}) \arrow{r}{f_{\mathcal{V},!}}    
    & 
    \Alg_{\mathcal{P}}(\mathcal{V}).
    \end{tikzcd}
    \]
    Since the right adjoint $i_*$ is fully faithful, the counit
    $L_{*}i_{*} \isoto \id$ is invertible, and combining this with the
    equivalence from the square we get
    \[ f_{\mathcal{V},!} \simeq f_{\mathcal{V},!}L_{*}i_{*} \simeq L_{*}f_{!}i_{*}.\]
    Since $L$ preserves colimits and cocartesian
    morphisms, this implies that 
    $f_{\mathcal{V},!}$ satisfies
    \[ f_{\mathcal{V},!}A(P) \simeq L\left(\colim_{(O,\phi \colon f(O) \actto P) \in
      \mathcal{O}^{\act}_{/P}} \phi_{!}i(A(O))\right) \simeq
    \colim_{(O,\phi \colon f(O) \actto P) \in
      \mathcal{O}^{\act}_{/P}} \phi_{!} A(O),\]
  as required.
\end{proof}

\begin{cor}\label{cor:FOmonadic}
  Suppose $\mathcal{O}$ is an extendable cartesian pattern, and
  $\mathcal{V}$ is a presentably $\mathcal{O}$-monoidal \icat{}. Then
  the restriction
  \[ U_{\mathcal{O}} \colon \Alg_{\mathcal{O}}(\mathcal{V}) \to
    \Alg_{\mathcal{O}^{\xint}/\xxO}(\mathcal{V}) \simeq
    \Fun_{/\mathcal{O}^{\el}}(\mathcal{O}^{\el},
    \mathcal{V}) \]
  has a left adjoint $F_{\mathcal{O}} \colon \Fun_{/\mathcal{O}^{\el}}(\mathcal{O}^{\el},
  \mathcal{V}) \to \Alg_{\mathcal{O}}(\mathcal{V})$, which for
  $\Phi \colon \mathcal{O}^{\el} \to \mathcal{V}$ and $E \in
  \mathcal{O}^{\el}$ is given by
  \[ F_{\mathcal{O}}\Phi(E) \simeq \colim_{(\phi \colon O \actto E) \in
      \name{Act}_{\mathcal{O}}(E)}
    \phi_{!}(\Phi(O_{1}),\ldots,\Phi(O_{n})).\]
  Moreover, the adjunction $F_{\mathcal{O}} \dashv U_{\mathcal{O}}$ is monadic.
\end{cor}
\begin{proof}
  The existence of the left adjoint follows from \cref{cor:f!alg} 
  applied to the map $\mathcal{O}^{\xint} \to \mathcal{O}$ (together
  with the equivalence of \cref{lem:AlgOint}). To see the adjunction
  is monadic we apply the monadicity theorem for \icats{},
  \cite[Theorem 4.7.3.5]{ha}. We then need to show that
  $U_{\mathcal{O}}$ detects equivalences, which is clear, and that
  $U_{\mathcal{O}}$-split simplicial objects have colimits and these
  are preserved by $U_{\mathcal{O}}$. Suppose therefore that we
  have a $U_{\mathcal{O}}$-split simplicial diagram $\phi \colon \Dop
  \to \Alg_{\mathcal{O}}(\mathcal{V})$. Since $\mathcal{V}^{\otimes}$
  is presentably $\mathcal{O}$-monoidal, it is an
  $\mathcal{O}$-monoidal localization of a Day convolution
  $\mathcal{P}_{\mathcal{O}}(\mathcal{C})^{\otimes}$, which gives a commutative
  diagram
  \[
    \begin{tikzcd}
      \Alg_{\mathcal{O}}(\mathcal{V}) \arrow[hookrightarrow]{r} \arrow{d}{U_{\mathcal{O}}}&
      \Mon_{\mathcal{C}^{\op,\otimes}}(\mathcal{S}) \arrow{d}{U_{\mathcal{O}}'}\\
      \Fun_{/\mathcal{O}^{\el}}(\mathcal{O}^{\el}, \mathcal{V})
      \arrow[hookrightarrow]{r} & \Fun(\mathcal{C}^{\op,\otimes}_{/\el}, \mathcal{S}).
    \end{tikzcd}
  \]
  Here the functor $U_{\mathcal{O}}'$ is a monadic right adjoint by
  \cite[Corollary 8.2]{patterns} and so the colimit of the
  $U'_{\mathcal{O}}$-split simplicial diagram that is the image of
  $\phi$ exists and is preserved by
  $U'_{\mathcal{O}}$. (Alternatively, this follows from the
  description of sifted colimits of monoids in \cref{rmk:Monsifted}.)
  But by \cref{rmk:algloc} the commutative square above is cartesian,
  hence the colimit in $\Mon_{\mathcal{C}^{\op,\otimes}}(\mathcal{S})$
  actually lies in the full subcategory
  $\Alg_{\mathcal{O}}(\mathcal{V})$. Thus the $U_{\mathcal{O}}$-split
  simplicial diagram $\phi$ has a colimit and this is preserved by
  $U_{\mathcal{O}}$, as required.
\end{proof}

\begin{remark}
  We can remove the presentability condition in \cref{cor:FOmonadic}:
  since a small $\mathcal{O}$-monoidal \icat{} $\mathcal{C}^{\otimes}$
  is  always a full
  $\mathcal{O}$-monoidal subcategory of the presentably
  $\mathcal{O}$-monoidal \icat{}
  $\PSh_{\mathcal{O}}(\mathcal{C})^{\otimes}$, we can embed any 
  large $\mathcal{O}$-monoidal \icat{}
  $\mathcal{V}^{\otimes}$ in a presentably $\mathcal{O}$-monoidal
  \icat{} in a larger universe. Moreover, we can do so in a way that
  preserves small colimits, in which case we see that the left adjoint
  from \cref{cor:FOmonadic} restricts to the full subcategory of
  $\mathcal{O}$-algebras in $\mathcal{V}^{\otimes}$ provided the
  $\mathcal{O}$-monoidal structure is compatible with colimits of
  shape $\Act_{\mathcal{O}}(E)$ for $E \in \mathcal{O}^{\el}$.
\end{remark}

\begin{remark}
  If $\mathcal{O}$ is an extendable cartesian pattern, then the
  formula for the free $\mathcal{O}$-monoid monad on
  $\Fun(\mathcal{O}^{\el}, \mathcal{S})$ shows that this is an
  \emph{analytic} monad in the sense of \cite{AnalMnd}, and hence
  corresponds by the results of that paper to an \iopd{} (in the sense
  of a not necessarily complete dendroidal Segal space).
  We expect that this observation can be strengthened:
  there should be a canonical morphism
  $\mathcal{O} \to \mathcal{O}_{\name{opd}}$ of cartesian patterns
  where $\mathcal{O}_{\name{opd}}$ is a symmetric \iopd{}, such that
  $\Mon_{\mathcal{O}}(\mathcal{S}) \isoto
  \Mon_{\mathcal{O}_{\name{opd}}}(\mathcal{S})$ and
  $\mathcal{O}^{\el} \to \mathcal{O}_{\name{opd}}^{\el}$ is an
  epimorphism (\ie{} is surjective on $\pi_{0}$). We hope to address
  this question elsewhere.
\end{remark}

\section{Examples of Extendability}\label{sec:extex}
In this section we will give some examples of extendable patterns and
morphisms, and spell out what our results from the previous
section amount to in these examples.

\begin{example}
  The pattern $\xF_{*}^{\flat}$ is extendable: The category
  $\xF_{*}^{\act}$ can be identified with the category $\xF$ of
  (unpointed) finite sets, so the desired equivalence
  \[ \Act(\angled{n}) \isoto \to \prod_{i=1}^{n}
    \Act(\angled{1}) \]
  corresponds to underlying equivalence of groupoids from the (``straightening'') equivalence
  \[ \xF_{/\mathbf{n}} \to \prod_{i=1}^{n} \xF \]
  between sets over $\mathbf{n}:=\{1, \ldots, n\}$ and families of sets indexed by
  $\mathbf{n}$, given by taking fibres at $i \in \mathbf{n}$.
  The groupoid $\Act_{\xF_{*}}(\angled{1})$ is equivalent to the
  groupoid $\xF^{\simeq}$ of finite sets and bijections, \ie{}
  $\coprod_{n=0}^{\infty} B \Sigma_{n}$, and we recover the expected
  formula for free commutative algebras in a presentably symmetric
  monoidal \icat{} $\mathcal{V}^{\otimes}$: 
  \[ U_{\xF_{*}}F_{\xF_{*}}(V)  \simeq
    \colim_{\phi \colon \angled{n} \to \angled{1} \in
      \Act_{\xF_{*}}(\angled{1})} \phi_{!}(V,\ldots,V) 
    \simeq \coprod_{n=0}^{\infty} V^{\otimes n}_{h \Sigma_{n}}.\]
\end{example}

\begin{example}\label{ex:symopdmorext}
  By \cref{ex:opdext}, every morphism $f \colon \mathcal{O} \to
  \mathcal{O}'$ of symmetric
  \iopds{} is extendable. If $\mathcal{V}^{\otimes}$ is a presentably
  $\mathcal{O}'$-monoidal \icat{}, we recover the formula for the operadic
  left Kan extension $f_{!} \colon \Alg_{\mathcal{O}/\mathcal{O}'}(\mathcal{V}) \to
  \Alg_{\mathcal{O}'}(\mathcal{V})$ from \cite{ha}: for $X \in
  \mathcal{O}'_{\angled{1}}$ and $A \in
  \Alg_{\mathcal{O}/\mathcal{O}'}(\mathcal{V})$, we have
  \[ f_{!}A(X) \simeq \colim_{(O,\phi) \in 
      \mathcal{O}^{\act}_{/X}} f(\phi)_{!}A(O). \]
  In particular, if $\mathcal{V}^{\otimes}$ is a symmetric monoidal
  \icat{}, then we have
  \[ f_{!}A(X) \simeq \colim_{(O,\phi) \in 
      \mathcal{O}^{\act}_{/X}} |\phi|_{!}A(O)
    \simeq
    \colim_{(O,\phi) \in
      \mathcal{O}^{\act}_{/X}} \bigotimes_{i} A(O_{i}).\]
\end{example}

\begin{example}\label{ex:symopdext}
  As a special case of the previous example, every symmetric \iopd{}
  $\mathcal{O}$ is an extendable cartesian pattern, and our results
  recover the expected formula for free $\mathcal{O}$-algebras in a
  presentably symmetric monoidal \icat{} $\mathcal{V}^{\otimes}$: the
  forgetful functor $U_{\mathcal{O}} \colon \Alg_{\mathcal{O}}(\mathcal{V}) \to
  \Fun(\mathcal{O}^{\el}, \mathcal{V})$ has a left adjoint
  $F_{\mathcal{O}}$, which for $E \in \mathcal{O}^{\el}$ satisfies
  \[ F_{\mathcal{O}}\Phi(E) \simeq \colim_{O \xto{\phi} E \in
      \Act_{\mathcal{O}}(E)} \bigotimes_{i}\Phi(O_{i}).\]
  If $\mathcal{O}^{\el} := \mathcal{O}^{\simeq}_{\angled{1}} \simeq *$
  we can define $\mathcal{O}(n)$ to be the fibre of
  $\Act_{\mathcal{O}}(*) \to \xF^{\simeq}$ at the point of
  $B\Sigma_{n}$ (with its canonical $\Sigma_{n}$-action), and then
  rewrite the formula in the more familiar form
  \[ F_{\mathcal{O}}V \simeq \coprod_{n=0}^{\infty}
    \mathcal{O}(n) \otimes_{\Sigma_{n}} V^{\otimes n}.\]
\end{example}

\begin{remark}
  The formula in the previous example does not agree with that given
  in \cite[\S 3.1.3]{ha}. This is because \cite{ha} uses the term
  ``free algebras'' in a non-standard way: instead of considering the
  operadic left Kan extension to $\mathcal{O}$ from the \iopd{}
  $\mathcal{O}^{\xint}$ containing only the inert morphisms in
  $\mathcal{O}$, Lurie considers the extension from the \iopd{}
  $\mathcal{O} \times_{\xF_{*}} \xF_{*}^{\xint}$ containing those
  morphisms that map to inert morphisms in $\xF_{*}$ (but are not
  necessarily cocartesian). The difference is that
  $\mathcal{O} \times_{\xF_{*}} \xF_{*}^{\xint}$ remembers \emph{all} the
  unary operations in $\mathcal{O}$, while $\mathcal{O}^{\xint}$
  remembers only the \emph{invertible} ones. 
\end{remark}

\begin{example}\label{ex:gensymopdext}
  Suppose $f \colon \mathcal{O} \to \mathcal{P}$ is a morphism of
  \emph{generalized} symmetric \iopds{} (in the sense of \cite[\S
  2.3.2]{ha}, or equivalently weak Segal fibrations over
  $\xF_{*}^{\natural}$ in the terminology of \cite{patterns}). This
  certainly has unique lifting of inert morphisms, and so is
  extendable as a morphism of cartesian patterns \IFF{} for every
  $P \in \mathcal{P}$ over $\angled{n}$, the functor
  \[ \mathcal{O}^{\act}_{/P} \to \prod_{i=1}^{n}
    \mathcal{O}^{\act}_{/P_{i}} \]
  is cofinal. In particular, a generalized \iopd{} $\mathcal{O}$ is
  extendable \IFF{}
  \[ \Act_{\mathcal{O}}(O) \to \prod_{i=1}^{n}
    \Act_{\mathcal{O}}(O_{i}) \]
   is an equivalence for $O \in \mathcal{O}$ over $\angled{n}$. By
   \cite[Proposition 9.15]{patterns}, we do have an equivalence
   between $\Act_{\mathcal{O}}(O)$ and the iterated fibre product
   \[ \Act_{\mathcal{O}}(O) \to \Act_{\mathcal{O}}(O_{1})
     \times_{\Act_{\mathcal{O}}(\sigma_{!}O)} \cdots
     \times_{\Act_{\mathcal{O}}(\sigma_{!}O)}
     \Act_{\mathcal{O}}(O_{n}),\]
   where $\sigma$ denotes the unique map $\angled{n} \to
   \angled{0}$. Since the only active map to $\angled{0}$ in $\xF_{*}$
   is the identity, for $X \in \mathcal{O}_{\angled{0}}$ the \igpd{}
   $\Act_{\mathcal{O}}(X)$ is equivalent to
   $\mathcal{O}_{\angled{0}/X}^{\simeq}$. If
   $\mathcal{O}_{\angled{0}}$ is an \igpd{}, then
   $\Act_{\mathcal{O}}(X)$ is therefore contractible for all $X \in
   \mathcal{O}_{\angled{0}}$. This shows that a generalized symmetric
   \iopd{} $\mathcal{O}$ such that $\mathcal{O}_{\angled{0}}$ is an
   \igpd{} is always extendable. More generally, if $f \colon
   \mathcal{O} \to \mathcal{P}$ is a morphism of generalized \iopds{}
   such that $\mathcal{O}_{\angled{0}}$ and $\mathcal{P}_{\angled{0}}$
   are \igpds{} and $f_{\angled{0}} \colon \mathcal{O}_{\angled{0}}
   \to \mathcal{P}_{\angled{0}}$ is an equivalence, then $f$ is extendable,
   since we have
   \[ \mathcal{O}^{\act}_{0} \times_{\mathcal{P}^{\act}_{0}}
     \mathcal{P}^{\act}_{0/P} \simeq *\]
   for every $P \in \mathcal{P}_{0}$. (Note, however, that a more general
   morphism between generalized \iopds{} whose fibres at
     $\angled{0}$ are \igpds{} may still fail to be extendable.)
\end{example}

\begin{example}\label{ex:Dopext}
  The pattern $\simp^{\op,\flat}$ is extendable:
  We can identify the category $\simp^{\op,\act}$ with the
  category $\mathbb{O}$ of finite ordered sets; then the desired
  equivalence
  \[ \Act([n]) \isoto  \to \prod_{i=1}^{n} \Act([1]) \]
  becomes underlying equivalence of groupoids arising from the obvious equivalence
  \[ \mathbb{O}_{/\mathbf{n}} \to \prod_{i=1}^{n}
    \mathbb{O}_{/\mathbf{1}} \]
  that takes an ordered set over $\mathbf{n}$ to its fibres at the
  points of $\mathbf{n}$. Since every object of $\Dop$ has a
  unique active map to $[1]$, the groupoid $\Act_{\Dop}([1])$ is
  isomorphic to the set $\{0,1,\ldots\}$. If $\mathcal{V}^{\otimes}$
  is a presentably $\Dop$-monoidal \icat{} we get the expected formula
  for free associative algebras:
  \[ T_{\Dop}(V) \simeq \colim_{\phi \colon [n] \to [1] \in
      \Act_{\Dop}([1])} \phi_{!}(V,\ldots,V) \simeq
    \coprod_{n=0}^{\infty} V^{\otimes n}. \]
  We also get analogues of
  \cref{ex:symopdmorext,ex:symopdext,ex:gensymopdext}:
  \begin{itemize}
  \item every morphism of non-symmetric \iopds{} is extendable,
  \item every non-symmetric \iopd{} is extendable,
  \item every generalized non-symmetric \iopd{} whose fibre at $[0]$
    is an \igpd{} is extendable,
  \item every morphism of generalized non-symmetric \iopds{} whose
    fibres at $[0]$ are \igpds{} and whose restriction to these is an equivalence, is extendable.\footnote{The existence
      of operadic left Kan extensions in this case was used in
      \cite{enriched}.}
  \end{itemize}
\end{example}

\section{Morita Equivalences}\label{sec:Morita}
In this section we use our results on extendable cartesian patterns to
give a condition for a morphism of cartesian patterns to induce
equivalences on \icats{} of algebras, \ie{} to be a \emph{Morita
  equivalence} in the following sense:

\begin{defn}
  We say that a morphism of cartesian patterns $f \colon \mathcal{O}
  \to \mathcal{P}$ is a \emph{Morita equivalence} if for every
  $\mathcal{P}$-monoidal \icat{} $\mathcal{V}^{\otimes}$ the functor
  \[ f^{*} \colon \Alg_{\mathcal{P}}(\mathcal{V}) \to
    \Alg_{\mathcal{O}/\mathcal{P}}(\mathcal{V}), \]
  given by composition with $f$, is an equivalence.
\end{defn}
\begin{remark}
  If $f$ is a Morita equivalence, then as a special case (taking
  $\mathcal{V}^{\otimes}$ to be $\CatI^{\times}$) we have that pullback
  along $f$ gives an equivalence between $\mathcal{P}$-monoidal and
  $\mathcal{O}$-monoidal \icats{}.
\end{remark}

Our discussion of free algebras leads to a checkable criterion for a
morphism of extendable cartesian patterns to be a Morita equivalence:
\begin{propn}\label{propn:Morita}
  Suppose $\mathcal{O}$ and $\mathcal{P}$ are extendable cartesian
  patterns and $f \colon \mathcal{O} \to \mathcal{P}$ is a morphism of
  cartesian patterns such that
  \begin{enumerate}[(1)]
  \item $f^{\el} \colon \mathcal{O}^{\el} \to \mathcal{P}^{\el}$ is an
    equivalence of \igpds,
  \item for every $E \in \mathcal{O}^{\el}$ the functor
    \[ \Act_{\mathcal{O}}(E) \to \Act_{\mathcal{P}}(f(E)), \]
    induced by $f$,
    is an equivalence of \igpds{}.
  \end{enumerate}
  Then $f$ is a Morita equivalence.
\end{propn}
\begin{proof}
  We first consider the case where $\mathcal{V}^{\otimes}$ is
  presentably $\mathcal{P}$-monoidal. Then we have a commutative
  diagram
  \[
    \begin{tikzcd}
      \Alg_{\mathcal{P}}(\mathcal{V}) \arrow{r}{f^{*}} \arrow{d}{U_{\mathcal{P}}} &
      \Alg_{\mathcal{O}/\mathcal{P}}(\mathcal{V}) \arrow{d}{U_{\mathcal{O}}} \\
      \Fun_{/\mathcal{P}^{\el}}(\mathcal{P}^{\el}, \mathcal{V})
      \arrow{r}{f^{\el,*}} & \Fun_{/\mathcal{O}^{\el}}(\mathcal{O}^{\el}, f^{*}\mathcal{V})
    \end{tikzcd}
  \]
  where the vertical maps are monadic right adjoints by
  \cref{cor:FOmonadic} and the bottom
  horizontal map is an equivalence by assumption (1). Using
  \cite[Corollary 4.7.3.16]{ha} we see that $f^{*}$ is an equivalence
  \IFF{} for every $\xi \in
  \Fun_{/\mathcal{P}^{\el}}(\mathcal{P}^{\el}, \mathcal{V})$ the
  natural map
  \[ F_{\mathcal{O}}f^{\el,*}\xi \to f^{*}F_{\mathcal{P}}\xi \]
  is an equivalence. Since $U_{\mathcal{P}}$ detects equivalences, it
  suffices to show that the induced map
  \[ (F_{\mathcal{O}}f^{\el,*}\xi)(E) \to  (F_{\mathcal{P}}\xi)(f(E)) \]
  is an equivalence for all $E \in \mathcal{P}^{\el}$. Since
  $\mathcal{O}$ and $\mathcal{P}$ are extendable we have colimit
  formulas for $F_{\mathcal{O}}$ and $F_{\mathcal{P}}$, which identify
  this map with the map
  \[ \colim_{\alpha \colon O \to E \in \Act_{\mathcal{O}}(E)}
    \alpha_{!}\xi(f(O)) \to \colim_{\beta \colon P \to f(E) \in
      \Act_{\mathcal{P}}(f(E))} \beta_{!}\xi(P), \]
  which is the natural map of colimits arising from the morphism
  $\Act_{\mathcal{O}}(E) \to
  \Act_{\mathcal{P}}(f(E))$ induced by $f$. This is an equivalence by
  assumption (2). Thus $f^{*}$ is an equivalence for every presentably
  $\mathcal{P}$-monoidal \icat{}. Since we can embed any
  $\mathcal{P}$-monoidal \icat{} fully faithfully in a presentably
  $\mathcal{P}$-monoidal one (possibly after passing to a larger
  universe) this completes the proof.
\end{proof}

\begin{remark}
  If $f \colon \mathcal{O} \to \mathcal{P}$ is a morphism between
  extendable cartesian patterns such that $f^{\el} \colon
  \mathcal{O}^{\el} \to \mathcal{P}^{\el}$ is an equivalence, then
  this proof in the case of $\mathcal{S}$ shows that condition (2) is
  also \emph{necessary} for $f$ to be a Morita equivalence.
\end{remark}

In some cases this criterion can be used to identify the
\iopd{} corresponding to a cartesian pattern without using the highly
technical machinery of \emph{approximations} from \cite[\S 2.3.3]{ha},
as we will now demonstrate in some simple examples:
\begin{example}[Associative algebras]
  Let $\Ass \to \xF_{*}$ denote the (symmetric) associative
  \iopd{}. As in \cite[Remark 4.1.1.4]{ha} we can think of this as a
  category whose objects are the pointed finite sets
  $\angled{n} \in \xF_{*}$, with a morphism
  $\angled{n} \to \angled{m}$ given by a morphism
  $\phi \colon \angled{n} \to \angled{m}$ in $\xF_{*}$ together with
  linear orderings $\leq_{i}$ of the preimages $\phi^{-1}(i)$,
  $1 \leq i \leq m$. The composite of
  $(\phi, \leq_{i}) \colon \angled{n} \to \angled{m}$ and
  $(\psi, \leq'_{j}) \colon \angled{m} \to \angled{k}$ is given by the
  composite $\psi \phi$ in $\xF_{*}$ with the ordering $\leq''_{t}$ of
  $(\psi \phi)^{-1}(t)$ given by
  \[ i \leq''_{t} i' \iff \phi(i) \leq'_{t} \phi(i') \text{ and } i
    \leq_{s} i' \text{ if } s = \phi(i) =\phi(i').\]
  There is a functor $\mathrm{cut} \colon \Dop \to \Ass$ that
  takes $[n] \in \Dop$ to $\angled{n}$ and a morphism $\phi \colon [m]
  \to [n]$ in $\simp$ to the morphism $\angled{n} \to \angled{m}$
  given by
  \[ i \mapsto
    \begin{cases}
      j,  & \phi(j-1) < i \leq \phi(j), \\
      0, & \text{if no such $j$ exists}
    \end{cases}
  \]
  with the linear ordering of $\mathrm{cut}(\phi)^{-1}(j)$ that given
  by identifying this with $\{i : \phi(j-1) < i \leq \phi(j) \}$. It
  is easy to see that $\mathrm{cut} \colon \Dop \to \Ass$ is a
  morphism of cartesian patterns, and we claim that it is a Morita
  equivalence: Both patterns are extendable by
  \cref{ex:symopdext,ex:Dopext}, with
  $\simp^{\op,\el} \simeq \Ass^{\el} \simeq *$. Moreover,
  $\Act_{\Dop}([1])$ is the discrete set
  $\{[n] \to [1] \colon n = 0,1,\ldots\}$ while
  $\Act_{\Ass}(\angled{1})$ can be identified with the disjoint union
  over $n$ of the contractible groupoid of linear orderings of
  $\{1,\ldots,n\}$. The conditions of \cref{propn:Morita} therefore
  hold, and so we get for any ($\Ass$-)monoidal \icat{}
  $\mathcal{V}^{\otimes}$ an equivalence
  \[ \Alg_{\Dop/\Ass}(\mathcal{V}) \isoto
    \Alg_{\Ass}(\mathcal{V}).\]
\end{example}

\begin{example}[Bimodules]
  Let $\simp_{/[1]}^{\op,\flat}$ denote the category $\Dop_{/[1]} :=
  (\simp_{/[1]})^{\op}$ with the inert/active factorization system
  lifted from $\Dop$ (along the left fibration $\Dop_{/[1]} \to \Dop$)
  and the three maps $[1] \to [1]$ as elementary objects. We can think
  of a morphism $[n] \to [1]$ as a sequence $(i_{0},\ldots,i_{n})$
  with $0 \leq i_{0} \leq \cdots \leq i_{n} \leq 1$; then the
  elementary objects are $(0,0)$, $(0,1)$, and $(1,1)$. We also let
  $\Bimod \to \xF_{*}$ denote the (symmetric) bimodule operad (whose
  algebras are given by a pair of associative algebras and a bimodule
  between them). This can be described (cf.~\cite[Notation
  4.3.1.5]{ha}) as a category where
  \begin{itemize}
  \item objects are lists $(\angled{n},
  (a_{1},b_{1}),\ldots,(a_{n},b_{n}))$ where $0 \leq a_{i} \leq b_{i}
  \leq 1$,
\item  a morphism $(\angled{n},
  (a_{1},b_{1}),\ldots,(a_{n},b_{n})) \to (\angled{m},
  (a'_{1},b'_{1}),\ldots,(a'_{m},b'_{m}))$ is given by a morphism
  $(\phi, \leq_{i}) \colon \angled{n} \to \angled{m}$ in $\Ass$ such
  that for $j = 1,\ldots,m$, if $\phi^{-1}(j) = \{i_{1}<_{j} i_{2}
  <_{j} \cdots <_{j} i_{k})$, then
  \[ a'_{j} = a_{i_{1}} \leq b_{i_{1}} = a_{i_{2}} \leq \cdots \leq
    a_{i_{k}} \leq b_{i_{k}} = b'_{j}.\]
  \end{itemize}
  We then define a functor $\simp_{/[1]}^{\op} \xto{\mathrm{cut}'}
  \Bimod$ by
  \begin{itemize}
  \item $\mathrm{cut}'(i_{0},\ldots,i_{n}) = (\angled{n},
    (i_{0},i_{1}), \ldots, (i_{n-1},i_{n})),$
  \item for a morphism $\phi \colon (i_{0},\ldots,i_{n}) \to
    (j_{0},\ldots,j_{m})$, which is given by $\phi \colon [m] \to [n]$ in
    $\simp$ such that $j_{s} = i_{\phi(s)}$, we set
    $\mathrm{cut}'(\phi) = \mathrm{cut}(\phi)$, which satisfies the
    required condition.
  \end{itemize}
  The functor $\mathrm{cut}'$ then fits in a commutative square
  \[
    \begin{tikzcd}
      \simp_{/[1]}^{\op} \arrow{r}{\mathrm{cut}'} \arrow{d} &  \Bimod
      \arrow{d} \\
      \Dop \arrow{r}{\mathrm{cut}} & \Ass,
    \end{tikzcd}
  \]
  and is a morphism of cartesian patterns (since the factorization
  systems are lifted from those on $\Dop$ and $\Ass$).
  The pattern $\simp_{/[1]}^{\op,\flat}$ is extendable, \eg{} by the
  non-symmetric analogue of \cref{ex:gensymopdext}, while $\Bimod$ is
  extendable by \cref{ex:symopdext}. We have
  $\simp_{/[1]}^{\op,\el} = \{(0,0),(0,1),(1,1)\}$ which is isomorphic
  to
  $\Bimod^{\el} = \{(\angled{1}, (0,0)), (\angled{1}, (0,1)),
  (\angled{1}, (1,1))$. The \igpds{} $\Act_{\Dop_{/[1]}}((0,0))$ and
  $\Act_{\Dop_{/[1]}}((1,1))$ are discrete sets isomorphic to
  $\mathbb{N}$ (with $n$ corresponding to the unique active morphisms
  $(0,\ldots,0) \actto (0,0)$ and $(1,\ldots,1) \actto (1,1)$, respectively,
  with the source lying over $[n]$), while $\Act_{\Dop_{/[1]}}((0,1))$
  is isomorphic to $\mathbb{N} \times \mathbb{N}$ (with $(n,m)$
  corresponding to the unique active map
  $(0,\ldots,0,1,\ldots,1) \to (0,1)$ with $(n+1)$ 0's and $(m+1)$
  1's). On the other hand, the \igpd{}
  $\Act_{\Bimod}((\angled{1}, (i,j))$ we can describe as a coproduct of contractible
  groupoids indexed by the set $\Act_{\Dop_{/[1]}}(i,j))$. Hence
  \cref{propn:Morita} applies, and so we get for any $\Bimod$-monoidal \icat{}
  $\mathcal{V}^{\otimes}$ an equivalence
  \[ \Alg_{\Dop_{/[1]}/\Bimod}(\mathcal{V}) \isoto
    \Alg_{\Bimod}(\mathcal{V}).\]
\end{example}

\begin{example}[Modules over commutative algebras]
  Let $\xF_{*,\angled{1}/}^{\flat}$ denote the slice category
  $\xF_{*,\angled{1}/}$ with the inert/active factorization system
  lifted from $\xF_{*}$ (along the left fibration $\xF_{*,\angled{1}/}
  \to \xF_{*}$) with the two maps $\angled{1} \to \angled{1}$ as
  elementary objects; then $\xF_{*,\angled{1}/}^{\flat}$ is a
  cartesian pattern. We can also think of the objects as pairs
  $(\angled{n}, i)$ with $i \in \angled{n}$, with a morphism
  $(\angled{n},i) \to (\angled{m},j)$ given by a morphism $\phi \colon
  \angled{n} \to \angled{m}$ in $\xF_{*}$ such that $\phi(i) = j$.
  The cartesian pattern $\xF_{*,\angled{1}/}^{\flat}$ is extendable:
  The \igpd{} $\Act_{\xF_{*,\angled{1}/}}(\angled{n},i)$ we can
  describe as the groupoid of pairs $(\phi \colon \angled{m} \to
  \angled{n}, j \in \phi^{-1}(i))$ with $\phi$ active, and so in the
  commutative square
  \[
    \begin{tikzcd}
      \Act_{\xF_{*,\angled{1}/}}(\angled{n},i) \arrow{r} \arrow{d} &
      \prod_{j=1}^{n} \Act_{\xF_{*,\angled{1}/}}(\angled{1},\rho_{j}(i))\arrow{d} \\
      \Act_{\xF_{*}}(\angled{n}) \arrow{r}{\sim} & \prod_{j=1}^{n} \Act_{\xF_{*}}(\angled{1}),
    \end{tikzcd}
  \]
  the map on fibres over each active map $\phi \in
  \Act_{\xF_{*}}(\angled{n})$ is an isomorphism; this is therefore a
  pullback square, so the top horizontal morphism is an equivalence,
  as required.

  We let $\CMod \to \xF_{*}$ denote the (symmetric) \iopd{} whose
  algebras are a commutative algebra together with a module over
  it. This can be described as a category with
  \begin{itemize}
  \item objects lists $(\angled{n}, i_{1},\ldots,i_{n})$ with $i_{s}
    \in \{0,1\}$ (with
    $(\angled{1}, 0)$ representing the algebra and $(\angled{1},1)$
    the module,
  \item a morphism $(\angled{n}, i_{1},\ldots,i_{n}) \to (\angled{m},
    j_{1},\ldots,j_{m})$ is given by a morphism $\phi \colon
    \angled{n} \to \angled{m}$ in $\xF_{*}$ such that for all $s =
    1,\ldots,m$, we have
    \[ \sum_{t \in \phi^{-1}(s)} i_{t} = j_{s}.\]
  \end{itemize}
  We can define a functor $\mu \colon \xF_{*,\angled{1}} \to \CMod$
  given on objects by
  \[\mu(\angled{n}, i) = (\angled{n},
    \delta_{1i},\ldots,\delta_{ni}),\]
  where $\delta_{ji} = 1$ if $j=i$, and $0$ otherwise.
  Given a morphism $(\angled{n},i) \to (\angled{m},j)$ over
  $\phi \colon \angled{n} \to \angled{m}$, we assign to it the
  morphism $\mu(\angled{n},i) \to \mu(\angled{m},j)$ over
  $\phi$, which indeed exists. It is clear that $\mu$ is a morphism of
  cartesian patterns, and identifies $\xF_{*,\angled{1}/}$ with the
  full subcategory of $\CMod$ spanned by the objects with a most one
  copy of $1$. We claim that $\mu$ is a Morita equivalence:
  $\xF_{*,\angled{1}/}^{\el}$ and $\CMod^{\el}$ are both the
  2-element set containing $(\angled{1}, i)$ ($i = 0,1$), and all
  active morphisms to $(\angled{1},i)$ in $\CMod$ are in the image of
  $\mu$, so that \[\Act_{\xF_{*,\angled{1}/}}(\angled{1},i) \simeq\Act_{\CMod}(\angled{1},i).\]
  Since $\CMod$ is extendable by \cref{ex:symopdext}, we can apply
  \cref{propn:Morita} to get for any $\CMod$-monoidal \icat{}
  $\mathcal{V}^{\otimes}$ an equivalence
  \[ \Alg_{\xF_{*,\angled{1}/}/\CMod}(\mathcal{V}) \isoto
    \Alg_{\CMod}(\mathcal{V}).\]
\end{example}

\bibliographystyle{hamsalpha}

\begin{thebibliography}{GHK17}

\bibitem[Bar18]{bar}
Clark Barwick, \href {https://doi.org/10.2140/gt.2018.22.1893} {\emph{From
  operator categories to topological operads}}, Geom. Topol. \textbf{22}
  (2018), no.~4, 1893--1959.

\bibitem[CH20]{ChuHaugseng}
Hongyi Chu and Rune Haugseng, \href {https://doi.org/10.1016/j.aim.2019.106913}
  {\emph{Enriched {$\infty$}-operads}}, Adv. Math. \textbf{361} (2020), 106913.
  \MR{4038556}

\bibitem[CH21]{patterns}
\bysame, \href {https://doi.org/10.1016/j.aim.2021.107733}
  {\emph{Homotopy-coherent algebra via {S}egal conditions}}, Adv. Math.
  \textbf{385} (2021), 107733. \MR{4256131}

\bibitem[CHH18]{ChuHaugsengHeuts}
Hongyi Chu, Rune Haugseng, and Gijs Heuts, \href
  {https://doi.org/10.1112/topo.12071} {\emph{Two models for the homotopy
  theory of $\infty$-operads}}, J. Topol. \textbf{11} (2018), no.~4, 857--873.

\bibitem[GH15]{enriched}
David Gepner and Rune Haugseng, \href
  {https://doi.org/10.1016/j.aim.2015.02.007} {\emph{Enriched
  {$\infty$}-categories via non-symmetric {$\infty$}-operads}}, Adv. Math.
  \textbf{279} (2015), 575--716. \MR{3345192}

\bibitem[GHK17]{AnalMnd}
David Gepner, Rune Haugseng, and Joachim Kock, \href
  {http://arxiv.org/abs/arXiv:1712.06469} {\emph{$\infty$-operads as analytic
  monads}}, 2017.

\bibitem[GHN17]{freepres}
David Gepner, Rune Haugseng, and Thomas Nikolaus, \href
  {http://arxiv.org/abs/arXiv:1501.02161} {\emph{Lax colimits and free
  fibrations in {$\infty$}-categories}}, Doc. Math. \textbf{22} (2017),
  1225--1266. \MR{3690268}

\bibitem[Gla16]{GlasmanDay}
Saul Glasman, \href {https://doi.org/10.4310/MRL.2016.v23.n5.a6} {\emph{Day
  convolution for {$\infty$}-categories}}, Math. Res. Lett. \textbf{23} (2016),
  no.~5, 1369--1385. \MR{3601070}

\bibitem[Hei18]{HeineThesis}
Hadrian Heine, \href
  {https://repositorium.ub.uni-osnabrueck.de/handle/urn:nbn:de:gbv:700-201909201996}
  {\emph{Restricted $\mathrm{L}_{\infty}$-algebras}}, Ph.D. thesis, University
  of Osnabrück, 2018.

\bibitem[Hin16]{hinloc}
Vladimir Hinich, \href {https://doi.org/10.4310/HHA.2016.v18.n1.a3}
  {\emph{Dwyer-{K}an localization revisited}}, Homology Homotopy Appl.
  \textbf{18} (2016), no.~1, 27--48. \MR{3460765}

\bibitem[HMS19]{cois}
Rune Haugseng, Valerio Melani, and Pavel Safronov, \href
  {http://arxiv.org/abs/arXiv:1904.11312} {\emph{Shifted coisotropic
  correspondences}}, 2019.

\bibitem[HR18]{PropLect}
Philip Hackney and Marcy Robertson, \emph{Lecture notes on infinity-properads},
  2016 {MATRIX} annals, MATRIX Book Ser., vol.~1, Springer, Cham, 2018,
  pp.~351--374. \MR{3792530}

\bibitem[HRY15]{HRYProperad}
Philip Hackney, Marcy Robertson, and Donald Yau, \href
  {https://doi.org/10.1007/978-3-319-20547-2} {\emph{Infinity properads and
  infinity wheeled properads}}, Lecture Notes in Mathematics, vol. 2147,
  Springer, Cham, 2015. \MR{3408444}

\bibitem[Koc11]{Kock}
Joachim Kock, \href {https://doi.org/10.1093/imrn/rnq068} {\emph{Polynomial
  functors and trees}}, Int. Math. Res. Not. IMRN (2011), no.~3, 609--673.
  \MR{2764874}

\bibitem[Koc16]{Kock_Properads}
\bysame, \href {https://doi.org/10.1007/s13348-015-0160-0} {\emph{Graphs,
  hypergraphs, and properads}}, Collect. Math. \textbf{67} (2016), no.~2,
  155--190. \MR{3484016}

\bibitem[Lur09]{ht}
Jacob Lurie, \href {https://doi.org/10.1515/9781400830558} {\emph{Higher topos
  theory}}, Annals of Mathematics Studies, vol. 170, Princeton University
  Press, Princeton, NJ, 2009. \MR{2522659}

\bibitem[Lur17]{ha}
\bysame, \emph{Higher algebra}, 2017, available from
  \url{http://www.math.ias.edu/~lurie/papers/HA.pdf}.

\bibitem[MW07]{MoerdijkWeiss}
Ieke Moerdijk and Ittay Weiss, \href {https://doi.org/10.2140/agt.2007.7.1441}
  {\emph{Dendroidal sets}}, Algebr. Geom. Topol. \textbf{7} (2007), 1441--1470.
  \MR{2366165}

\bibitem[Seg74]{SegalCatCohlgy}
Graeme Segal, \href {https://doi.org/10.1016/0040-9383(74)90022-6}
  {\emph{Categories and cohomology theories}}, Topology \textbf{13} (1974),
  293--312. \MR{0353298}

\end{thebibliography}

\providecommand{\bysame}{\leavevmode\hbox to3em{\hrulefill}\thinspace}
\providecommand{\MR}{\relax\ifhmode\unskip\space\fi MR }
\providecommand{\MRhref}[2]{%
  \href{http://www.ams.org/mathscinet-getitem?mr=#1}{#2}
}
\providecommand{\href}[2]{#2}

\end{document}